\documentclass[11pt,oneside,english]{amsart}
\RequirePackage{amsthm,amsmath,mathtools}

\usepackage[T1]{fontenc}
\usepackage{geometry}
\usepackage{babel,verbatim}
\usepackage{float}
\usepackage{amstext}
\usepackage{amssymb}
\usepackage{amsmath,bbm}
\usepackage{graphicx}
\usepackage[numbers]{natbib}
\RequirePackage[colorlinks,citecolor=blue,urlcolor=blue]{hyperref}

\PassOptionsToPackage{subsection=false,section=false}{beamerouterthememiniframes}
\usepackage{tikz}
\usetikzlibrary{shapes.multipart,shapes,arrows,decorations.pathreplacing}
\usepackage{algorithm,algpseudocode,multirow}
\usepackage{pgfplots}
\usepackage{caption,subcaption}

\renewcommand{\O}{\mathcal{O}}

\newcommand{\D}{\mathcal{D}}
\newcommand{\mS}{\mathcal{S}}
\newcommand{\E}{\mathbb{E}}

\newcommand{\R}{\mathbb{R}}
\newcommand{\N}{\mathbb{N}}
\renewcommand{\P}{\mathbb{P}}
\newcommand{\mZ}{\mathcal{Z}}
\newcommand{\mL}{\mathcal{L}}
\renewcommand{\mS}{\mathcal{S}}
\newcommand{\ov}[1]{\overline{#1}}
\newcommand{\ti}[1]{\tilde{#1}}
\newcommand{\un}[1]{\underline{#1}}
\newcommand{\nf}[1]{{\normalfont#1}}

\newcommand{\1}{\mathbbm{1}}

\newcommand{\ua}{\uparrow}
\newcommand{\da}{\downarrow}
\newcommand{\la}{\leftarrow}
\newcommand{\ra}{\rightarrow}

\newcommand{\me}{\text{\normalfont me}}
\newcommand{\ex}{\text{\normalfont ex}}
\newcommand{\cadlag}{c\`adl\`ag}
\newcommand{\Beta}{\mathrm{Beta}}

\numberwithin{equation}{section}
\numberwithin{figure}{section}
\theoremstyle{plain}

\theoremstyle{plain}
\newtheorem{thm}{\protect\theoremname}
\theoremstyle{remark}
\newtheorem*{rem*}{\protect\remarkname}
\theoremstyle{remark}
\newtheorem{rem}[]{\protect\remarkname}
\theoremstyle{plain}
\newtheorem*{assumption*}{\protect\assumptionname}
\theoremstyle{plain}
\newtheorem{lem}[thm]{\protect\lemmaname}
\theoremstyle{plain}
\newtheorem{cor}[thm]{\protect\corollaryname}
\theoremstyle{plain}
\newtheorem{prop}[thm]{\protect\propositionname}
\theoremstyle{definition}
\newtheorem*{example*}{\protect\examplename}

\newcommand{\BrownianMotion}[7]{% initial points (2), points, advance, rand factor, options, end label
	\draw[#6] (#1,#2)
	\foreach \x in {1,...,#3}
	{-- ++ (#4,rand*#5)}
	node[below right] {#7};
}

\makeatother
  \providecommand{\algorithmname}{Algorithm}
  \providecommand{\assumptionname}{Assumption}
  \providecommand{\examplename}{Example}
  \providecommand{\lemmaname}{Lemma}
  \providecommand{\propositionname}{Proposition}
  \providecommand{\remarkname}{Remark}
\providecommand{\corollaryname}{Corollary}
\providecommand{\theoremname}{Theorem}

\begin{document}

\title[$\varepsilon$-strong simulation of convex minorants]{$\varepsilon$-strong simulation of the convex minorants of stable processes and meanders}


\author{Jorge Gonz\'alez C\'azares, Aleksandar Mijatovi\'c, \and
	Ger\'onimo Uribe Bravo}

\address{Department of Statistics, University of Warwick, \& The Alan Turing Institute, UK}

\email{jorge.gonzalez-cazares@warwick.ac.uk}

\address{Department of Statistics, University of Warwick, \& The Alan Turing Institute, UK}

\email{a.mijatovic@warwick.ac.uk}

\address{Universidad Nacional Aut\'onoma de M\'exico, M\'exico}

\email{geronimo@matem.unam.mx}

\begin{abstract}
Using marked Dirichlet processes we characterise the law of the convex minorant of 
the meander for a certain class of L\'evy processes, which includes subordinated 
stable and symmetric L\'evy processes. We apply this characterisaiton to construct 
$\varepsilon$-strong simulation ($\varepsilon$SS) algorithms for the convex 
minorant of stable meanders, the finite dimensional distributions of stable meanders 
and the convex minorants of weakly stable processes. We prove that the running times 
of our $\varepsilon$SS algorithms have finite exponential moments. We implement the 
algorithms in Julia 1.0 (available on GitHub) and present numerical examples 
supporting our convergence results.
\end{abstract}

\maketitle

\section{Introduction}

\subsection{Setting and motivation}
% (I)
% 1. Stable processes are improtant in app prob
% 2. stable meanders play a key role 
% 3. simulation of stable meander is difficult 
% 4. Q: how to sample? epsilon strong simulation 
% (II) Understand the convex minorants of stable meanders 

The universality of stable laws, processes and their path transformations makes them 
ubiquitous in probability theory and many areas of statistics and natural and social 
sciences (see e.g.~\citep{MR1745764,MR3160562} and the references therein). Brownian 
meanders, for instance, have been used in applications, ranging from stochastic 
partial differential equations~\citep{MR2019968} to the pricing of 
derivatives~\citep{MR3404143} and unbiased and exact simulation of the solutions of 
stochastic differential equations~\citep{MR2998701,MR3092549}. Analytic information 
is generally hard to obtain for either the maximum~\citep{MR3098676} and its temporal 
location~\citep[p.~2]{MR3925667} or for the related path 
transformations~\citep{MR2164035}, even in the case of the Brownian motion with 
drift~\citep{Iafrate2019}. Moreover, except in the case of Brownian motion with 
drift~\citep{MR1912205,MR2730908}, exact simulation of path-functionals of weakly 
stable processes is rarely available. In particular, exact simulation of functionals 
of stable meanders, which arise in numerous path transformations 
(see~\citep[Sec.~VIII]{MR1406564} and~\citep{MR3160578}), appears currently to be 
out of reach, as even the maximum of a stable processes can only be simulated as a 
univariate random variable in the strictly stable case~\citep{ExactSim}. A natural 
question (Q1) arises: \emph{does there exist a simulation algorithm with almost sure 
control of the error for stable meanders and path-functionals related to the extrema 
of weakly stable processes?} 

A complete description of the law of the convex minorant of L\'evy processes is given 
in~\citep{MR2978134}. Its relevance in the theory of simulation was highlighted in 
recent contributions~\citep{ExactSim,LevySupSim}, which developed sampling 
algorithms for certain path-functionals related to the extrema of L\'evy processes 
(see also Subsection~\ref{subsec:sim-conv-min} below). Thus, as L\'evy meanders 
arise in numerous path transformations and functionals of L\'evy processes~\citep{
	MR0436354,MR1465814,MR1836746,MR3160578,MR3531705,Iafrate2019}, 
it is natural to investigate their simulation problem via their convex minorants, 
leading to question (Q2): \emph{does there exist a tractable characterisation of the law 
of convex minorants of L\'evy meanders given in terms of the marginals of 
the corresponding L\'evy process?} 
%analogous to the one in~\citep{MR2978134} 
This question is not trivial for the following two reasons. (I) A description of the convex 
minorant of a L\'evy meander is known only for a Brownian meander~\cite{MR2948693} 
and is given in terms of the marginals of the meander, not the marginals of the Brownian 
motion (cf. Subsection~\ref{subsubsec:convex_minorant} below). 
(II) Tractable descriptions of the convex minorant of a process $X$ typically rely on 
the exchangeability of the increments of $X$ in a fundamental 
way~\citep{MR2978134,2019arXiv190304745A}, a property clearly not satisfied when 
$X$ is a L\'evy meander. 

%More specifically,~\citep{LevySupSim} introduces a general and geometrically 
%convergent algorithm to approximately simulate the state, supremum and time of 
%the supremum of a L\'evy process. Paper~\citep{ExactSim} constructs a perpetuity 
%and exploits the Markovian structure in the stable case to obtain an \emph{exact} 
%simulation algorithm for the supremum alone. 

\subsection{Contributions}
In this paper we answer affirmatively both questions (Q1) \& (Q2) stated above. 
%(under certain assumptions on the L\'evy process in (Q2)).
More precisely, in Theorem~\ref{thm:meander-minorant} below, we establish a 
characterisation of the law of the convex minorant of L\'evy meanders, based on 
marked Dirichlet processes, for L\'evy processes with constant probability of being 
positive (e.g. subordinated stable and symmetric L\'evy processes). In particular, 
Theorem~\ref{thm:meander-minorant} gives an alternative description of the law of 
the convex minorant of a Brownian meander to the one in~\cite{MR2948693}. Our 
description holds for the meanders of the aforementioned class of L\'evy processes, 
while~\cite{MR2948693} is valid for Brownian motion only 
(see Subsection~\ref{subsubsec:convex_minorant} below for more details). 

% stress no exchangeable increments, 2019arXiv190304745A
% non-tractability of increments (no known density, etc.)
% this is a new characterisation the Brownian meander case 
% cite Pitman and Ross! MR2948693
The description in Theorem~\ref{thm:meander-minorant} yields a Markovian structure 
(see Theorem~\ref{thm:CMSM} below) used to construct $\varepsilon$-strong 
simulation ($\varepsilon$SS) algorithms for the convex minorants of weakly stable 
processes and stable meanders, as well as for the finite-dimensional distributions of 
stable meanders. We apply our algorithms to the following problems: exact simulation of 
barrier crossing events; unbiased simulation of certain path-functionals of stable 
processes %and other related processes 
such as the moments of the crossing times of weakly stable processes; estimation of 
the moments of the normalised stable excursion. We report on the numerical 
performance in Section~\ref{sec:e-app_and_num}. 
Finally, we establish Theorem~\ref{thm:all_complexities} below stating that the 
running times of all of these algorithms have exponential moments, a property not 
seen before in the context of $\varepsilon$SS (cf. discussion in 
Subsection~\ref{subsec:lit}). Moreover, to the best of our knowledge, our results 
constitute the first simulation algorithms for stable meanders to appear in the 
literature. Due to the analytical intractability of their law, no simulation 
algorithms have been proposed so far. 

\subsection{Connections with the literature\label{subsec:lit}}
Our results are linked to seemingly disparate areas in pure and applied probability.
We discuss connections to each of the areas separately.

\subsubsection{Convex minorants of L\'evy meanders}
\label{subsubsec:convex_minorant}
The convex minorant of the path of a process on a fixed time interval
is the (pointwise) largest convex function dominated by the path. Typically,
the convex minorant is a piecewise linear function with a countably infinite number of 
linear segments known as \textit{faces},
see Subsection~\ref{subsec:approx_linear} for definition of such functions. 
Note that the chronological ordering of its faces  coincides
with the ordering by increasing slope.

A description of the convex minorant of a Brownian meander is given 
in~\cite{MR2948693}. To the best of our knowledge, the convex minorant of no other 
L\'evy meander has been characterised prior to the results presented below. The 
description in~\cite{MR2948693} of the faces of the convex minorant of a Brownian 
meander depends in a fundamental way 
%fundamentally on the scaling property of Brownian motion and the 
on the analytical tractability of the density of the marginal of a Brownian meander 
at the final time, a quantity not available for other L\'evy processes. 
%and the availability 
%of certain explicit densities. 
Furthermore, \cite{MR2948693} describes the faces of the convex minorant of Brownian 
meanders in chronological order, a strategy feasible in the Brownian case because 
the right end of the interval is the only accumulation point for the faces, but 
infeasible in general. For example, the convex minorant of a Cauchy meander has 
infinitely many faces in any neighborhood of the origin since the set of its slopes 
is a.s. dense in $\R$. Hence, if a generalisation of the description 
in~\cite{MR2948693} to other L\'evy meanders existed, it could work only if the sole 
accumulation point is the right end of the interval. Moreover, the \emph{scaling} and 
\emph{time inversion} properties of Brownian motion, not exhibited by other L\'evy 
processes~\cite{MR2136870,MR3912200}, are central in the description 
of~\cite{MR2948693}.

In contrast, the description in Theorem~\ref{thm:meander-minorant} holds for the L\'evy 
processes with constant probability of being positive, including Brownian motion, and 
does not require any explicit knowledge of the transition probabilities of the L\'evy 
meander. Moreover, to the best of our knowledge, ours is the only characterisation of 
the faces of the convex minorant in size-biased order where the underlying process 
does \emph{not} possess exchangeable increments (a key property in all such 
descriptions~\cite{MR2825583,MR2978134,2019arXiv190304745A}).

\subsubsection{$\varepsilon$-strong simulation algorithms}
%Paragraph on epsilon-strong simulation
$\varepsilon$SS is a procedure that generates a random element 
whose distance to the target random element is at most $\varepsilon$ almost surely. 
The tolerance level $\varepsilon>0$ is given \emph{a priori} and can be refined 
(see Section~\ref{sec:algorithms} for details). The notion of $\varepsilon$SS was 
introduced in~\cite{MR2995793} in the context of the simulation of a Brownian path 
on a finite interval. This framework was extended to the reflected Brownian motion 
in~\cite{MR3404635}, jump diffusions in~\cite{MR3449801}, multivariate It\^o 
diffusions in~\cite{MR3619789}, max-stable random fields in~\cite{BlanchetMaxStable} 
and the fractional Brownian motion in~\cite{eSS_fBM}. In general, an $\varepsilon$SS 
algorithm is required to terminate almost surely, but might have infinite expected 
complexity as is the case in~\cite{MR2995793,MR3449801}. The termination times of 
the algorithms in~\cite{MR3404635,MR3619789,eSS_fBM} are shown to have finite 
means. In contrast,  the running times of the $\varepsilon$SS algorithms in the 
present paper have finite exponential moments (see 
Theorem~\ref{thm:all_complexities} below), making them efficient in applications 
(see Subsection~\ref{subsec:numerics} below).

%$\varepsilon$SS algorithms have been developed both as auxiliary procedures that lead 
%to exact simulation algorithms~\cite{MR3404635,MR3619789,BlanchetPerfectSim, 
%	MR3780387} and as approximate simulation algorithms in their own right when exact 
%simulation is not feasible~\cite{MR2995793}. These algorithms provide a powerful 
%control of the error that leads to a natural construction of confidence intervals (see 
%Subsection~\ref{subsec:e-app} below). Moreover, $\varepsilon$SS algorithms are 
%typically applicable to the exact simulation of ``localising'' functionals (see
%Subsection~\ref{subsec:indicators} below) and give rise to  unbiased sampling of other 
%functionals including those of first passage times (see Subsection~\ref{subsec:e-app} below). 
%$\varepsilon$SS has been applied successfully in rejection sampling and other simulation 
%settings based on localisations~\cite{MR2995793,MR3404635,MR3092549,
%	BlanchetPerfectSim,MR3780387,BlanchetMaxStable}.

In addition to the strong control on the error, $\varepsilon$SS algorithms have been 
used in the literature as auxiliary procedures yielding exact and unbiased simulation 
algorithms~\cite{MR2995793,MR3404635,MR3092549,BlanchetPerfectSim,
	MR3780387,BlanchetMaxStable}. We apply our $\varepsilon$SS algorithms to obtain 
exact samples of indicator functions of the form $\1_A(\Lambda)$ for certain random 
elements $\Lambda$ and suitable sets $A$ (see Subsection~\ref{subsec:indicators} 
below). The exact simulation of these indicators in turn yields unbiased samples of 
other functionals of $\Lambda$, including those of the (analytically intractable) first 
passage times of weakly stable processes (see Subsection~\ref{subsec:e-app} below).
%The simulation of such indicators has also been 
%applied to create rejection sampling algorithms and other simulation procedures 
%based on localisations~\cite{MR2995793,MR3404635,MR3092549,
%	BlanchetPerfectSim,MR3780387,BlanchetMaxStable}.

\subsubsection{Simulation algorithms based on convex minorants}
\label{subsec:sim-conv-min}
Papers~\citep{ExactSim,LevySupSim} developed simulation algorithms for the extrema 
of L\'evy processes in various settings.  We stress that algorithms and results 
in~\citep{ExactSim,LevySupSim} cannot be applied to the simulation of the 
path-functionals of L\'evy meanders considered in this paper. There are a number of 
reasons for this. First, the law of a L\'evy meander on a fixed time interval $[0,T]$ is 
given by the law of the original process $X$ \textit{conditioned} on $X$ being positive 
on $(0,T]$, an event of probability zero if, for instance, $X$ has infinite 
variation~\cite[Thm~47.1]{MR3185174}. Since the algorithms 
in~\citep{ExactSim,LevySupSim}  apply to the \textit{unconditioned} process, they are 
clearly of little direct use here. Second, the theoretical tools developed 
in~\citep{ExactSim,LevySupSim} are \textit{not} applicable to the problems considered 
in the present paper. Specifically, \cite{LevySupSim}  proposes a new simulation 
algorithm for the state $X_T$, the infimum and the time the infimum is attained on 
$[0,T]$  for a general (\textit{unconditioned}) L\'evy process $X$ and establishes the 
geometric decay of the error in $L^p$ of the corresponding  Monte Carlo algorithm. 
In contrast to the almost sure control of the simulation error for various 
path-functionals of $X$ \textit{conditioned} on $\{X_t>0:t\in(0,T]\}$ established in the 
present paper, the results in~\cite{LevySupSim} imply that the random error 
in~\cite{LevySupSim}, albeit very small in expectation, can take arbitrarily large values 
with positive probability. This makes the methods of~\cite{LevySupSim} completely 
unsuitable for the analysis of algorithms requiring an almost sure control of the error,
such as the ones in the present paper.

Paper~\cite{ExactSim} develops an exact simulation algorithm for the infimum of $X$
over the time interval $[0,T]$, where $X$ is an (unconditioned) strictly stable process. 
The scaling property of $X$ is crucial for the results in~\cite{ExactSim}. Thus, the results 
in~\cite{ExactSim} do \textit{not} apply to the simulation of the convex minorant (and 
cosequently the infimum) of the weakly stable processes considered here. 
This problem is solved in the present paper via a novel method based on tilting $X$ and 
then sampling the convex minorants of the corresponding meanders, see 
Subsection~\ref{subsec:e-SS_Levy} for details.
%An exact simulation algorithm for the univariate law $\mathcal{X}$ of the supremum of 
%a strictly stable process (over a fixed time interval) is constructed in~\cite{ExactSim}. 

The dominated-coupling-from-the-past (DCFTP) method in~\cite{ExactSim} 
is based on a perpetuity equation 
$\mathcal{X}\overset{d}{=}V(U^{1/\alpha}\mathcal{X}+(1-U)^{1/\alpha}S)$
established therein, where $\mathcal{X}$ denotes the law of the supremum of a 
strictly stable  process $X$. This perpetuity appears similar to the one in 
Theorem~\ref{thm:CMSM}(c) below, characterising the law of $X_1$ 
conditioned on $\{X_t>0:t\in(0,1]\}$. However, the analysis in~\cite{ExactSim} 
\textit{cannot} be applied to the perpetuity in  Theorem~\ref{thm:CMSM}(c) for the 
following reason: the ``nearly'' uniform factor $V$ in the perpetuity above
($U$ is uniform on $[0,1]$ and $S$ is as in Theorem~\ref{thm:CMSM}(c))
is used in~\cite{ExactSim} to modify it so that the resulting Markov chain exhibits 
coalescence with positive probability, a necessary feature for the DCFTP to work.
%The key difference is in the fact that the right-hand side of the perpetuity 
%in~\cite{ExactSim} is multiplied by the random variable $V$ that is ``nearly'' uniform 
%($U$ is uniform on $[0,1]$ and $S$ is as in Theorem~\ref{thm:CMSM}(c)),  
%enabling a modification in~\cite{ExactSim} to a perpetuity for $\mathcal{X}$ 
%that exhibit coalescence with positive probability, a necessary feature for the DCFTP 
%to work.
%The Markov 
%chain based on this perpetuity, as well as on the one in Theorem~\ref{thm:CMSM}(c), 
%does not exhibit coalescence, a necessary feature for the DCFTP to work. However, the 
%aw of the multiplying factor $V$ makes it possible to construct a new perpetuity for 
%$\mathcal{X}$ where coalescence occurs with positive probability, leading to the exact 
%simulation algorithm in~\cite{ExactSim}. 
Such a modification appears to be out of reach for the perpetuity in 
Theorem~\ref{thm:CMSM}(c) due to the absence of the multiplying factor, making exact 
simulation of stable meanders infeasible. However, even though the coefficients of the 
perpetuity in Theorem~\ref{thm:CMSM}(c) are dependent and have heavy tails, the 
Markovian structure for the error based on Theorem~\ref{thm:CMSM} allows us to 
define, in the present paper, a dominating process for the error whose return times to a 
neighbourhood of zero possess exponential moments. Since the dominating process 
can be simulated backwards in time, this leads to fast $\varepsilon$SS algorithms for the 
convex minorant of the stable meander and, consequently, of a weakly stable process.

%{Iafrate2019} % Iafrate: Drifted Brownian meander
%{MR3404143} % Yor: Meander options
%{MR2318402} % Svante: Graph enumeration Brownian Meander
%{MR2019968} % Bonaccorsi, Zambotti: SPDEs Brownian Meander
%{MR1465814} % Chaumont (Thm 3) : Stable excursions and meanders
%{MR3404635} % Blanchet (2015): steady-state reflected BM
%{MR3780387} % Blanchet, Murthy: Exact R^d Reflected BM based on eSS by Blanchet
%{BlanchetPerfectSim} % Blanchet, Zhang: Exact Ito Diffusion based on eSS by Blanchet
%{MR3092549} % Chen, Huang: Localization and exact simulation of BM-SDEs
%{MR2187299} % Beskos (2005): Exact simulation of diffusions
%{MR2274855} % Beskos, Papaspiliopoulos, Roberts : Retrospective exact SDE
%{MR2998701} % Chen, Huang: Girsanov + Meanders for unbiased sampling of SDE
%{MR2995793} % Beskos, Peluchetti and Roberts : e-SS of Brownian path (Apps: Sec. 7)
%{MR3619789} % Blanchet, Chen, Dong : e-SS of SDEs
%{MR3449801} % Pollock, Johansen, Roberts: exact and s-SS of SDEs for jump diffusions (infinite runtime: p. 30)
%{eSS_fBM} https://arxiv.org/pdf/1902.07824.pdf : e-SS fractional BM

\subsection{Organisation}
The remainder of the paper is structured as follows. In 
Section~\ref{sec:ConcaveMajorant} we 
state and prove Theorem~\ref{thm:meander-minorant}, which identifies the 
distribution of the convex minorant of a L\'evy meander in a certain class of L\'evy processes. In 
Section~\ref{sec:algorithms} we define $\varepsilon$SS and construct the main 
algorithms for the simulation from the laws of 
the convex minorants of both stable meanders and weakly stable processes, as well as from the
finite dimensional distributions of stable meanders. 
Numerical examples illustrating 
the methodology, its speed and stability are in Section~\ref{sec:e-app_and_num}. 
Section~\ref{sec:proofs} contains the analysis of the computational complexity (i.e. the proof of Theorem~\ref{thm:all_complexities}), the 
technical tools required in Section~\ref{sec:algorithms}
and the proof of Theorem~\ref{thm:CMSM} and its Corollary~\ref{cor:moments} 
(on the moments of stable meanders) used in Section~\ref{sec:e-app_and_num}.

\section{The law of the convex minorants of L\'evy meanders\label{sec:ConcaveMajorant}}

\subsection{Convex minorants and splitting at the minimum\label{subec:Con_Minor_Splitt}}

Let $X=(X_t)_{t\in[0,T]}$ be a L\'evy process on $[0,T]$, where $T>0$ is a 
fixed time horizon, started at zero $\P(X_0=0)=1$. If $X$ is a compound 
Poisson process with drift, exact simulation of the entire path of $X$ is 
typically available. We hence work with processes that are not compound
Poisson process with drift. By Doeblin's diffuseness 
lemma~\citep[Lem.~13.22]{MR1876169}, this assumption is equivalent to

\begin{assumption*}[D]
\label{asm:(D)} $\P(X_t=x)=0$ for all $x\in\R$
and for some (and then all) $t>0$.
\end{assumption*}

The convex minorant of a function $f:[a,b]\to\R$ is the pointwise
largest convex function $C(f):[a,b]\to\R$ such that $C(f)(t)\leq f(t)$ for all $t\in[a,b]$. 
Under~(\nameref{asm:(D)}), the convex minorant $C=C(X)$ of a path of $X$
turns out to be piecewise linear with infinitely many faces (i.e. linear segments). 
By convexity, sorting the faces by increasing slope coincides with their chronological 
ordering~\citep{MR2978134}. However, the ordering by increasing slopes is not 
helpful in determining the law of $C$. Instead, in the description of the law of 
$C$ in~\citep{MR2978134}, the faces are selected using size-biased sampling 
(see e.g.~\cite[Sec.~4.1]{LevySupSim}). 

\begin{figure}[H] % Requires \usepackage{subcaption}
	\centering{}\resizebox{.32\linewidth}{!}{
		\begin{subfigure}[l]{0.5\textwidth}
			\begin{tikzpicture}
			% Draw the Levy process
			\pgfmathsetseed{101101}
			\BrownianMotion{0}{0}{100}{0.005}{-0.2}{black}{};%1
			\BrownianMotion{101*.005}{-2.5}{123}{0.005}{-0.2}{black}{};%2 
			\BrownianMotion{225*.005}{-2.9}{54}{0.005}{-0.2}{black}{};%3
			\BrownianMotion{280*.005}{-2.4}{172}{0.005}{-0.2}{black}{};%4
			\BrownianMotion{453*.005}{-0.8}{27}{0.005}{-0.2}{black}{};%5
			\BrownianMotion{481*.005}{-2.0}{80}{0.005}{-0.2}{black}{};%6
			\BrownianMotion{562*.005}{-1.3}{69}{0.005}{-0.2}{black}{};%7
			\BrownianMotion{632*.005}{-4.6}{126}{0.005}{-0.2}{black}{};%8
			\BrownianMotion{759*.005}{-0.4}{208}{0.005}{-0.2}{black}{};%9
			\BrownianMotion{968*.005}{-0.8}{67}{0.005}{-0.2}{black}{};%10
			\BrownianMotion{1036*.005}{-1.8}{132}{0.005}{-0.2}{black}{};%11
			\BrownianMotion{1169*.005}{-0.1}{164}{0.005}{-0.2}{black}{};%12 ANOTHER EDGE
			\BrownianMotion{1324*.005}{-1.0}{45}{0.005}{-0.2}{black}{};%13
			
			%\draw[help lines] (0,-1.2) grid (1370*.005,5.5);
			
			% Draw the Concave Majorant
			\draw[red] (0,0) -- (6*.005,-.455) 
			-- (18*.005,-.95)
			-- (180*.005,-4.6)
			-- (199*.005,-4.96)
			-- (635*.005,-5.28)
			-- (668*.005,-5.31)
			-- (745*.005,-4.99)
			-- (1148*.005,-3.0)
			-- (1275*.005,-2.08)
			-- (1337*.005,-1.4)
			-- (1362*.005,-.89)
			-- (1369*.005,-.68)
			node[right] {\Large $C$};
			
			% Draw Uniform U
			\node[circle, fill=black, scale=0.5] at (933*.005,-188*3/403-215*5.03/403) {}{};
			\node[circle, fill=black, scale=0.5, label=above right:{\Large $U_1$}] at (933*.005,0) {}{};
			\draw[dashed] (933*.005,0) -- (933*.005,-188*3/403-215*5.03/403);
			
			% Draw g_U and C_{g_U} 
			\node[circle, fill=blue, scale=0.5] at (745*.005,-5.03) {}{};
			\node[circle, fill=blue, scale=0.5, label=above: {\Large \color{blue}$g_1$}] at (745*.005,0) {}{};
			\node[circle, fill=blue, scale=0.5, label=left: {\Large \color{blue}$C(g_1)$}] at (0,-5.03) {}{};
			\draw[blue,dashed] (745*.005,0) -- (745*.005,-5.03);
			\draw[blue,dashed] (0,-5.03) -- (745*.005,-5.03);
			
			% Draw d_U and C_{d_U} 
			\node[circle, fill=blue, scale=0.5] at (1148*.005,-3) {}{};
			\node[circle, fill=blue, scale=0.5, label=above: {\Large \color{blue}$d_1$}] at (1148*.005,0) {}{};
			\node[circle, fill=blue, scale=0.5, label=left: {\Large \color{blue}$C(d_1)$}] at (0,-3) {}{};
			\draw[blue,dashed] (1148*.005,0) -- (1148*.005,-3);
			\draw[blue,dashed] (0,-3) -- (1148*.005,-3);
			
			% Draw the axis
			\draw [thin, draw=gray, ->] (0,0) -- (1400*.005,0);
			\draw [thin, draw=gray, ->] (0,0) -- (0,-5.5);
			
			\node[circle, fill=black, scale=0.5, label=above:{\Large $T$}] at (1369*.005,0) {}{};
			\draw[very thick, draw=blue] (0,0) -- (1369*.005,0);
			\end{tikzpicture}
			\subcaption{\Large First face}
	\end{subfigure}}
	\centering{}\resizebox{.32\linewidth}{!}{
		\begin{subfigure}[c]{0.5\textwidth}
			\begin{tikzpicture}
			% Draw the Levy process
			\pgfmathsetseed{101101}
			\BrownianMotion{0}{0}{100}{0.005}{-0.2}{black}{};%1
			\BrownianMotion{101*.005}{-2.5}{123}{0.005}{-0.2}{black}{};%2
			\BrownianMotion{225*.005}{-2.9}{54}{0.005}{-0.2}{black}{};%3
			\BrownianMotion{280*.005}{-2.4}{172}{0.005}{-0.2}{black}{};%4
			\BrownianMotion{453*.005}{-0.8}{27}{0.005}{-0.2}{black}{};%5
			\BrownianMotion{481*.005}{-2.0}{80}{0.005}{-0.2}{black}{};%6
			\BrownianMotion{562*.005}{-1.3}{69}{0.005}{-0.2}{black}{};%7
			\BrownianMotion{632*.005}{-4.6}{113}{0.005}{-0.2}{black}{};%8-1
			\phantom{
				\BrownianMotion{632*.005}{-4.6}{13}{0.005}{-0.2}{black}{};%8-2
				\BrownianMotion{759*.005}{-0.4}{208}{0.005}{-0.2}{black}{};%9
				\BrownianMotion{968*.005}{-0.8}{67}{0.005}{-0.2}{black}{};%10
				\BrownianMotion{1036*.005}{-1.8}{113}{0.005}{-0.2}{black}{};%11-1
			}
			\BrownianMotion{1149*.005}{-3.0}{19}{0.005}{-0.2}{black}{};%11-2
			\BrownianMotion{1169*.005}{-0.1}{164}{0.005}{-0.2}{black}{};%12
			\BrownianMotion{1324*.005}{-1.0}{45}{0.005}{-0.2}{black}{};%13
			
			%\draw[help lines] (0,-1.2) grid (1370*.005,5.5);
			
			% Draw the Concave Majorant
			\draw[red] (0,0) -- (6*.005,-.455) 
			-- (18*.005,-.95)
			-- (180*.005,-4.6)
			-- (199*.005,-4.96)
			-- (635*.005,-5.28)
			-- (668*.005,-5.31)
			--(745*.005,-4.99){};
			\draw[red] (1148*.005,-3.0)
			-- (1275*.005,-2.08)
			-- (1337*.005,-1.4)
			-- (1362*.005,-.89)
			-- (1369*.005,-.68)
			node[right] {\Large $C$};
			
			% Draw Uniform U
			\node[circle, fill=black, scale=0.5] at (433*.005,-202*5.33/436-234*5.02/436) {}{};
			\node[circle, fill=black, scale=0.5, label=above:{\Large $U_2$}] at (433*.005,0) {}{};
			\draw[dashed] (433*.005,0) -- (433*.005,-202*5.33/436-234*5.02/436);
			
			% Draw g_U and C_{g_U} 
			\node[circle, fill=blue, scale=0.5] at (199*.005,-5.02) {}{};
			\node[circle, fill=blue, scale=0.5, label=above: {\Large \color{blue}$g_2$}] at (199*.005,0) {}{};
			\node[circle, fill=blue, scale=0.5, label=above left: {\Large \color{blue}$C(g_2)$}] at (0,-5.02) {}{};
			\draw[blue,dashed] (199*.005,0) -- (199*.005,-5.02);
			\draw[blue,dashed] (0,-5.02) -- (199*.005,-5.02);
			
			% Draw d_U and C_{d_U} 
			\node[circle, fill=blue, scale=0.5] at (635*.005,-5.33) {}{};
			\node[circle, fill=blue, scale=0.5, label=above: {\Large \color{blue}$d_2$}] at (635*.005,0) {}{};
			\node[circle, fill=blue, scale=0.5, label=left: {\Large \color{blue}$C(d_2)$}] at (0,-5.33) {}{};
			\draw[blue,dashed] (635*.005,0) -- (635*.005,-5.33);
			\draw[blue,dashed] (0,-5.33) -- (635*.005,-5.33);
			
			% Draw the axis
			\draw [thin, draw=gray, ->] (0,0) -- (1400*.005,0);
			\draw [thin, draw=gray, ->] (0,0) -- (0,-5.5);
			\draw [very thick, draw=blue] (0,0) -- (745*.005,0);
			\draw [very thick, draw=blue] (1148*.005,0) -- (1369*.005,0);
			
			\node[circle, fill=black, scale=0.5, label=above:{\Large $T$}] at (1369*.005,0) {}{};
			
			\end{tikzpicture}
			\subcaption{\Large Second face}    
	\end{subfigure}}
	\centering{}\resizebox{.32\linewidth}{!}{
		\begin{subfigure}[r]{0.5\textwidth}
			\begin{tikzpicture}
			% Draw the Levy process
			\pgfmathsetseed{101101}
			\BrownianMotion{0}{0}{100}{0.005}{-0.2}{black}{};%1
			\BrownianMotion{101*.005}{-2.5}{98}{0.005}{-0.2}{black}{};%2-1
			\phantom{
				\BrownianMotion{101*.005}{-2.5}{25}{0.005}{-0.2}{black}{};%2-2
				\BrownianMotion{225*.005}{-2.9}{54}{0.005}{-0.2}{black}{};%3
				\BrownianMotion{280*.005}{-2.4}{172}{0.005}{-0.2}{black}{};%4
				\BrownianMotion{453*.005}{-0.8}{27}{0.005}{-0.2}{black}{};%5
				\BrownianMotion{481*.005}{-2.0}{80}{0.005}{-0.2}{black}{};%6
				\BrownianMotion{562*.005}{-1.3}{69}{0.005}{-0.2}{black}{};%7
				\BrownianMotion{632*.005}{-4.6}{2}{0.005}{-0.2}{black}{};%8-1
			}
			\BrownianMotion{635*.005}{-4.95}{111}{0.005}{-0.2}{black}{};%8-1
			\phantom{
				\BrownianMotion{632*.005}{-4.6}{13}{0.005}{-0.2}{black}{};%8-2
				\BrownianMotion{759*.005}{-0.4}{208}{0.005}{-0.2}{black}{};%9
				\BrownianMotion{968*.005}{-0.8}{67}{0.005}{-0.2}{black}{};%10
				\BrownianMotion{1036*.005}{-1.8}{113}{0.005}{-0.2}{black}{};%11-1
			}
			\BrownianMotion{1149*.005}{-3.0}{19}{0.005}{-0.2}{black}{};%11-2
			\BrownianMotion{1169*.005}{-0.1}{164}{0.005}{-0.2}{black}{};%12
			\BrownianMotion{1324*.005}{-1.0}{45}{0.005}{-0.2}{black}{};%13
			
			% Draw the Concave Majorant
			\draw[red] (0,0) -- (6*.005,-.455) 
			-- (18*.005,-.95)
			-- (180*.005,-4.6)
			-- (199*.005,-4.96){};
			\draw[red] (635*.005,-5.28)
			-- (668*.005,-5.31)
			--(745*.005,-4.99){};
			\draw[red] (1148*.005,-3.0)
			-- (1275*.005,-2.08)
			-- (1337*.005,-1.4)
			-- (1362*.005,-.89)
			-- (1369*.005,-.68)
			node[right] {\Large $C$};
			
			% Draw Uniform U
			\node[circle, fill=black, scale=0.5] at (150*.005,-30*.95/162-132*4.6/162) {}{};
			\node[circle, fill=black, scale=0.5, label=above:{\Large $U_3$}] at (150*.005,0) {}{};
			\draw[dashed] (150*.005,0) -- (150*.005,-30*.95/162-132*4.6/162);
			
			% Draw g_U and C_{g_U} 
			\node[circle, fill=blue, scale=0.5] at (18*.005,-.95) {}{};
			\node[circle, fill=blue, scale=0.5, label=above:{\Large \color{blue}$g_3$}] at (18*.005,0) {}{};
			\node[circle, fill=blue, scale=0.5, label=below left:{\Large \color{blue}$C(g_3)$}] at (0,-.95) {}{};
			\draw[blue,dashed] (18*.005,0) -- (18*.005,-.95);
			\draw[blue,dashed] (0,-.95) -- (18*.005,-.95);
			
			% Draw d_U and C_{d_U} 
			\node[circle, fill=blue, scale=0.5] at (180*.005,-4.6) {}{};
			\node[circle, fill=blue, scale=0.5, label=above right:{\Large \color{blue}$d_3$}] at (180*.005,0) {}{};
			\node[circle, fill=blue, scale=0.5, label=below left:{\Large \color{blue}$C(d_3)$}] at (0,-4.6) {}{};
			\draw[blue,dashed] (180*.005,0) -- (180*.005,-4.6);
			\draw[blue,dashed] (0,-4.6) -- (180*.005,-4.6);
			
			% Draw the axis
			\draw [thin, draw=gray, ->] (0,0) -- (1400*.005,0);
			\draw [thin, draw=gray, ->] (0,0) -- (0,-5.5);
			\draw [very thick, draw=blue] (0,0) -- (199*.005,0);
			\draw [very thick, draw=blue] (635*.005,0) -- (745*.005,0);
			\draw [very thick, draw=blue] (1148*.005,0) -- (1369*.005,0);
			
			\node[circle, fill=black, scale=0.5, label=above:{\Large $T$}] at (1369*.005,0) {}{};
			
			\end{tikzpicture}
			\subcaption{\Large Third face}     
	\end{subfigure}}
	\caption{\footnotesize Selecting the first three faces of the concave 
		majorant: the total length of the thick blue segment(s) on the abscissa 
		equal the stick sizes $T$, $T-(d_1-g_1)$ and $T-(d_1-g_1)-(d_2-g_2)$,
		respectively. The independent random variables $U_1,U_2,U_3$ are uniform 
		on the sets $[0,T]$, $[0,T]\setminus(g_1,d_1)$, 
		$[0,T]\setminus\bigcup_{i=1}^2(g_i,d_i)$, respectively. Note that the residual 
		length after $n$ samples is $L_n$.}\label{fig:FacesCM}%of unsampled faces 
\end{figure}
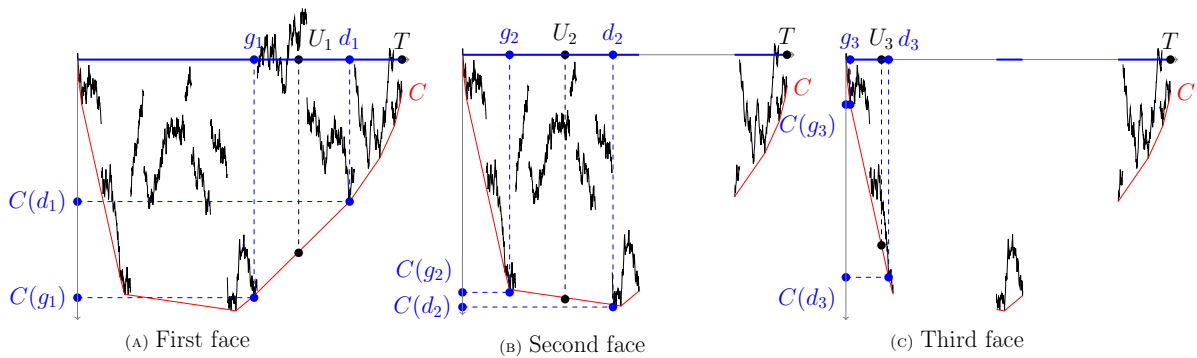

Put differently, choose the faces of $C$ \textit{independently at random 
uniformly on lengths}, as shown in Figure~\ref{fig:FacesCM}, and let $g_n$ and 
$d_n$ be the left and right ends of the $n$-th face, respectively. One way of inductively 
constructing the variables $(U_n)_{n\in\N}$  (and hence the 
sequence of the faces of $C$) in Figure~\ref{fig:FacesCM} is from an independent 
identically distributed (iid) sequence $\mathcal{V}$ of uniforms on $[0,T]$, 
which is independent of $X$: $U_1$ is the first value in $\mathcal{V}$ and, 
for any $n\in\N=\{ 1,2,\ldots\}$, $U_{n+1}$ is the first value in $\mathcal{V}$ 
after $U_n$ not contained in the union of intervals $\bigcup_{i=1}^n(g_i,d_i)$. 
Then, for any $n\in\N$, the length of the $n$-th face is $\ell_n=d_n-g_n$ and its 
height is $\xi_n=C(d_n)-C(g_n)$. In~\citep[Thm~1]{MR2978134}, a complete 
description of the law of the sequence $((\ell_n,\xi_n))_{n\in\N}$ is given. In 
order to generalise this results to L\'evy meanders, it is helpful to state the 
characterisation in terms of Dirichlet processes, see~\eqref{eq:levy-minorant} 
in Section~\ref{subsec:MarkedDirProc} below.

The behaviour of certain statistics of the path of $X$, such as the minimum
$\un{X}_T=\inf_{t\in[0,T]}X_t$ and its temporal location 
$\tau_T=\tau_{[0,T]}(X)=\inf\{t>0:\min\{X_t,X_{t-}\}=\un{X}_T\}$, is determined 
by that of the faces of $C$ whose heights are negative (we assume throughout 
that $X$ is right-continuous with left limits (\cadlag) and denote 
$X_{t-}=\lim_{s\ua t}X_s$ for $t>0$ and $X_{0-}=0$). Analysis of their behaviour 
amounts to the analysis of the convex minorants of the  pre- and post-minimum 
processes $X^\la=(X^\la_t)_{t\in[0,T]}$ and $X^\ra=(X^\ra_t)_{t\in[0,T]}$, where 
\begin{equation}\label{def:pre_post}
X^{\la}_t =
\begin{cases}
X_{(\tau_T-t)-}-\un{X}_T, & t\in[0,\tau_T],\\
\dag, & t\in(\tau_T,T],
\end{cases} 
%\enskip
\text{ and}
\enskip
X^{\ra}_t =
\begin{cases}
X_{\tau_T+t}-\un{X}_T, & t\in[0,T-\tau_T],\\
\dag, & t\in(\tau_T,T],
\end{cases}
\end{equation}
respectively ($\dag$ denotes a cemetery state, required only to define the 
processes on $[0,T]$). Clearly, as indicated by 
Figure~\ref{fig:SplitConvexMinorant}, $C$ may be recovered from the convex 
minorants $C^{\la}=C(X^\la)$ and $C^\ra=C(X^\ra)$
of $X^{\la}|_{[0,\tau_T]}$ and $X^{\ra}|_{[0,T-\tau_T]}$,
respectively. For convenience, we suppress the time interval in the notation for  
$C^{\la}=C(X^\la)$ and $ C^{\ra}=C(X^\ra)$.
In particular, throughout the paper, $C^{\la}$ and $C^{\ra}$
are the convext minorants of $X^\la$ and $X^\ra$, respectively,
only while the processes are ``alive''.

\begin{figure}[H] % Requires \usepackage{subcaption}
	\centering{}\resizebox{.32\linewidth}{!}{
		\begin{subfigure}[l]{0.55\textwidth}
			\begin{tikzpicture}
			% Draw the Levy process
			\pgfmathsetseed{101101}
			\BrownianMotion{0}{0}{100}{0.005}{-0.2}{black}{};%1
			\BrownianMotion{101*.005}{-2.5}{123}{0.005}{-0.2}{black}{};%2
			\BrownianMotion{225*.005}{-2.9}{54}{0.005}{-0.2}{black}{};%3
			\BrownianMotion{280*.005}{-2.4}{172}{0.005}{-0.2}{black}{};%4
			\BrownianMotion{453*.005}{-0.8}{27}{0.005}{-0.2}{black}{};%5
			\BrownianMotion{481*.005}{-2.0}{80}{0.005}{-0.2}{black}{};%6
			\BrownianMotion{562*.005}{-1.3}{69}{0.005}{-0.2}{black}{};%7
			\BrownianMotion{632*.005}{-4.6}{126}{0.005}{-0.2}{black}{};%8
			\BrownianMotion{759*.005}{-0.4}{208}{0.005}{-0.2}{black}{};%9
			\BrownianMotion{968*.005}{-0.8}{67}{0.005}{-0.2}{black}{};%10
			\BrownianMotion{1036*.005}{-1.8}{132}{0.005}{-0.2}{black}{};%11
			\BrownianMotion{1169*.005}{-0.1}{164}{0.005}{-0.2}{black}{};%12
			\BrownianMotion{1324*.005}{-1.0}{45}{0.005}{-0.2}{black}{\Large $X$};%13
			
			%\draw[help lines] (0,-1.2) grid (1370*.005,5.5);
			
			% Draw the Concave Majorant
			\draw[red] (0,0) -- (6*.005,-.455) 
			-- (18*.005,-.95)
			-- (180*.005,-4.6)
			-- (199*.005,-4.98)
			-- (635*.005,-5.28)
			-- (668*.005,-5.31){};
			\draw[red, dashed] (668*.005,-5.31)
			-- (745*.005,-4.99)
			-- (1148*.005,-3.0)
			-- (1275*.005,-2.08)
			-- (1337*.005,-1.4)
			-- (1362*.005,-.89)
			-- (1369*.005,-.68)
			node[above right] {\Large $C$};
			
			% Draw the axis
			\draw [thin, draw=gray, ->] (0,0) -- (1400*.005,0);
			\draw [thin, draw=gray, ->] (0,0) -- (0,-5.5);
			
			\node[circle, fill=black, scale=0.5, label=above:{\Large $\tau_T$}] at (668*.005,0) {}{};
			\draw[densely dotted] (668*.005,0) -- (668*.005,-5.31){};
			
			\node[circle, fill=black, scale=0.5, label=above:{\Large $T$}] at (1369*.005,0) {}{};
			\end{tikzpicture}
			\subcaption{\Large $C$}
	\end{subfigure}}
	\centering{}\resizebox{.32\linewidth}{!}{
		\begin{subfigure}[l]{0.55\textwidth}
			\begin{tikzpicture}
			% Draw the Levy process
			\pgfmathsetseed{101101}
			\BrownianMotion{668*.005-0}{5.31-0}{100}{-0.005}{-0.2}{black}{};%1
			\BrownianMotion{668*.005-101*.005}{5.31-2.5}{123}{-0.005}{-0.2}{black}{};%2
			\BrownianMotion{668*.005-225*.005}{5.31-2.9}{54}{-0.005}{-0.2}{black}{};%3
			\BrownianMotion{668*.005-280*.005}{5.31-2.4}{172}{-0.005}{-0.2}{black}{};%4
			\BrownianMotion{668*.005-453*.005}{5.31-0.8}{27}{-0.005}{-0.2}{black}{};%5
			\BrownianMotion{668*.005-481*.005}{5.31-2.0}{80}{-0.005}{-0.2}{black}{};%6
			\BrownianMotion{668*.005-562*.005}{5.31-1.3}{69}{-0.005}{-0.2}{black}{};%7
			\BrownianMotion{668*.005-632*.005}{5.31-4.6}{30}{-0.005}{-0.2}{black}{};%8-1
			\phantom{
				\BrownianMotion{0}{5.31-4.6}{90}{0.005}{-0.2}{black}{};%8-2
				\BrownianMotion{0}{5.31-0.4}{208}{0.005}{-0.2}{black}{};%9
				\BrownianMotion{0}{5.31-0.8}{67}{0.005}{-0.2}{black}{};%10
				\BrownianMotion{0}{5.31-1.8}{132}{0.005}{-0.2}{black}{};%11
				\BrownianMotion{0}{5.31-0.1}{164}{0.005}{-0.2}{black}{};%12
				\BrownianMotion{0}{5.31-1.0}{45}{0.005}{-0.2}{black}{};%13
			}
			\node[circle, scale=0.5, label=above right:{\Large $X^\la$}] at (668*.005,5.31) {}{};
			\node[circle, scale=0.5, label=below right:{\Large \color{red}$C^\la$}] at (668*.005,5.31) {}{};
			
			% Draw the Concave Majorant
			\draw[red, mark=+] 
			(668*.005-0,5.31-0) 
			-- (668*.005-6*.005,5.31-.455)
			-- (668*.005-18*.005,5.31-.95)
			-- (668*.005-180*.005,5.31-4.6)
			-- (668*.005-199*.005,5.31-4.98)
			-- (668*.005-635*.005,5.31-5.28)
			-- (668*.005-668*.005,5.31-5.31) % maximum
			%			-- (668*.005-745*.005,5.31-4.99)
			%			-- (668*.005-1148*.005,5.31-3.0)
			%			-- (668*.005-1275*.005,5.31-2.08)
			%			-- (668*.005-1337*.005,5.31-1.4)
			%			-- (668*.005-1362*.005,5.31-.89)
			%			-- (668*.005-1369*.005,5.31-.68)
			{};
			
			% Draw the axis
			\draw [thin, draw=gray, ->] (0,0) -- (1400*.005,0);
			\draw [thin, draw=gray, ->] (0,0) -- (0,5.5);
			
			\node[circle, fill=black, scale=0.5, label=below:{\Large $\tau_T$}] at (668*.005,0) {}{};
			\draw[densely dotted] (668*.005,0) -- (668*.005,5.31){};
			
			\node[circle, scale=0.5, label=below:{\Large $T$}] at (1369*.005,0) {}{};
			\end{tikzpicture}
			\subcaption{\Large $C^\la$}
	\end{subfigure}}
	\centering{}\resizebox{.32\linewidth}{!}{
		\begin{subfigure}[l]{0.55\textwidth}
			\begin{tikzpicture}
			% Draw the Levy process
			\pgfmathsetseed{101101}
			\phantom{
				\BrownianMotion{0}{5.31-0}{100}{0.005}{-0.2}{black}{};%1
				\BrownianMotion{0}{5.31-2.5}{123}{0.005}{-0.2}{black}{};%2
				\BrownianMotion{0}{5.31-2.9}{54}{0.005}{-0.2}{black}{};%3
				\BrownianMotion{0}{5.31-2.4}{172}{0.005}{-0.2}{black}{};%4
				\BrownianMotion{0}{5.31-0.8}{27}{0.005}{-0.2}{black}{};%5
				\BrownianMotion{0}{5.31-2.0}{80}{0.005}{-0.2}{black}{};%6
				\BrownianMotion{0}{5.31-1.3}{69}{0.005}{-0.2}{black}{};%7
				\BrownianMotion{0}{5.31-4.6}{30}{0.005}{-0.2}{black}{};%8-1
			}
			\BrownianMotion{668*.005-668*.005}{5.31-5.13}{96}{0.005}{-0.2}{black}{};%8-2
			\BrownianMotion{759*.005-668*.005}{5.31-0.4}{208}{0.005}{-0.2}{black}{};%9
			\BrownianMotion{968*.005-668*.005}{5.31-0.8}{67}{0.005}{-0.2}{black}{};%10
			\BrownianMotion{1036*.005-668*.005}{5.31-1.8}{132}{0.005}{-0.2}{black}{};%11
			\BrownianMotion{1169*.005-668*.005}{5.31-0.1}{164}{0.005}{-0.2}{black}{};%12
			\BrownianMotion{1324*.005-668*.005}{5.31-1.0}{45}{0.005}{-0.2}{black}{};%13
			% Draw the Concave Majorant
			\draw[red, dashed] 
			%			(668*.005-0,5.31-0) 
			%			-- (668*.005-6*.005,5.31-.455)
			%			-- (668*.005-18*.005,5.31-.95)
			%			-- (668*.005-180*.005,5.31-4.6)
			%			-- (668*.005-199*.005,5.31-4.98)
			%			-- (668*.005-635*.005,5.31-5.28)
			%			-- 
			(668*.005-668*.005,5.31-5.31) % maximum
			-- (745*.005-668*.005,5.31-4.99)
			-- (1148*.005-668*.005,5.31-3.0)
			-- (1275*.005-668*.005,5.31-2.08)
			-- (1337*.005-668*.005,5.31-1.4)
			-- (1362*.005-668*.005,5.31-.89)
			-- (1369*.005-668*.005,5.31-.68)
			node[above right] {};
			
			% Draw the axis
			\draw [thin, draw=gray, ->] (0,0) -- (1400*.005,0);
			\draw [thin, draw=gray, ->] (0,0) -- (0,5.5);
			
			\node[circle, scale=0.5, label=above right:{\Large $X^\ra$}] at (1369*.005-668*.005,5.31-.68) {}{};
			\node[circle, scale=0.5, label=below right:{\Large \color{red} $C^\ra$}] at (1369*.005-668*.005,5.31-.68) {}{};
			
			\node[circle, fill=black, scale=0.5, label=below:{\Large $T-\tau_T$}] at (1369*.005-668*.005,0) {}{};
			\draw[densely dotted] (1369*.005-668*.005,0) -- (1369*.005-668*.005,5.31-.68){};
			
			\node[circle, scale=0.5, label=below:{\Large $T$}] at (1369*.005,0) {}{};
			\end{tikzpicture}
			\subcaption{\Large $C^\ra$}
	\end{subfigure}}
	\caption{\footnotesize Decomposing $(X,C)$ into 
		$(X^\la, C^\la)$ and $(X^\ra, C^\ra)$.}
	\label{fig:SplitConvexMinorant}
\end{figure}

\subsection{Convex minorants as marked Dirichlet processes}\label{subsec:MarkedDirProc}
Our objective now is to obtain a description of the law of  the convex minorants 
$C^\la$ and $C^\ra$. For any $n\in\N$ and positive reals 
$\theta_0,\ldots,\theta_n>0$, the \textit{Dirichlet distribution} with parameter 
$(\theta_0,\ldots,\theta_n)$ is given by a density proportional to $x\mapsto
\prod_{i=0}^nx_i^{\theta_i-1}$, supported on the standard $n$-dimensional 
symplex in $\R^{n+1}$ (i.e. the set of points $x=(x_0,\ldots,x_n)$ satisfying 
$\sum_{i=0}^n x_i=1$ and $x_i\in(0,1)$ for $i\in\{0,\ldots,n\}$). In the 
special case $n=1$, we get  the beta distribution $\Beta(\theta_0,\theta_1)$ on 
$[0,1]$. In particular, the uniform distribution equals $U(0,1)=\Beta(1,1)$ and, 
for any $\theta>0$, we append the limiting cases $\Beta(\theta,0)=\delta_1$ 
and $\Beta(0,\theta)=\delta_0$, where $\delta_x$ is the Dirac measure at $x\in\R$.

Let $(\mathbb{X},\mathcal{X},\mu)$ be a measure space with
$\mu(\mathbb{X})\in(0,\infty)$. A random probability measure $\Xi$ (i.e. a 
stochastic process indexed by the sets in $\mathcal{X}$) is a \textit{Dirichlet 
process} on $(\mathbb{X},\mathcal{X})$ based on the finite measure $\mu$ if for 
any measurable partition $\{B_0,\ldots,B_n\}\subset\mathcal{X}$ (i.e. $n\in\N$, 
$B_i\cap B_j=\emptyset$ for all distinct $i,j\in\{0,\ldots,n\}$, $\bigcup_{i=0}^n
B_i=\mathbb{X}$ and $\mu(B_i)>0$ for all $i\in\{0,\ldots,n\}$), the vector 
$(\Xi(B_0),\ldots,\Xi(B_n))$ is a Dirichlet random vector with parameters 
$(\mu(B_0),\ldots,\mu(B_n))$. We use the notation $\Xi\sim\D_\mu$ throughout.

Define the sets $\mZ^n=\{k\in\mathbb{Z}:k<n\}$ and $\mZ^n_m=\mZ^n\setminus\mZ^m$ 
for $n,m\in\mathbb{Z}$ and adopt the convention $\prod_{k\in\emptyset}=1$ and 
$\sum_{k\in\emptyset}=0$ ($\mathbb{Z}$ denotes the integers). 
Sethuraman~\citep{MR1309433} introduced the construction 
\begin{equation}\label{eq:Dirichlet}
\Xi=\sum_{n=1}^\infty\pi_{n}\delta_{x_n}\sim\D_\mu,
\end{equation}
where $(x_n) _{n\in\N}$ is an iid sequence with distribution 
$\mu(\mathbb{X})^{-1}\mu$ and $(\pi_n)_{n\in\N}$ is a stick-breaking process 
based on $\Beta(1,\mu(\mathbb{X}))$ constructed as follows:
$\pi_n=\beta_n\prod_{k\in\mZ_1^n}(1-\beta_k)$ where $(\beta_n)_{n\in\N}$ is an 
iid sequence with distribution $\Beta(1,\mu(\mathbb{X}))$. 

Consider a further measurable space $(\mathbb{Y},\mathcal{Y})$ and a triple 
$(\theta,\mu,\kappa)$, where $\theta>0$, $\mu$ is a finite measure on 
$(\mathbb{X},\mathcal{X})$ and $\kappa:[0,\theta]\times\mathbb{X}\to\mathbb{Y}$ 
is a measurable function. Let $\Xi$ be as in~\eqref{eq:Dirichlet}. A
\textit{marked Dirichlet process} on $(\mathbb{Y},\mathcal{Y})$ is given by the 
random probability measure $\sum_{n=1}^\infty\pi_n\delta_{\kappa(\theta\pi_n,x_n)}$ 
on $(\mathbb{Y},\mathcal{Y})$. We denote its distribution by
$\D_{(\theta,\mu,\kappa)}$.

Let $F(t,x)=\P(X_t\leq x)$, $x\in\R$, be the distribution function of $X_t$ for 
$t\geq0$ and let $G(t,\cdot)$ be the generalised right inverse of $F(t,\cdot)$. 
Hence $G(t,U)$ follows the law $F(t,\cdot)$ for any uniform random variable 
$U\sim U(0,1)$. Given these definitions,~\citep[Thm~1]{MR2978134} can be 
rephrased as 
\begin{equation}\label{eq:levy-minorant}
T^{-1}\sum_{n=1}^\infty \ell_n\delta_{\xi_n}=
T^{-1}\sum_{n=1}^\infty(d_n-g_n)\delta_{ C(d_n)- C(g_n)}\sim
\D_{(T,U(0,1),G)},
\end{equation}
where $C$ is the convex minorant of $X$ over the interval $[0,T]$, with the 
length and height of the $n$-th face given by $\ell_n=d_n-g_n$ and $\xi_n=
C(d_n)-C(g_n)$, respectively, as defined in Section~\ref{subec:Con_Minor_Splitt} 
above. Consequently, the faces of $C$ are easy to simulate if one can sample 
from $F(t,\cdot)$. Indeed, $(\ell_n/T)_{n\in\N}$ has the law of the 
stick-breaking process with uniform sticks and, given $\ell_n$, we have 
$\xi_n\sim F(\ell_n,\cdot)$ for all $n\in\N$. 

It is evident that the size-biased sampling of the faces of $C^\la$ and $C^\ra$, 
analogous to the one described in the second paragraph of Section~\ref{subec:Con_Minor_Splitt} for the faces of $C$ (see also 
Figure~\ref{fig:SplitConvexMinorant}), can be applied on the intervals
$[0,\tau_T]$ and $[0,T-\tau_T]$, respectively. However, in order to 
characterise the respective laws of the two sequences of lengths and heights, 
we need to restrict to the following class of L\'evy processes. 
\begin{assumption*}[P]\label{asm:(P)} 
The map $t\mapsto\P(X_t>0)$ is constant for $t>0$ and equals some 
$\rho\in[0,1]$.
\end{assumption*}

The family of L\'evy processes that satisfy (\nameref{asm:(P)}) has, to the best of 
our knowledge, not been characterised in terms of the characteristics of the 
process $X$ (e.g. its L\'evy measure or characteristic exponent). However, it is 
easily seen that it includes the following wide variety of examples: symmetric L\'evy 
processes with $\rho=1/2$, stable processes with $\rho$ given by its positivity 
parameter (see e.g.~\cite[App.~A]{ExactSim}) and subordinated stable processes 
%where the subordinator is a.s. strictly increasing, 
with $\rho$ equal to the positivity parameter of the stable process. Note also that 
under~(\nameref{asm:(P)}), the random variable $G(t,U)$ (with $U$ uniform on 
$[0,1]$) is negative if and only if $U\leq 1-\rho$.

\begin{prop}\label{prop:pre_post}
Let $X$ be a L\'{e}vy process on $[0,T]$ satisfying~(\nameref{asm:(D)}) 
and~(\nameref{asm:(P)}) for $\rho\in[0,1]$ with pre- and post-minimum processes
$X^\la$ and $X^\ra$, respectively, defined in~\eqref{def:pre_post}. Let 
$((\ell_n^\la,-\xi^\la_n))_{n\in\N}$ and $((\ell_n^\ra,\xi^\ra_n))_{n\in\N}$ be 
the faces of $C^\la=C(X^\la)$ and $C^\ra=C(X^\ra)$, respectively, when sampled 
independently at random uniformly on lengths as described in 
Section~\ref{subec:Con_Minor_Splitt}. Then $\tau_T/T$ follows the law 
$\Beta(1-\rho,\rho)$, the random functions $C^\la$ and $C^\ra$ are conditionally 
independent given $\tau_T$ and, conditional on $\tau_T$, we have
\begin{equation}\label{eq:pre_post}
\begin{split}
\tau_T^{-1}\sum_{n=1}^\infty \ell_n^\la\delta_{\xi_n^\la}
&\sim\D_{(\tau_T,U(0,1)|_{[0,1-\rho]},G)},\\
(T-\tau_T)^{-1}\sum_{n=1}^\infty \ell_n^\ra\delta_{\xi_n^\ra}
&\sim\D_{(T-\tau_T,U(0,1)|_{[1-\rho,1]},G)}.
\end{split}
\end{equation}
\end{prop}

\begin{rem}\label{rem:pre_post}
(i) The measure $U(0,1)|_{[0,1-\rho]}$ (resp. $U(0,1)|_{[1-\rho,1]}$) on the 
interval $[0,1]$ has a density $x\mapsto \1_{[0,1-\rho]}(x)$ (resp. $x\mapsto
\1_{[1-\rho,1]}(x)$).\footnote{Here and throughout $\1_A$ denotes the indicator 
function of a set $A$.} In the case $\rho=1$, $X$ is a subordinator 
by~\cite[Thm~24.11]{MR3185174}. Then $\tau_T=T$ and only the first equality in 
law in~\eqref{eq:pre_post} makes sense (since there is no pre-minimum process) 
and equals that in~\eqref{eq:levy-minorant}. The case $\rho=0$ is analogous.\\
(ii) Proposition~\ref{prop:pre_post} provides a simple proof of the generalized 
arcsine law: under~(\nameref{asm:(D)}) and~(\nameref{asm:(P)}), 
we have $\tau_T/T\sim \Beta(1-\rho,\rho)$ (see~\citep[Thm~VI.3.13]{MR1406564} 
for a classical proof of this result).\\
(iii) Proposition~\ref{prop:pre_post} implies that  the heights 
$(\xi_n^\la)_{n\in\N}$ (resp. $(\xi_n^\ra)_{n\in\N}$) of the faces of the convex 
minorant $C^\la$ (resp. $C^\ra$) are conditionally independent given 
$(\ell_n^\la)_{n\in\N}$ (resp. $(\ell_n^\ra)_{n\in\N}$). Moreover, $\xi_n^\la$ 
(resp. $\xi_n^\ra$) is distributed as $F(\ell_n^\la,\cdot)$ 
(resp. $F(\ell_n^\ra,\cdot)$) conditioned to the negative (resp. positive) 
half-line. Given~$\tau_T$, the sequence $(\ell_n^\la/\tau_T)_{n\in\N}$ 
(resp. $(\ell_n^\ra/(T-\tau_T))_{n\in\N}$) is a stick-breaking process based on 
$\Beta(1,1-\rho)$ (resp. $\Beta(1,\rho)$).\\
(iv) If $T$ is an exponential random variable with mean $\theta>0$ 
independent of $X$, the random times $\tau_T$ and $T-\tau_T$ are 
independent gamma random variables with common scale parameter 
$\theta$ and shape parameters $1-\rho$ and $\rho$, respectively. This is 
because, the distribution of $\tau_T/T$, conditional on any value of $T$, is 
$\Beta(1-\rho,\rho)$ (see Proposition~\ref{prop:pre_post}), making $\tau_T/T$ 
and $T$ independent. Furthermore, by~\citep[Cor.~2]{MR2978134}, the 
random measures $\sum_{n=1}^\infty\delta_{(\ell_n^\la,\xi_n^\la)}$ and 
$\sum_{n=1}^\infty\delta_{(\ell_n^\ra,\xi_n^\ra)}$ are independent Poisson 
point processes with intensities given by the restriction of the measure 
$e^{-t/\theta}t^{-1}dt\P(X_t\in dx)$ on $(t,x)\in[0,\infty)\times\R$ to the subsets
$[0,\infty)\times(-\infty,0)$ and $[0,\infty)\times[0,\infty)$, respectively.
\end{rem}

The proof of Proposition~\ref{prop:pre_post} relies on the following property 
of Dirichlet processes, which is a direct consequence of the definition and~\citep[Lem.~3.1]{MR1309433}.

\begin{lem}\label{lem:Dir-decomp}
Let $\mu_{1}$ and $\mu_{2}$ be two non-trivial finite measures on a measurable 
space $(\mathbb{X},\mathcal{X})$. Let $\Xi_i\sim\D_{\mu_i}$ for $i=1,2$ and 
$\beta\sim \Beta(\mu_1(\mathbb{X}),\mu_2(\mathbb{X}))$ be jointly independent, 
then \[\beta\Xi_1+(1-\beta)\Xi_2\sim\D_{\mu_1+\mu_2}.\]
\end{lem}

\begin{proof}[Proof of Proposition~\ref{prop:pre_post}]
Recall that $\ell_n=d_n-g_n$ (resp. $\xi_n=C(d_n)-C(g_n)$) denotes the length 
(resp. height) of the $n$-th face of the convex minorant $C$ of $X$ (see 
Section~\ref{subec:Con_Minor_Splitt} above for definition).
By~\eqref{eq:levy-minorant}, the random variables $\upsilon_n=F(\ell_n,\xi_n)$ 
form a $U(0,1)$ distributed iid sequence $(\upsilon_n)_{n\in\N}$ independent of 
the stick-breaking process $(\ell_n)_{n\in\N}$. Since the faces of $C$ are 
placed in a strict ascending order of slopes, 
by~\eqref{eq:Dirichlet}--\eqref{eq:levy-minorant} the convex minorant $C$ of a 
path of $X$ is in a one-to-one correspondence with a realisation of the marked 
Dirichlet process $T^{-1}\sum_{n=1}^\infty \ell_n\delta_{\xi_n}$ and thus with 
the Dirichlet process $\Xi=T^{-1}\sum_{n=1}^\infty \ell_n\delta_{\upsilon_n}
\sim\D_{U(0,1)}$.

Assume now that $\rho\in(0,1)$. Since $U(0,1)|_{[0,1-\rho]}+U(0,1)|_{[1-\rho,1]}
=U(0,1)$ as measures on the interval $[0,1]$, Lemma~\ref{lem:Dir-decomp} 
and~\citep[Thm~5.10]{MR1876169} imply that by possibly extending the probability 
space we may decompose $\Xi=\beta\Xi^\la+(1-\beta)\Xi^\ra$, where the random 
elements $\beta\sim \Beta(1-\rho,\rho)$, $\Xi^\la\sim\D_{U(0,1)|_{[0,1-\rho]}}$ 
and $\Xi^\ra\sim\D_{U(0,1)|_{[1-\rho,1]}}$ are independent (note that we can 
distinguish between values above and below $1-\rho$ a.s. since, with probability 
$1$, no variable $\upsilon_n$ is exactly equal $1-\rho$). 
Since $\rho\in(0,1)$, condition~(\nameref{asm:(D)})
and~\cite[Thm~24.10]{MR3185174} imply that $\P(0<X_t<\epsilon)>0$ for all 
$\epsilon>0$ and $t>0$. Then~(\nameref{asm:(P)}) implies the equivalence: 
$F(t,x)\leq1-\rho$ if and only if $x\leq 0$. 

The construction of $(\upsilon_n)_{n\in\N}$ ensures that the faces of $C$ with 
negative (resp. positive) heights correspond to the atoms of $\Xi^\la$ (resp. 
$\Xi^\ra$). Therefore the identification between the faces of $C$ with the 
Dirichlet process $\Xi$ described above implies that $\Xi^\la$ (resp. $\Xi^\ra$) 
is also in one-to-one correspondence with the faces of $C^\la$ (resp. $C^\ra$). 
In particular, since $\tau_T=\sum_{n\in\N}\ell_n\cdot\1_{\{\xi_n<0\}}$ equals 
the sum of all the lengths of the faces of $C$ with negative heights, this 
identification implies $\tau_T\sim T\beta$ and the generalised arcsine law 
$\tau_T/T\sim \Beta(1-\rho,\rho)$ follows from Lemma~\ref{lem:Dir-decomp} 
applied to the measures $U(0,1)|_{[0,1-\rho]}$ and $U(0,1)|_{[1-\rho,1]}$ on 
$[0,1]$. Moreover, the lengths of the faces of $C^\la$ correspond to the masses 
of the atoms of $\beta\Xi^\la$. The independence of $\beta$ and $\Xi^\la$ 
implies that the sequence of the masses of the atoms of $\beta\Xi^\la$ is 
precisely a stick-breaking process based on the distribution $\Beta(1,1-\rho)$ 
multiplied by $\beta$. Similarly, the random variables $F(\ell_n^\la,\xi^\la_n)$ 
can be identified with the atoms of $\Xi^\la$ and thus form an iid sequence of 
uniform random variables on the interval $[0,1-\rho]$. Hence, conditional on 
$\tau_T$, the law of $\tau_T^{-1}\sum_{n=1}^\infty\ell_n^\la\delta_{\xi_n^\la}$ 
is as stated in the proposition. An analogous argument yields the correspondence 
between the Dirichlet process $\Xi^\ra$ and the faces of $C^\ra$. The fact that 
the orderings correspond to size-biased samplings follows 
from~\cite[Sec.~3.2]{MR2245368}. 

It remains to consider the case $\rho\in\{0,1\}$. By~\cite[Thm~24.11]{MR3185174}, 
$X$ (resp. $-X$) is a subordinator if $\rho=1$ (resp. $\rho=0$) 
satisfying~(\nameref{asm:(D)}). Then, clearly, $\rho=1$, $\tau_T=0$, $C^\ra=C$ 
(resp. $\rho=0$, $\tau_T=T$, $C^\la=C$) and the proposition follows 
from~\eqref{eq:levy-minorant}.% concluding the proof. 
\end{proof}

\subsection{L\'evy meanders and their convex minorants\label{subsec:Levy-meanders}}
If $0$ is regular for $(0,\infty)$, then it is possible to define the L\'evy 
meander $X^{\me,T}=(X^{\me,T}_t)_{t\in[0,T]}$ as the weak limit as 
$\varepsilon\da0$ of the law of $X$ conditioned on the event 
$\{\un{X}_T>-\varepsilon\}$ (see~\citep[Lem.~7]{MR2164035} 
and~\citep[Cor.~1]{MR2375597}). Condition~(\nameref{asm:(P)}) and Rogozin's 
criterion~\citep[Prop.~VI.3.11]{MR1406564} readily imply that 0 is regular for 
$(0,\infty)$ if $\rho>0$, in which case the respective L\'{e}vy meander is well 
defined. As discussed in Section~\ref{subsec:MarkedDirProc}, the case $\rho=0$ 
corresponds to the negative of a subordinator where the meander does not exist.

In this section we will use the following assumption, which implies the existence 
of a density of $X_t$ for every $t>0$ and hence also Assumption~(\nameref{asm:(D)}). 

\begin{assumption*}[K]\label{asm:(K)} 
$\int_\R\left|\mathbb{E}\left(e^{iuX_t}\right)\right|du<\infty$ for every $t>0$. 
\end{assumption*}

L\'evy meanders arise under certain path transformations of L\'evy 
processes~\citep[Sec.~VI.4]{MR1406564}. For instance, by~\citep[Thm~2]{MR3160578}, 
if~(\nameref{asm:(K)}) holds and $0$ is regular for both $(-\infty,0)$ and  
$(0,\infty)$, then the pre- and post-minimum processes $X^\la$ and $X^\ra$ are 
conditionally independent given $\tau_T$ and distributed as meanders of $-X$ and $X$ 
on the intervals $[0,\tau_T]$ and $[0,T-\tau_T]$, respectively, generalising the 
result for stable processes~\citep[Cor.~VIII.4.17]{MR1406564}. The next theorem 
constitutes the main result of this section. 

\begin{thm}\label{thm:meander-minorant}
Assume $X$ satisfies~(\nameref{asm:(P)}) with $\rho\in(0,1]$ 
and~(\nameref{asm:(K)}). Pick a finite time horizon $T>0$ and let $X^{\me,T}$ 
be the L\'evy meander and let $((\ell_n^\me,\xi_n^\me))_{n\in\N}$ be the lengths 
and heights of the faces of $C(X^{\me,T})$ chosen independently at random 
uniformly on lengths. The sequence $((\ell_n^\me,\xi_n^\me))_{n\in\N}$ encodes 
a marked Dirichlet process as follows:
\begin{equation}\label{eq:meander_minorant}
T^{-1}\sum_{n=1}^\infty \ell_n^\me\delta_{\xi_n^\me}
\sim\D_{(T,\rho U(1-\rho,1),G)}.
\end{equation}
\end{thm}

\begin{proof}
The case $\rho=1$ is trivial since $X$ is then a subordinator 
by~\cite[Thm~24.11]{MR3185174}, clearly equal to its meander, 
and~\eqref{eq:meander_minorant} is the same as~\eqref{eq:levy-minorant}. If 
$\rho\in(0,1)$, then $0$ is regular for both half lines by Rogozin's 
criterion~\citep[Prop.~VI.3.11]{MR1406564}. Fix $T'>T$ and consider the L\'evy process 
$X$ on $[0,T']$. Conditional on $\tau_{T'}=T'-T$, the 
post-minimum process $(X^\ra_t)_{t\in[0,T']}$ defined 
in~\eqref{def:pre_post} is killed at $\tau_{T'}=T'-T$ and the law of 
$(X^\ra_t)_{t\in[0,T]}$ prior to the killing time is the same as the law of the meander 
$X^{\me,T}$ on $[0,T]$ by~\citep[Thm~2]{MR3160578}. Hence, conditional 
on $\tau_{T'}=T'-T$, the law of the faces of the convex minorant $C(X^\ra)$ on 
$[0,T]$ agree with those of the convex minorant $C(X^{\me,T})$. Thus, the 
distributional characterisation of Proposition~\ref{prop:pre_post} also applies 
to $T^{-1}\sum_{n=1}^\infty \ell_n^\me\delta_{\xi_n^\me}$, concluding the proof.
\end{proof}

\begin{rem}\label{rem:meander_gamma}
(i) Condition~(\nameref{asm:(K)}) is slightly stronger than~(\nameref{asm:(D)}). In 
fact, it holds if there is a Brownian component or if the L\'evy measure has sufficient 
activity~\citep[Sec.~5]{MR628873} (see also Lemma~\ref{lem:Kallenberg_regularity} 
in Appendix~\ref{app:On_regularity} below). Hence Condition~(\nameref{asm:(K)}) 
is satisfied by most subordinated stable processes.\\
%(see Lemma~\ref{lem:subordinated-stable} below).
(ii) Although sufficient and simple, Condition~(\nameref{asm:(K)}) is not a 
necessary condition for~\citep[Thm~2]{MR3160578}. The minimal requirement is 
that the density $(t,x)\mapsto\frac{\partial}{\partial x}F(t,x)$ exists and is uniformly 
continuous for $t>0$ bounded away from 0.\\
(iii) Identity~\eqref{eq:levy-minorant} (in the form of~\citep[Thm~1]{MR2978134}), 
applied to concave majorants, was used in~\citep{LevySupSim} to obtain a 
geometrically convergent simulation algorithm of the triplet 
$(X_T,\ov{X}_T,\tau_T(-X))$, where $\tau_T(-X)$ is the location of the supremum 
$\ov{X}_T=\sup_{t\in[0,T]}X_t$. In the same manner, a geometrically convergent 
simulation of the marginal $X_T^{\me,T}$ can be constructed using the identity  in~\eqref{eq:meander_minorant}.\\
(iv) The proof of Theorem~\ref{thm:meander-minorant} and 
Remark~\ref{rem:pre_post}(iv) above imply that if $T$ is taken to be an 
independent gamma random variable with shape parameter $\rho$ and scale 
parameter $\theta>0$, then the random measure 
$\sum_{n=1}^\infty\delta_{(\ell_n^\me,\xi_n^\me)}$ is a Poisson point process on 
$(t,x)\in[0,\infty)\times[0,\infty)$ with intensity $e^{-t/\theta}t^{-1}dt\P(X_t\in dx)$. 
This description and~\cite[Thm~6]{MR2948693} imply~\cite[Thm~4]{MR2948693}, 
the description of the chronologically ordered faces of the convex minorant of a 
Brownian meander. However, as noted in~\cite{MR2948693}, a direct proof 
of~\cite[Thm~6]{MR2948693}, linking the chronological and Poisson point 
process descriptions of the convex minorant of a Brownian meander, appears to 
be out of reach.
%(v) Taking $\theta\to\infty$ in the description of the Poisson point process over 
%the gamma time horizon in the previous remark, we obtain the description of the 
%faces of the convex minorant of $X$ conditioned to be positive for all time 
%(via~\cite[Prop.~1]{MR2375597}). This result is consistent with the description 
%obtained via the post-minimum process~\cite{MR1232850} on an exponential time 
%interval with mean $\theta>0$ and taking the limit $\theta\to\infty$ 
%in~\citep[Cor.~2]{MR2978134}.
\end{rem}

\section{$\varepsilon$-strong simulation algorithms for convex minorants\label{sec:algorithms}}
As mentioned in the introduction, $\varepsilon$SS algorithm is a simulation 
procedure with random running time, which constructs a random element that is 
$\varepsilon$-close in the  essential supremum norm to the random element of 
interest, where $\varepsilon>0$ is an \textit{a priori} specified tolerance level. 
Moreover, the simulation procedure can be continued retrospectively if, given the 
value of the simulated random element, the tolerance level $\varepsilon$ needs 
to be reduced. Thus an $\varepsilon$SS scheme provides a way to compute the 
random element of interest to arbitrary precision almost surely,  leading to a 
number of applications (including exact and unbiased algorithms for related 
random elements) discussed in Subsection~\ref{subsec:e-app} below.

We now give a precise definition of an $\varepsilon$SS algorithm. Consider a 
random element $\Lambda$ taking values in a metric space $(\mathbb{X},d)$. A 
simulation algorithm that for any $\varepsilon>0$ constructs in finitely many steps 
a random element $\Lambda^\varepsilon$ in $\mathbb{X}$ satisfying (I) and (II) 
below is termed an $\varepsilon$SS algorithm: (I) there exists a coupling 
$(\Lambda,\Lambda^\varepsilon)$ on a probability space $\Omega$ such that the 
%$L^\infty$-Wassserstein distance (i.e. the 
essential supremum $\mathrm{ess~sup}
\{d(\Lambda(\omega),\Lambda^\varepsilon(\omega)):\omega\in\Omega\}$ is at 
most  $\varepsilon$; (II) for any $m\in\N$, decreasing sequence 
$\varepsilon_1>\dots>\varepsilon_m>0$, random elements 
$\Lambda^{\varepsilon_1},\ldots,\Lambda^{\varepsilon_{m}}$ (satisfying (I) for the 
respective $\varepsilon_1,\ldots,\varepsilon_m$) and $\varepsilon'\in
(0,\varepsilon_m)$, we can sample $\Lambda^{\varepsilon'}$, given 
$\Lambda^{\varepsilon_1},\ldots,\Lambda^{\varepsilon_{m}}$, which satisfies (I) for 
$\varepsilon'$. Condition~(II), known as the tolerance-enforcement property of 
$\varepsilon$SS, can be seen as a measurement of the realisation of the random 
element $\Lambda$ whose error may be reduced in exchange for additional 
computational effort.

Throughout this paper, the metric $d$ in the definition above 
%our $\varepsilon$SS algorithms are constructed with respect to the 
is given by the supremum norm on either the space of continuous functions on a 
compact interval or on a finite dimensional Euclidean space.   
%That is, the metrics with respect to which our $\varepsilon$SS algorithms are constructed are $((x_1,\ldots,x_d),(y_1,\ldots,y_d))
%\mapsto\max_{i\in\{1,\ldots,d\}}|x_i-y_i|$ in $\R^d$ and $(f,g)\mapsto 
%\sup_{x\in[a,b]}|f(x)-g(x)|$ in the space of continuous functions on $[a,b]$.
The remainder of this section is structured as follows. 
Section~\ref{subsec:e-SS_Levy} reduces the problems of constructing 
$\varepsilon$SS algorithms for the finite dimensional distributions of L\'evy meanders 
and the convex minorants of L\'evy processes, to constructing an $\varepsilon$SS 
algorithm of the convex minorants of  L\'evy meanders. In 
Subsection~\ref{subsec:algs-stable} we apply 
Theorem~\ref{thm:meander-minorant} of Section~\ref{sec:ConcaveMajorant} to 
construct an $\varepsilon$SS algorithm for the convex minorant of a L\'evy meander 
under certain technical conditions.
%by sampling finitely many of its faces and 
%bounding the sum of the heights of the remaining (infinitely many) faces. 
In Theorem~\ref{thm:CMSM} we state a stochastic perpetuity 
equation~\eqref{eq:SMPerEq}, established in Section~\ref{sec:proofs} using 
Theorem~\ref{thm:meander-minorant}, that implies these technical conditions in the 
case of stable meanders. Subsection~\ref{subsec:algs-stable} concludes with the 
statement of Theorem~\ref{thm:all_complexities} describing the computational 
complexity of the  $\varepsilon$SS algorithm constructed in 
Subsection~\ref{subsec:algs-stable}.

\subsection{$\varepsilon$SS of the convex minorants of L\'evy  processes\label{subsec:e-SS_Levy}}
In the present subsection we construct $\varepsilon$SS algorithms for the convex 
minorant $C(X)$ and for the finite dimensional distributions of $X^{\me,T}$. 
Both algorithms require the following assumption. 

\begin{assumption*}[S]\label{asm:(S)} 
There is an $\varepsilon$SS algorithm for $C((\pm X)^{\me,t})$ for any $t>0$.
\end{assumption*}

In the case of stable processes, an algorithm satisfying Assumption~(\nameref{asm:(S)}) 
is given in the next subsection. In this subsection we assume that~(\nameref{asm:(P)}) 
and~(\nameref{asm:(S)}) hold for the process $T_cX$ for some $c\in\R$, where $T_c$ 
denotes the \emph{linear tilting} functional $T_c: f\mapsto (t\mapsto f(t)+ct)$ for any real 
function $f$. We construct an $\varepsilon$SS algorithm for the convex minorant $C(X)$, 
and hence for $(X_T,\un{X}_T)$, as follows ($\mL(\cdot)$ denotes the law of the random 
element in its argument).

\begin{algorithm}
\caption{$\varepsilon$SS of the convex minorant $C(X)$ of a L\'evy process $X$, such 
that $T_{c}X$ satisfies Assumptions~(\nameref{asm:(P)}) and~(\nameref{asm:(S)}) for 
some $c\in\R$.}\label{alg:eps-CM}
\begin{algorithmic}[1]
	\Require{Time horizon $T>0$, accuracy $\varepsilon>0$ and $c\in\R$.}
	\State{Sample $\beta\sim B(1-\rho,\rho)$ and put $s\la T\beta$}
	\State{Sample $\varepsilon/2$-strongly $f^\la$ from $\mL(C((-T_{c}X)^{\me,s}))$}
	\Comment{Assumption~(\nameref{asm:(S)})}
	\State{Sample $\varepsilon/2$-strongly $f^\ra$ from $\mL(C(T_{c}X^{\me,T-s}))$}
	\Comment{Assumption~(\nameref{asm:(S)})}
	\State{\Return $f_\varepsilon:t\mapsto -ct + f^\la(s-\min\{t,s\})-
		f^\la(s)+f^\ra(\max\{t,s\}-s)$ for $t\in[0,T]$.}
\end{algorithmic}
\end{algorithm}
\begin{rem}\label{rem:eps-CM}
(i) Note that $(f_\varepsilon(T),\un{f_\varepsilon}(T))$ is an $\varepsilon$SS of 
$(X_T,\un{X}_T)$ as $f\mapsto \un{f}(T)$ is a Lipschitz functional 
on the space of \cadlag~funcitons with respect to the supremum norm. Although 
$\tau_{[0,T]}(f_\varepsilon)=\inf\{t\in[0,T]:f_\varepsilon(t)=\un{f_\varepsilon}(T)\}\to
\tau_{[0,T]}(X)$ as $\varepsilon\downarrow0$ by~\citep[Lem.~14.12]{MR1876169}, 
\emph{a priori} control on the error does not follow directly in general. 
%	(see, 
%	e.g.~\citep[p.~138]{MR1700749}). Indeed, let $d_\D$ be the metric of $\D[0,1]$ 
%	defined in~\citep[p.~113]{MR1700749}. Then for any $\epsilon>0$ and any $Z'
%	\in\D[0,1]$ such that $d_\D(Z,Z')<\epsilon$ there exists a continuous and 
%	strictly increasing $\lambda:[0,1]\mapsto[0,1]$ such that $\lambda(0)=0$, 
%	$\lambda(1)=1$ and $\sup_{t\in[0,1]}|Z_t-Z'_{\lambda(t)}|<\epsilon$. Hence 
%	$|\sup_{t\in[0,1]}Z_t-\sup_{t\in[0,1]}Z_t'|<\epsilon$, proving the 
%	continuity of $x\mapsto\sup_{t\in[0,1]}x_t$ in $\D[0,1]$. Next, since the time of 
%	the maximum of a nonzero stable process $(\max\{Z_t,Z_{t-}\})_{t\in[0,1]}$ is 
%	a.s. unique (see e.g.~\citep[Prop.~VIII.4.15]{MR1406564}), then $\tau_1$ is a.s. 
%	continuous with respect to the law of $Y$~\citep[Lem.~14.12]{MR1876169}.
In the case of weakly stable processes, we will construct in steps~2 and~3 of
Algorithm~\ref{alg:eps-CM}, piecewise linear convex functions that sandwich 
$C(X)$, which yield a control on the error of the approximation of $\tau_T$ 
by Proposition~\ref{prop:General_UpLo}(c).\\
(ii) The algorithm may be used to obtain an $\varepsilon$SS of $-C(-X)$, 
the concave majorant of $X$.
\end{rem}
Fix $0=t_0<t_1<\ldots<t_m\leq t_{m+1}=T$ and recall from 
Subsection~\ref{subsec:Levy-meanders} above that $X^{\me,T}$ follows the law of 
$(X)_{t\in[0,T]}$ conditional on $\{\un{X}_T\geq 0\}$. Note that $\P(\un{X}_T\geq0)=0$, 
but $\P(\un{X}_T\geq0|\un{X}_{t_1}\geq0)>0$ by Assumption~(\nameref{asm:(D)}). Thus, 
sampling $(X^{\me,T}_{t_1},\ldots,X^{\me,T}_{t_m})$ is reduced to jointly simulating 
$(X_{t_1},\ldots,X_{t_m})$ and $\un{X}_T$ conditional on $\{\un{X}_{t_1}\geq 0\}$ and 
rejecting all samples not in the event $\{\un{X}_T\geq0\}$. More precisely, we get the 
following algorithm.

\begin{algorithm}
\caption{$\varepsilon$-strong simulation of the vector 
$(X^{\me,T}_{t_1},\ldots,X^{\me,T}_{t_m})$.}\label{alg:fd_meander}
\begin{algorithmic}[1]
		\Require{Times $0=t_0<t_1<\ldots<t_m\leq t_{m+1}=T$ and accuracy $\varepsilon>0$.} 
		\Repeat
			\State{Put $(\Pi_0,\varepsilon_0,i)\la (\emptyset,2\varepsilon/(m+1),0)$}
			\Repeat{~Conditionally on the variables in the set $\Pi_i$}
				\State{Put $(\varepsilon_{i+1},i)\la (\varepsilon_i/2,i+1)$}
				\State{Sample $\varepsilon_i$-strongly $z_1^{\varepsilon_i}$ from $\mL(X_{t_1}^{\me,t_1})$ and put $(x_1^{\varepsilon_i},\un{x}_1^{\varepsilon_i})\la (z_1^{\varepsilon_i},z_1^{\varepsilon_i})$}
				\Comment{Assumption~(\nameref{asm:(S)})}
				%\State{}
				\For{$k=2,\ldots,m+1$}
					\State{Sample ${\varepsilon_i}$-strongly $(z_k^{\varepsilon_i},\un{z}_k^{\varepsilon_i})$ from
			$\L(X_{t_k-t_{k-1}},\un{X}_{t_k-t_{k-1}})$}
					\Comment{Remark~\ref{rem:eps-CM}(i)}
					\State{Put $(x_k^{\varepsilon_i},\un{x}_k^{\varepsilon_i}) \la(x_{k-1}^{\varepsilon_i}+z_k^{\varepsilon_i},
					\min\{\un{x}_{k-1}^{\varepsilon_i},x_{k-1}^{\varepsilon_i}+
					\un{z}_k^{\varepsilon_i}\})$}
				\EndFor
				\State{Put $\Pi_{i}\la\Pi_{i-1}\cup\{(z_k^{\varepsilon_i},
					\un{z}_k^{\varepsilon_i})\}_{k=1}^{m+1}$}
				\Until{$\un{x}_{m+1}^{\varepsilon_i}-(m+1)\varepsilon_i\geq0$ or $\un{x}_{m+1}^{\varepsilon_i}+(m+1)\varepsilon_i<0$}
		\Until{$\un{x}_{m+1}^{\varepsilon_i}-(m+1)\varepsilon_i\geq0$}
		\State{\Return $(x_1^{\varepsilon_i},\ldots,x_m^{\varepsilon_i})$.}
	\end{algorithmic}
\end{algorithm}
\begin{rem}\label{rem:alg_1}
(i) All the simulated values are dropped when the condition in 
line~13 fails.\\
(ii) If the algorithm satisfying Assumption~(\nameref{asm:(S)}) is the result of a 
sequential procedure, % indexed by $i\in\N$, 
one may remove the explicit reference to $\varepsilon_i$ in line~4 and instead run all 
pertinent algorithms for another step until condition in line~12 holds. This is, 
for instance, the case for the algorithms we present for stable meanders.\\
%(iii) The running time of Algorithm~\ref{alg:fd_meander} depends directly on how long 
%it takes until either event in line~12 takes place. In the stable case, this algorithm has 
%finite expected running time by (I) Proposition~\ref{prop:MainRuntime} and (II) 
%Proposition~\ref{prop:hit_time_tail}, respectively.%\\
%(iv) It is possible to do a sequential rejection sampling whereby each increment is 
%validated every time it is sampled, so that only a single increment (and not the entire 
%path) need be thrown away upon rejection. However, this would rely on 
%say~\citep[Alg.~1]{ExactSim} in the stable case and it is currently unknown if its 
%expected runtime is finite or not. On the other hand, the running time of 
%Algorithm~\ref{alg:fd_meander} has exponential moments by
%Theorem~\ref{thm:all_complexities}.
\end{rem}

\subsection{Simulation of the convex minorant of stable meanders\label{subsec:algs-stable}}
In the remainder of the paper, we let $Z=(Z_t)_{t\in[0,T]}$ be a stable process with 
stability parameter $\alpha\in(0,2]$ and positivity parameter $\P(Z_1>0)=\rho\in(0,1]$, 
using Zolotarev's (C) form (see e.g.~\cite[App.~A]{ExactSim}). 
It follows from~\cite[Eq.~(A.1)\&(A.2)]{ExactSim} that 
Assumptions~(\nameref{asm:(K)}) and~(\nameref{asm:(P)}) are satisfied by $Z$. In 
the present subsection, we will construct an $\varepsilon$SS algorithm for the 
convex minorant of stable meanders, required by 
Assumption~(\nameref{asm:(S)}) of Subsection~\ref{subsec:e-SS_Levy}.

The scaling property implies that $(Z^{\me,T}_{sT})_{s\in[0,1]}
\overset{d}{=}(T^{1/\alpha}Z^{\me,1}_s)_{s\in[0,1]}$ and thus 
\[(C(Z^{\me,T})(sT))_{s\in[0,1]}\overset{d}{=}
(C(T^{1/\alpha}Z^{\me,1})(s))_{s\in[0,1]}=
(T^{1/\alpha}C(Z^{\me,1})(s))_{s\in[0,1]}.\]
By the relation in display, it is sufficient to consider the case of the normalised 
stable meander $Z^\me=Z^{\me,1}$ in the remainder of the paper.

%Consider two continuous functions $f,f_\varepsilon:[0,1]\to\R$ such that their 
%distance in the supremum norm is at most $\epsilon$. Clearly, we may sandwich $f$ 
%by defining $f_\epsilon-\epsilon$ and $f_\epsilon+\epsilon$. The same is clearly 
%true of the outputs of Algorithm~\ref{alg:eps-CM} and the algorithm satisfying
%Assumption~\ref{asm:(S)}. However, there may be more natural ways of doing this. 
%Algorithm~\ref{alg:eps-CM-SM} below is such an example. Hence, lines~2 and~3 in 
%Algorithm~\ref{alg:eps-CM} will, when applied to $Z$, have two candidates each. 
%For this reason, we will later explain how to appropriately combine them to 
%sandwich the convex minorant $C(Z)$ in a more natural way that what we described 
%at the start of the paragraph. In particular, both sandwiching functions will be 
%$\varepsilon$-strong samples of $C(Z)$. 

\subsubsection{Sandwiching\label{subsec:sandwich}}

To obtain an $\varepsilon$SS of the convex minorant of a meander, we will construct 
two convex and piecewise linear functions with finitely many faces that sandwich the 
convex minorant and whose distance from each other, in the supremum norm, is at 
most $\varepsilon$. Intuitively, the sandwiching procedure relies on two ingredients: 
(I) the ability to sample, for each $n$, the first $n$ faces in the minorant and (II) doing 
so jointly with a variable $c_n>0$ that dominates the sum of the heights of all the 
unsampled faces. Conditions (I) and (II) are, by Proposition~\ref{prop:Sandwich-CM} 
below, sufficient to sandwich the convex minorant: lower (resp. upper) bound 
$C(Z^\me)_n^\da$ (resp. $C(Z^\me)_n^\ua$) is constructed by adding a final face of 
height $0$ (resp. $c_n$) and length equal to the sum of the lengths of the remaining 
faces and sorting all $n+1$ faces in increasing order of slopes. The distance (in the 
supremum norm) between the convex functions $C(Z^\me)_n^\da$ and 
$C(Z^\me)_n^\ua$ equals $c_n$ (see  Proposition~\ref{prop:Sandwich-CM} for details). 
The $\varepsilon$SS algorithm is then obtained by stopping at $-N(\varepsilon)$, 
the smallest integer $n$ for which the error $c_n$ is smaller than $\varepsilon$ (see 
Algorithm~\ref{alg:eps-CM-SM} below).

In general, condition~(I) is relatively easy to satisfy under the assumptions of 
Theorem~\ref{thm:meander-minorant}. Condition~(II) however, is more challenging. 
In the stable case, we first establish a stochastic perpetuity in 
Theorem~\ref{thm:CMSM} and use ideas from~\citep{ExactSim} to sample the 
variables  $c_n$, $n\in\N$, in condition~(II) (see Equation~\eqref{eq:c_n}).

Figure~\ref{fig:eSS-CM}(a) below illustrates the output of the $\varepsilon$SS 
Algorithm~\ref{alg:eps-CM-SM} below for the convex minorant $C(Z^\me)$. By gluing 
two such outputs for the (unnormalised) stable meanders $Z^\la$ and $Z^\ra$, 
straddling the minimum of $Z$ over the interval $[0,1]$ as in~\eqref{def:pre_post}, 
with $n$ and $m$ faces, respectively, we obtain a convex function $C(Z)_{n,m}^\da$ 
(resp. $C(Z)_{n,m}^\ua$) that is smaller (resp. larger) than the convex minorant $C(Z)$ 
of the stable process (see details in Proposition~\ref{prop:General_UpLo}). 
Figure~\ref{fig:eSS-CM}(b) illustrates how these approximations sandwich the convex 
minorant $C(Z)$.

\begin{figure}[H]
	\centering{}\resizebox{.49\linewidth}{!}{
		\begin{subfigure}[l]{0.55\textwidth}
			\begin{tikzpicture} 
			\begin{axis} 
			[
			title={$C(Z^\me)_n^\da$ and $C(Z^\me)_n^\ua$ for $(\alpha,\rho)=(1.3,0.69)$},
			ymin=-.1,
			ymax=5.35,
			xmin=-.05,
			xmax=1.05,
			xlabel=$t$,
			width=9cm,
			height=4.5cm,
			axis on top=true,
			ytick=\empty,
			xtick=\empty,
			xticklabels=none,
			yticklabels=none,
			axis x line=middle, 
			axis y line=none,
			legend pos=north west ] 
			
			% Draw the CM with 4 faces
			\addplot[
			dotted,
			color=violet,
			mark=square*,
			mark size=1pt
			]
			coordinates {
				(0,0)(.174,.120)(.497,1.20)(.954,4.64)(1,5.31)
			};
			\addplot[
			dotted,
			color=violet,
			mark=square*,
			mark size=1pt
			]
			coordinates {
				(0,0)(.457,0)(.630,.120)(.954,1.20)(1,1.86)
			};
			
			% Draw the CM with 6 faces
			\addplot[
			dashdotted,
			color=orange,
			mark=square*,
			mark size=1pt
			]
			coordinates {
				(0,0)(.174,.120)(.325,.591)(.648,1.67)(.894,3.49)(.954,4.01)(1,4.68)
			};
			\addplot[
			dashdotted,
			color=orange,
			mark=square*,
			mark size=1pt
			]
			coordinates {
				(0,0)(.246,0)(.420,.120)(.571,.591)(.894,1.67)(.954,2.19)(1,2.86)
			};
			
			% Draw the CM with 9 faces
			\addplot[
			dashed,
			color=olive,
			mark=square*,
			mark size=1pt
			]
			coordinates {
				(0,0)(.174,.120)(.408,.678)(.559,1.15)(.882,2.22)(.887,2.26)(.947,2.79)(.993,3.45)(.994,3.48)(1,4.39)
			};
			\addplot[
			dashed,
			color=olive,
			mark=square*,
			mark size=1pt
			]
			coordinates {
				(0,0)(.234,0)(.408,.120)(.559,.591)(.882,1.67)(.887,1.70)(.947,2.23)(.993,2.89)(.994,2.93)(1,3.83)
			};
			
			% Draw the CM with 17 faces
			\addplot[
			solid,
			color=black,
			mark=square*,
			mark size=1pt
			]
			coordinates {
				(0,0)(.174,.120)(.318,.319)(.331,.338)(.337,.349)(.387,.484)(.389,.491)(.541,.963)(.864,2.04)(.871,2.07)(.882,2.14)(.887,2.18)(.947,2.70)(.993,3.37)(.994,3.40)(1,4.31)
			};
			% Legend 
			\legend {{\small $n=4$}, ,{\small $n=6$}, ,
				{\small $n=9$}, ,{\small $n=17$}};
			\end{axis}
			\end{tikzpicture}\subcaption{Sandwiching $C(Z^\me)$}
	\end{subfigure}}
	\centering{}\resizebox{.49\linewidth}{!}{
		\begin{subfigure}[l]{0.55\textwidth}
			\begin{tikzpicture} 
			\begin{axis} 
			[
			title={$C(Z)_{n,m}^\da$ and $C(Z)_{n,m}^\ua$ for $(\alpha,\rho)=(1.8,0.52)$},
			ymin=-4.2,
			ymax=.9,
			xmin=-.05,
			xmax=1.05,
			xlabel=$t$,
			width=9cm,
			height=4.5cm,
			axis on top=true,
			ytick=\empty,
			xtick=\empty,
			xticklabels=none,
			yticklabels=none,
			axis x line=middle, 
			axis y line=none,
			legend style={at={(.55,.5)},anchor=south east}] 
			
			% Draw the CM with 8 faces
			\addplot[
			dotted,
			color=violet,
			mark=square*,
			mark size=1pt
			]
			coordinates {
				(0.0,0.0)(0.041,-0.823)(0.042,-0.842)(0.063,-1.206)(0.15,-2.256)(0.179,-2.516)(0.192,-2.597)(0.23,-2.8)(0.261,-2.8)(0.36,-2.7)(0.592,-2.004)(0.632,-1.864)(0.878,-0.681)(0.911,-0.406)(0.937,-0.073)(0.948,0.076)(1.0,0.785)
			};
			\addplot[
			dotted,
			color=violet,
			mark=square*,
			mark size=1pt
			]
			coordinates {
				(0.0,-1.29)(0.041,-2.113)(0.042,-2.132)(0.063,-2.496)(0.15,-3.546)(0.179,-3.805)(0.192,-3.887)(0.23,-4.09)(0.261,-4.09)(0.313,-4.09)(0.411,-3.99)(0.644,-3.294)(0.684,-3.153)(0.93,-1.97)(0.962,-1.695)(0.989,-1.365)(1.0,-1.216)
			};
			
			% Draw the CM with 20 faces
			\addplot[
			solid,
			color=black,
			mark=square*,
			mark size=1pt
			]
			coordinates {
				(0.0,0.0)(0.0,-0.001)(0.0,-0.005)(0.0,-0.006)(0.0,-0.006)(0.0,-0.007)(0.0,-0.009)(0.0,-0.012)(0.002,-0.112)(0.005,-0.221)(0.005,-0.224)(0.046,-1.047)(0.047,-1.066)(0.062,-1.366)(0.083,-1.73)(0.17,-2.78)(0.199,-3.039)(0.21,-3.117)(0.223,-3.199)(0.261,-3.402)(0.261,-3.402)(0.36,-3.302)(0.592,-2.606)(0.632,-2.465)(0.878,-1.282)(0.915,-1.023)(0.948,-0.748)(0.974,-0.418)(0.986,-0.269)(0.991,-0.177)(0.991,-0.172)(0.999,0.028)(0.999,0.098)(1.0,0.12)(1.0,0.128)(1.0,0.144)(1.0,0.158)(1.0,0.161)(1.0,0.164)(1.0,0.167)(1.0,0.168)
			};
			% Legend 
			\legend {{\small $n=m=8$}, , {\small $n=m=20$}};
			\end{axis}
			\end{tikzpicture}\subcaption{Sandwiching $C(Z)$}
	\end{subfigure}}
	\caption{\footnotesize (A) Sandwiching of the convex 
	minorant $C(Z^\me)$ %of a stable meander $Z^\me$ 
	using $n$ faces. The lower and upper bounds are numerically indistinguishable 
	for $n=17$. (B) Sandwiching of the convex minorant $C(Z)$ %of the stable process $Z$ 
	using $n$ and $m$ faces of the convex minorants $C(Z^\la)$ and  
	$C(Z^\ra)$ of the meanders $Z^\la$ and $Z^\ra$, respectively. 
	Again the bounds are numerically indistinguishable for $n=m=20$.}\label{fig:eSS-CM}
\end{figure}
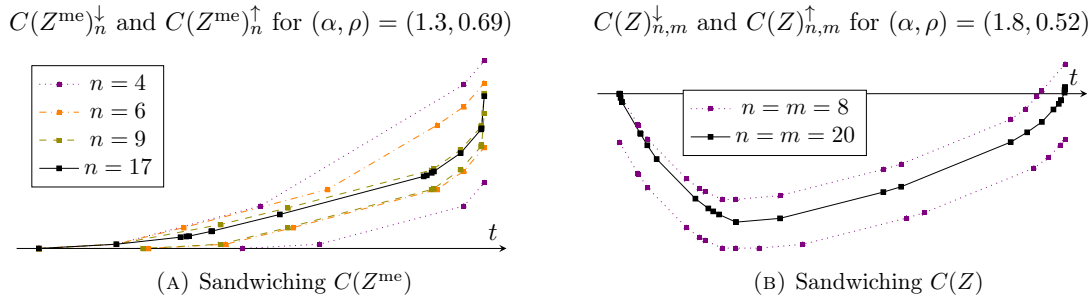

A linear tilting can be applied, as in Algorithm~\ref{alg:eps-CM} above, to obtain
a sandwich for the convex minorant of a weakly stable processes for all $\alpha
\in(0,2]\setminus\{1\}$ (see a numerical example in 
Subsection~\ref{subsec:num-hit} below).

\subsubsection{The construction of $c_n$}
Since stable processes satisfy Assumptions~(\nameref{asm:(P)}) 
and~(\nameref{asm:(K)}), we may use Theorem~\ref{thm:meander-minorant} and 
the scaling property of stable laws as stepping stones to obtain a Markovian 
description of the convex minorants of the corresponding meanders. Let 
$\mS(\alpha,\rho)$, $\mS^+(\alpha,\rho)$ and $\mS^\me(\alpha,\rho)$ be the laws 
of $Z_1$, $Z_1$ conditioned to be positive and $Z_1^\me$, respectively, where 
$(Z_t)_{t\in[0,1]}$ is a stable process with parameters $(\alpha,\rho)$. Recall the 
definition of the sets $\mZ^n=\{k\in\mathbb{Z}:k<n\}$ and 
$\mZ^n_m=\mZ^n\setminus\mZ^m$ for $n,m\in\mathbb{Z}$.

\begin{thm}\label{thm:CMSM}
Let $((\ell_n^\me,\xi_n^\me))_{n\in\N}$ be the faces of $C(Z^\me)$ chosen 
independently at random uniformly on lengths. Define the random variables 
$L_{n+1}=\sum_{m\in\mZ^{n+1}}\ell_{-m}^\me$, $U_n=\ell_{-n}^\me/L_{n+1}$,
\begin{equation}\label{eq:MC}
S_n =(\ell_{1-n}^{\me})^{-1/\alpha}\xi^\me_{1-n}\qquad\text{and}\qquad
M_{n+1}=L_{n+1}^{-1/\alpha}\sum_{m\in\mZ^{n+1}}\xi^\me_{-m},
\end{equation}
for all $n\in\mZ^0$. Then the following statements hold.\\
{\normalfont(a)} $((S_n,U_n))_{n\in\mZ^0}$ is an iid sequence with common law 
$\mS^+(\alpha,\rho)\times \Beta(1,\rho)$.\\
{\normalfont(b)} $(M_n)_{n\in\mZ^1}$ is a stationary Markov chain satisfying 
$M_0=Z^\me_1\sim\mS^\me(\alpha,\rho)$ and
\begin{equation}\label{eq:SMPerEq}
M_{n+1}=(1-U_n)^{1/\alpha}M_n+U_n^{1/\alpha}S_n,\qquad\text{for all }n\in\mZ^0.
\end{equation}
{\normalfont(c)} The law of $Z_1^\me$ is the unique solution to the perpetuity 
$Z_1^\me\overset{d}{=}(1-U)^{1/\alpha}Z_1^\me+U^{1/\alpha}S$ for independent 
$(S,U)\sim \mS^+(\alpha,\rho)\times \Beta(1,\rho)$.
\end{thm}

Theorem~\ref{thm:CMSM}, proved in Subsection~\ref{proofThmCMSM} below, 
enables us to construct a process $(D_n)_{n\in\mZ^1}$ that dominates 
$(M_n)_{n\in\mZ^1}$: $D_n\geq M_n$ for $n\in\mZ^1$ and can be simulated jointly 
with the sequence $((S_n,U_n))_{n\in\mZ^0}$ (see details in 
Appendix~\ref{sec:notationSupSim}). Thus, by~\eqref{eq:MC}, we may construct the 
sandwiching convex functions in Algorithm~\ref{alg:eps-CM-SM} below (see also 
Subsection~\ref{subsec:sandwich} above for an intuitive description) by setting
\begin{equation}\label{eq:c_n}
c_{-n}=L_n^{1/\alpha}D_n\geq\sum_{m\in\mZ^n}\xi_{-m}^\me,\qquad n\in\mZ^0.
\end{equation}

\subsubsection{The algorithm and its running time\label{subsec:runtime-stable}}
Let $-N(\varepsilon)$ be the smallest $n\in\N$ such that $c_{-n}<\varepsilon$ 
(see~\eqref{eq:N_epsilon} below for the precise definition).  

\begin{algorithm}
\caption{$\varepsilon$-strong simulation of the convex minorant $C(Z^\me)$.}
\label{alg:eps-CM-SM}
	\begin{algorithmic}[1]
		\Require{Tolerance $\varepsilon>0$ and burn-in parameter $m\in\N\cup\{0\}$}
		\State{Sample independently $(S_k,U_k)$  for $k\in\mZ_{-m}^0$ from 
			the law  $\mS^+(\alpha,\rho)\times \Beta(1,\rho)$}
		\State{Sample backwards in time $(S_k,U_k,D_k)$ for $k\in \mZ_{-\max\{m+1,|N(\varepsilon)|\}}^{-m}$}
		\Comment{\citep[Alg.~2]{ExactSim}}
		\State{Set	$n=\max\{m+1,|N(\varepsilon)|\}$}	
		\State{For $k\in\{1,\ldots,n\}$, set $\ell^\me_k=U_{-k}\prod_{j\in\mZ_{1-k}^0}(1-U_j)$ 
			and $\xi_k^\me=(\ell^\me_k)^{1/\alpha}S_{1-k}$}
		\State{Put $c_n=(\prod_{k\in\mZ_{-n}^0}(1-U_k))^{1/\alpha}D_{-n}$ and 
			\Return $(C(Z^\me)_n^\da,C(Z^\me)_n^\ua)$ }
		%\Comment{Defined in Proposition~\ref{prop:Sandwich-CM}}
	\end{algorithmic}
\end{algorithm}
\begin{rem}\label{rem:eps-CM-SM}
(i) Given the faces $\{(\ell_k^\me,\xi_k^\me):k\in\mZ_1^{n+1}\}$, the output 
$(C(Z^\me)_n^\da,C(Z^\me)_n^\ua)$ of Algorithm~\ref{alg:eps-CM-SM} is defined in 
Proposition~\ref{prop:Sandwich-CM}, see also Lemma~\ref{lem:minimal-faces}. 
In particular, Proposition~\ref{prop:Sandwich-CM} requires to sort $(n+1)$ faces, sampled 
in  Algorithm~\ref{alg:eps-CM-SM}. This has a complexity of at most $\O(n\log(n))$ under 
the \emph{Timsort} algorithm, making the complexity of Algorithm~\ref{alg:eps-CM-SM} 
proportional to $|N(\varepsilon)|\log|N(\varepsilon)|$. Moreover, the burn-in parameter $m$ 
is conceptually inessential (i.e. we can take it to be equal to zero without affecting the law 
of the output) but practically very useful. Indeed, since Algorithm~\ref{alg:eps-CM-SM} 
terminates as soon as $c_n<\varepsilon$,  the inexpensive simulation of the pairs 
$(S_k,U_k)$ increases the probability of having to sample fewer (computationally expensive) 
triplets $(S_k,U_k,D_k)$ in line~2 (cf.~\citep[Sec.~5]{ExactSim}).\\
(ii) An alternative to Algorithm~\ref{alg:eps-CM-SM} is to run forward in time a 
Markov chain based on the perpetuity in Theorem~\ref{thm:CMSM}(c). This would 
converge in the $L^\gamma$-Wasserstein distance at the rate 
$\O((1+\gamma/(\alpha\rho))^{-n})$ (see~\citep[Sec.~2.2.5]{MR3497380}), 
yielding an approximate simulation algorithm for the law $S^\me(\alpha,\rho)$.
\end{rem}

Note that the running time of Algorithm~\ref{alg:eps-CM-SM} is completely determined by 
$\max\{m+1,|N(\varepsilon)|\}$. Applications in Section~\ref{subsec:e-app} below rely on 
the exact simulation of $\1_{\{Z_1^\me>x\}}$ for arbitrary $x>0$ via $\varepsilon$SS (see 
Subsection~\ref{subsec:indicators} below), which is based the on sequential refinements of 
Algorithm~\ref{alg:eps-CM-SM}. Moreover, Algorithms~\ref{alg:eps-CM} 
and~\ref{alg:fd_meander} rely on Algorithm~\ref{alg:eps-CM-SM}. Hence bounding the tails 
of $N(\varepsilon)$ is key in bounding the running times of all those algorithms. 
Proposition~\ref{prop:MainRuntime} establishes bounds on the tails of $N(\varepsilon)$, 
which combined with Lemma~\ref{lem:indicators_time} below, implies the following result, 
proved in Subsection~\ref{proofThmComplexities} below.

\begin{thm}\label{thm:all_complexities}
The running times of Algorithms~\ref{alg:eps-CM},~\ref{alg:fd_meander} 
and~\ref{alg:eps-CM-SM} for the $\varepsilon$SS of $C(T_cZ)$ (for any $c\in\R$), finite 
dimensional distributions of $Z^\me$ and $C(Z^\me)$, respectively, have exponential 
moments. The same holds for the exact simulation algorithm of 
$\1_{\{Z_1^\me>x\}}$ for any $x>0$.
\end{thm}

%The proof, given in Subsection~\ref{subsec:complexity} below, relies on 
%Lemma~\ref{lem:indicators_time}, which gives a general bound on the running time for 
%detecting certain events in terms of the function $\varepsilon\mapsto N(\varepsilon)$.

\section{Applications and numerical examples\label{sec:e-app_and_num}}
In Subsection~\ref{subsec:e-app} we describe applications of $\varepsilon$SS 
paradigm to the exact and unbiased sampling of certain functionals. In 
Subsection~\ref{subsec:numerics} we present specific numerical examples. We apply 
algorithms from Section~\ref{sec:algorithms} to estimate expectations via MC 
methods and construct natural confidence intervals.

\subsection{Applications of $\varepsilon$SS algorithms\label{subsec:e-app}}
Consider a metric space $(\mathbb{X},d)$ and a random element 
$\Lambda$ taking values in $\mathbb{X}$. For every $\varepsilon>0$, the random 
element $\Lambda^\varepsilon$ is assumed to be the output of an $\varepsilon$SS
algorithm, i.e. it satisfies $d(\Lambda,\Lambda^\varepsilon)<\varepsilon$ a.s. (see Section~\ref{sec:algorithms} for definition). 

\subsubsection{Exact simulation of indicators\label{subsec:indicators}}
One may use $\varepsilon$SS algorithms to sample exactly indicators 
$\1_A(\Lambda)$ for any set $A\subset\mathbb{X}$ such that 
$\P(\Lambda\in\partial A)=0$, where $\partial A$ denotes the boundary of $A$ in 
$\mathbb{X}$. Since $d(\Lambda^\varepsilon,\partial A)>\varepsilon$ implies 
$\1_A(\Lambda)=\1_A(\Lambda^\varepsilon)$, it suffices to sample say the sequence 
$(\Lambda^{2^{-n}})_{n\in\N}$ until $d(\Lambda^{2^{-n}},\partial A)>2^{-n}$. Finite 
termination is ensured since $\{d(\Lambda,\partial A)>0\}=
\bigcup_{n\in\N}\{d(\Lambda,\partial A)>2^{-n}\}\subset
\bigcup_{n\in\N}\{d(\Lambda^{2^{-n-1}},\partial A)>2^{-n-1}\}$. 
%The event detection 
In particular, line~12 in Algorithm~\ref{alg:fd_meander} is based on this principle.

\subsubsection{Unbiased simulation of finite variation transformations of a continuous functional\label{subsec:continuous}}
Let $f_1:\mathbb{X}\to\R$ be continuous, $f_2:\R\to[0,\infty)$ be of finite variation on 
compact intervals and define $f=f_2\circ f_1$. The functional 
$\un{Z}_T\cdot\1_{\{\un{Z}_T>b\}}$, for some $b<0$, is a concrete example defined on the 
space of continuous functions, since the maps $C(Z)\mapsto\un{Z}_T$ and 
$\un{Z}_T\mapsto\un{Z}_T\cdot\1_{\{\un{Z}_T>b\}}$ are continuous and of finite variation, 
respectively.

By linearity, it suffices to consider a monotone $f_2:\R\to[0,\infty)$.  Let $\varsigma$ be 
an independent random variable with positive density $g:[0,\infty)\to(0,\infty)$. Then 
$\Sigma=\1_{(\varsigma,\infty)}(f(\Lambda))/g(\varsigma)$ is simulatable and unbiased for 
$\E[f(\Lambda)]$. 
%(I) Let $\Lambda_1=f_1(\Lambda)$. Since $f_2$ is monotone, the inverse images of 
%singletons are (possibly degenerate) intervals and thus no three reals 
%$x_1,x_2,x_3\in\R$ have a common point in the boundary of their inverse images, 
%i.e. $\bigcap_{i=1}^3 \partial f_2^{-1}(\{x_i\})=\emptyset$. Hence, every $x\in\R$ 
%is in the boundary of the inverse images of at most two singletons. Let $\ov B$ 
%denote the closure of $B\subset\R$ in $\R$ and let $\mathcal{A}_H$ be the (at most 
%countable) set of atoms of the random variable $H$. Then the cardinality of the 
%set of reals $A=\{x\in\R:\P\big(\Lambda_1\in\ov{f_2^{-1}(x)}\big)>0\}$ is bounded 
%by $2\#f_2(\mathcal{A}_{\Lambda_1})+\#\mathcal{A}_{f_2(\Lambda_1)}$ and is thus 
%at most countable. This implies that $\P(\varsigma\in A)=0$ and hence 
%$\P\Big(\Lambda_1\in \ov{f_2^{-1}(\{\varsigma\})}\Big)=0$. By the continuity of
%$f_1$, we have 
%\[\partial f^{-1}(\{\varsigma\})\subset\ov{f^{-1}(\{\varsigma\})}
%=\ov{f_1^{-1}(f_2^{-1}(\{\varsigma\}))}
%\subset \ov{f_1^{-1}\Big(\ov{f_2^{-1}(\{\varsigma\})}\Big)}
%=f_1^{-1}\Big(\ov{f_2^{-1}(\{\varsigma\})}\Big).\] 
Indeed, it is easily seen that $\P(\Lambda\in\partial f^{-1}(\{\varsigma\}))=0$. 
Thus Subsection~\ref{subsec:indicators} shows that the indicator
$\1_{(\varsigma,\infty)}(f(\Lambda))$, and hence $\Sigma$, may be simulated exactly. 
Moreover, $\Sigma$ is unbiased since 
\[\E[\Sigma]=\E\left[\E[\Sigma|\Lambda]\right]
=\E\left[\int_0^\infty \1_{[0,f(\Lambda))}(s)g(s)^{-1}g(s)ds\right]
=\E[f(\Lambda)],\] 
and its variance equals $\E[\Sigma^2]-\E[f(\Lambda)]^2=
\E[G(f(\Lambda))]-\E[f(\Lambda)]^2$, where $G(r)=\int_0^t\frac{1}{g(s)}ds$. If we 
use the density $g:s\mapsto\delta(1+s)^{-1-\delta}$, for some $\delta>0$, then the 
variance of $\Sigma$ (resp. $f(\Lambda)$) is 
$\E\big[\frac{1}{\delta(2+\delta)}((1+f(\Lambda))^{2+\delta}-1)\big]-
\E[f(\Lambda)]^2$ (resp. $\E[f(\Lambda)^2]-\E[f(\Lambda)]^2$). Thus, $\Sigma$ can 
have finite variance if $f(\Lambda)$ has a finite $2+\delta$-moment. This application 
was proposed in~\citep{MR3619789} for the identity function $f_2(t)=t$ and any 
Lipschitz functional $f_1$.

\subsubsection{Unbiased simulation of a continuous finite variation function of the first passage time\label{subsec:hit-times}} 
Let $(\mathcal{X}_t)_{t\in[0,T)}$, $0<T\leq\infty$ be a real-valued \cadlag~process 
such that $\mathcal{X}_0=0$  and, for every $t>0$, there is an $\varepsilon$SS 
algorithm of $\ov{\mathcal{X}}_t=\sup_{s\in[0,t]\cap[0,T)}\mathcal{X}_s$. Fix any 
$x>0$ satisfying $\P(\ov{\mathcal{X}}_t=x)=0$ for almost every $t\in[0,T)$. Then 
$\sigma_x=\min\{T,\inf\{t\in(0,T):\mathcal{X}_t>x\}\}$ (using the convention 
$\inf\emptyset=\infty$) is the first passage time of level $x$ and satisfies the 
identity $\{t<\sigma_x\}=\{\ov{\mathcal{X}}_{\min\{t,T\}}\leq x\}$, for $t\geq 0$. 
By linearity, it suffices to consider a nondecreasing continuous function 
$f:[0,T)\to[0,\infty)$ with generalised inverse $f^*$. 

Let $\varsigma$ be as in Subsection~\ref{subsec:continuous} and 
$f(T)=\lim_{t\ua T}f(t)$. By~\cite[Prop.~4.2]{GenInverses}, $f^*$ is strictly increasing 
and $\{\varsigma<f(\sigma_x)\}=\{f^*(\varsigma)<\sigma_x\}
=\{f^*(\varsigma)<T,\ov{\mathcal{X}}_{f^*(\varsigma)}\leq x\}$, where 
$\P(\ov{\mathcal{X}}_{f^*(\varsigma)}= x,f^*(\varsigma)< T)=0$ by assumption. Hence 
$\Sigma=\1_{(\varsigma,\infty)}(f(\sigma_x))/g(\varsigma)$ is simulated by sampling 
$\varsigma$ and then, as in Subsection~\ref{subsec:indicators}, setting
$\Sigma=\1_{[0,x]}(\mathcal{X}_{f^*(\varsigma)})/g(\varsigma)$ if $f^*(\varsigma)<T$ 
and otherwise putting $\Sigma=0$. Moreover, by Subsection~\ref{subsec:continuous}, 
$\Sigma$ is unbiased for $\E[f(\sigma_x)]$. We stress that, unlike the functionals 
considered in Subsections~\ref{subsec:indicators} and~\ref{subsec:continuous} 
above, it is not immediately clear how to estimate $\E[f(\sigma_x)]$ using a simulation 
algorithm for $\ov{\mathcal{X}}_t$, $t\in[0,T)$. In Subsection~\ref{subsec:num-hit} 
we present a concrete example for weakly stable processes.

We end with the following remark. Consider the time the process $\mathcal{X}$ 
down-crosses (resp. up-crosses) a convex (resp. concave) function mapping 
$[0,\infty)$ to $\R$ started below  (resp. above) $\mathcal{X}_0=0$. If one has an 
$\varepsilon$SS algorithm for the convex minorant (resp. concave majorant) of 
$\mathcal{X}$, then a simple modification of the argument in the previous paragraph 
yields an unbiased simulation algorithm of any finite variation continuous 
function of such a first passage time.

\subsection{Numerical results\label{subsec:numerics}}
In this subsection we explore three applications of the $\varepsilon$SS of stable 
meanders and their convex minorants. Since Algorithm~\ref{alg:eps-CM-SM} 
uses~\citep[Alg.~2]{ExactSim} for backward simulation, we specify the values of 
the parameters $(d,\delta,\gamma,\kappa,m,m^\ast)=\varpi(\alpha,\rho)$
appearing in~\citep[Sec.~4]{ExactSim} and $m$ in Algorithm~\ref{alg:eps-CM-SM}
as follows: $(d^*,r)=(\frac{2}{3\alpha\rho},\frac{19}{20})$ and
\[\varpi(\alpha,\rho)=\bigg(d^*,\frac{d^*}{2},r\alpha,
4+\max\bigg\{\frac{\log(2)}{3\eta(d^*)},\frac{1}{\alpha\rho}\bigg\},
\bigg\lceil\frac{|\log(\varepsilon/2)|}{\log\big(\frac{\Gamma(1+\rho+1/\alpha)}
{\Gamma(1+\rho)\Gamma(1+1/\alpha)}\big)}\bigg\rceil,
12+\Big\lfloor\frac{3\rho}{r}\log\E S^{r\alpha}\Big\rfloor^+\bigg),\] 
where $\eta(d)=-\alpha\rho-\mathcal{W}_{-1}(-\alpha\rho de^{-\alpha\rho d})/d$ is the 
unique positive root of the equation $dt=\log(1+t/(\alpha\rho))$ (here, $\mathcal{W}_{-1}$ 
is the secondary branch of the Lambert W function~\cite{MR1414285}) and $S$ follows 
the law $\mS^+(\alpha,\rho)$. As usual, $\lfloor x\rfloor=\sup\{n\in\mathbb{Z}:n\leq x\}$ 
and $\lceil x \rceil=\inf\{n\in\mathbb{Z}:n\geq x\}$ denote the floor and ceiling 
functions and $x^+=\max\{0,x\}$ for any $x\in\R$.
%We hereafter make reference to the notation, 
%results and algorithms from~\citep[Sec.~4]{ExactSim} (see also 
%Appendix~\ref{sec:notationSupSim} below). Throughout, the parameters involved in 
%the simulation algorithms are chosen as follows: 
%$(d,\delta,\gamma,\kappa,m,m^\ast)=\varpi(\alpha,\rho)$ where 
%$m$ is the burn-in 
%parameter of Algorithm~\ref{alg:eps-CM-SM} above, 
%$(d^*,r)=(\frac{2}{3\alpha\rho},\frac{19}{20})$ and
%\[\begin{split}
%\varpi(\alpha,\rho)
%=&\bigg(\frac{2}{3\alpha\rho},\frac{1}{3\alpha\rho},\frac{19}{20}\alpha,
%4+\max\bigg\{\frac{\log(2)}{3\eta(\frac{2}{3\alpha\rho})},
%\frac{1}{\alpha\rho}\bigg\},\\
%&\quad\quad\bigg\lceil\frac{|\log(\varepsilon/2)|}{\log\big(\frac{
%\Gamma(1+\rho+1/\alpha)}{\Gamma(1+\rho)\Gamma(1+1/\alpha)}\big)}\bigg\rceil,
%12+\Big\lfloor\frac{60}{19}\rho\log\E S_1^{\frac{19}{20}\alpha}\Big\rfloor^+\bigg).
%\end{split}\] 
This choice of $m$ satisfies $\E[U^{1/\alpha}]^m\approx\varepsilon/2$ for 
$\varepsilon<1$, where $U\sim \Beta(1,\rho)$. We fix $\varepsilon=2^{-32}$ 
throughout unless adaptive precision is required 
(see Subsections~\ref{subsec:indicators}--\ref{subsec:hit-times}). 

Figure~\ref{fig:Ne} graphs the empirical distribution function for the running time 
of Algorithm~\ref{alg:eps-CM-SM}, suggesting the existence of exponential moments of 
$|N(\varepsilon)|$, cf. Proposition~\ref{prop:MainRuntime} below. 

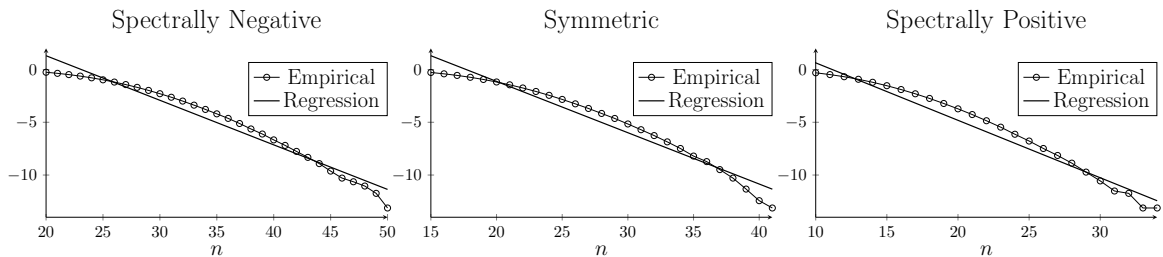
\begin{figure}[H]
	\centering{}\resizebox{.32\linewidth}{!}{
		\begin{subfigure}[h]{.55\linewidth}
			\begin{tikzpicture} 
			\begin{axis} 
			[
			xmin=20,
			xmax=50,
			ymin=-14,
			ymax=2,
			title={\LARGE Spectrally Negative},
			xlabel={\Large $n$},
			width=9.5cm,
			height=5.5cm,
			axis on top=true,
			axis x line=bottom,
			axis y line=middle,
			legend style = {at={(1,.6)}, anchor=south east}
			]
			
			% Estimate
			\addplot[
			%nothing,
			color=black,
			mark=o,
			]
			coordinates {
				(20.0,-0.260578)(21.0,-0.362811)(22.0,-0.482582)(23.0,-0.624558)(24.0,-0.788549)(25.0,-0.977416)(26.0,-1.19028)(27.0,-1.42788)(28.0,-1.69076)(29.0,-1.97994)(30.0,-2.29277)(31.0,-2.62804)(32.0,-2.99146)(33.0,-3.3777)(34.0,-3.78231)(35.0,-4.20171)(36.0,-4.63254)(37.0,-5.117)(38.0,-5.63407)(39.0,-6.11386)(40.0,-6.66874)(41.0,-7.19277)(42.0,-7.76107)(43.0,-8.32657)(44.0,-8.91767)(45.0,-9.62586)(46.0,-10.2892)(47.0,-10.6375)(48.0,-11.0429)(49.0,-11.7361)(50.0,-13.1224)
			};
			
			% Linear regression
			\addplot[
			solid,
			color=black,
			thick,
			]
			coordinates {
				%y=-0.42225*x+9.75028
				(20.0,1.3053)(50.0,-11.3622)
			};
			
			% Legend 
			\legend {\Large Empirical, \Large Regression};
			\end{axis}
			\end{tikzpicture}
	\end{subfigure}}
	\centering{}\resizebox{.32\linewidth}{!}{
		\begin{subfigure}[h]{.55\linewidth}
			\begin{tikzpicture} 
			\begin{axis} 
			[
			xmin=15,
			xmax=41,
			ymin=-14,
			ymax=2,
			title={\LARGE Symmetric},
			xlabel={\Large $n$},
			width=9.5cm,
			height=5.5cm,
			axis on top=true,
			axis x line=bottom,
			axis y line=middle,
			legend style = {at={(1,.6)}, anchor=south east}
			]
			
			% Estimate
			\addplot[
			%nothing,
			color=black,
			mark=o,
			]
			coordinates {
				(15.0,-0.281878)(16.0,-0.409395)(17.0,-0.559071)(18.0,-0.736878)(19.0,-0.94378)(20.0,-1.18219)(21.0,-1.45071)(22.0,-1.75006)(23.0,-2.0809)(24.0,-2.4462)(25.0,-2.83419)(26.0,-3.25424)(27.0,-3.69152)(28.0,-4.16344)(29.0,-4.63831)(30.0,-5.16134)(31.0,-5.70538)(32.0,-6.26065)(33.0,-6.85706)(34.0,-7.48401)(35.0,-8.20238)(36.0,-8.72791)(37.0,-9.48478)(38.0,-10.2892)(39.0,-11.3306)(40.0,-12.4292)(41.0,-13.1224)
			};
			
			% Linear regression
			\addplot[
			solid,
			color=black,
			thick,
			]
			coordinates {
				%y=-0.4863*x+8.5992
				(15.0,1.3047)(41.0,-11.3391)
			};
			
			% Legend 
			\legend {\Large Empirical, \Large Regression};
			\end{axis}
			\end{tikzpicture}
	\end{subfigure}}
	\centering{}\resizebox{.32\linewidth}{!}{
		\begin{subfigure}[h]{.55\linewidth}
			\begin{tikzpicture}
			\begin{axis} 
			[
			xmin=10,
			xmax=34,
			ymin=-14,
			ymax=2,
			title={\LARGE Spectrally Positive},
			xlabel={\Large $n$},
			width=9.5cm,
			height=5.5cm,
			axis on top=true,
			axis x line=bottom,
			axis y line=middle,
			legend style = {at={(1,.6)}, anchor=south east}
			]
			
			% Estimate
			\addplot[
			%nothing,
			color=black,
			mark=o,
			]
			coordinates {
				(10.0,-0.311439)(11.0,-0.480782)(12.0,-0.683312)(13.0,-0.926104)(14.0,-1.20998)(15.0,-1.53454)(16.0,-1.89856)(17.0,-2.30137)(18.0,-2.74087)(19.0,-3.21948)(20.0,-3.72712)(21.0,-4.26813)(22.0,-4.85517)(23.0,-5.46319)(24.0,-6.13487)(25.0,-6.77848)(26.0,-7.48401)(27.0,-8.14563)(28.0,-8.87387)(29.0,-9.72117)(30.0,-10.5574)(31.0,-11.5129)(32.0,-11.7361)(33.0,-13.1224)(34.0,-13.1224)
			};
			
			% Linear regression
			\addplot[
			solid,
			color=black,
			thick,
			]
			coordinates {
				%y=-0.5672*x+6.8470
				(9.0,1.175)(34.0,-12.4378)
			};
			
			% Legend 
			\legend {\Large Empirical, \Large Regression};
			\end{axis}
			\end{tikzpicture}
	\end{subfigure}}
	\caption{\footnotesize
	The graphs show the estimated value of $n\mapsto\log\P(|N(\varepsilon)|>n)$ 
	in the spectrally negative, symmetric and positive cases for $\alpha=1.5$, 
	$\varepsilon=2^{-32}$ and based on $N=5\times10^5$ samples. The curvature %observed 
	in all three graphs suggests that $|N(\varepsilon)|$ has exponential moments of 
	all orders, a stronger claim than those of Theorem~\ref{thm:all_complexities} 
	(see also Proposition~\ref{prop:MainRuntime}).}\label{fig:Ne}
\end{figure}

To demonstrate practical feasibility, we first study the running time of 
Algorithm~\ref{alg:eps-CM-SM}. We implemented Algorithm~\ref{alg:eps-CM-SM}
in the Julia 1.0 programming language (see~\cite{Jorge_GitHub2}) and ran it on 
macOS Mojave 10.14.3 (18D109) with a 4.2 GHz Intel\textregistered 
Core\texttrademark i7 processor and an 8 GB 2400 MHz DDR4 memory. Under these 
conditions, generating $N=10^4$ samples takes approximately 1.30 seconds for any 
$\alpha>1$ and all permissible $\rho$ as long as $\rho$ is bounded away from $0$. 
This task much less time  for $\alpha<1$ so long as $\alpha$ and $\rho$ are bounded 
away from $0$. The performance worsens dramatically as either $\alpha\to0$ or 
$\rho\to0$. 

This behaviour is as expected since the coefficient in front of $M_n$ 
in~\eqref{eq:SMPerEq} of Theorem~\ref{thm:CMSM} follows the law 
$\Beta(1,\alpha\rho)$ with mean $\frac{1}{1+\alpha\rho}$, which tends to $1$ as 
$\alpha\rho\to0$. Hence, the Markov chain decreases very slowly when 
$\alpha\rho$ is close to $0$. From a geometric viewpoint,  note that as $\rho\to0$, 
the mean length of each sampled face (as a proportion of the lengths of the remaining 
faces) satisfies $\E\ell_1^\me=\frac{\rho}{1+\rho}\to0$, implying that large faces are 
increasingly rare. Moreover, as the stability index $\alpha$ decreases, the tails of the 
density of the L\'evy measure become very heavy, making a face of small length 
and huge height likely. To illustrate this numerically, the approximate time (in seconds)
taken by Algorithm~\ref{alg:eps-CM-SM} to produce $N=10^4$ samples for certain 
combinations of parameters is found in the following table: 

\begin{center}\begin{tabular}{|c|c|c|c|c|c|c|}\hline
		$\alpha\setminus\rho$ & 0.95 & 0.5 & 0.1 & 0.05 & 0.01 & 0.005 \\ \hline
		0.5 & 0.301 & 0.314 & 0.690 & 1.165 & 4.904 & 9.724 \\ \hline
		0.1 & 0.197 & 0.242 & 0.738 & 1.367 & 6.257 & 12.148 \\ \hline
		0.05 & 0.229 & 0.318 & 1.125 & 2.137 & 9.864 & 20.131 \\ \hline
\end{tabular}\end{center}

The remainder of the subsection is as follows. In Subsection~\ref{subsec:num-marginal} 
we estimate the mean of $Z^\me_1$ as a function of the stability parameter $\alpha$ in 
the spectrally negative, symmetric and positive cases. The results are compared with the 
exact mean, computed in Corollary~\ref{cor:moments} via the perpetuity in 
Theorem~\ref{thm:CMSM}(c). In Subsection~\ref{subsec:num-hit} we numerically 
analyse the first passage times of weakly stable processes. In 
Subsection~\ref{subsec:num-excursion} we estimate the mean 
of the normalised stable excursion at time $1/2$ and construct confidence intervals.
%based on Subsection~\ref{subsec:continuous} and~\cite[Thm~3]{MR1465814}. 

\subsubsection{Marginal of the normalised stable meander $Z^\me_1$\label{subsec:num-marginal}}
Let $\{(\zeta_i^{\varepsilon,\da},\zeta_i^{\varepsilon,\ua})\}_{i\in\N}$ be an iid sequence 
of $\varepsilon$-strong samples of $Z_1^\me$. Put differently, for all $i\in\N$, we have 
$0<\zeta_i^{\varepsilon,\ua}-\zeta_i^{\varepsilon,\da}<\varepsilon$ and the 
corresponding sample of $Z_1^\me$ lies in the interval 
$(\zeta_i^{\varepsilon,\da},\zeta_i^{\varepsilon,\ua})$. For any continuous 
function $f:\R_+\to\R_+$ with $\E[|f(Z_1^\me)|]<\infty$, a Monte Carlo estimate of 
$\E[f(Z_1^\me)]$ is given by $\frac{1}{2N}\sum_{i=1}^N (f(\zeta_i^{\varepsilon,\da})+
f(\zeta_i^{\varepsilon,\ua}))$. If $f$ is nondecreasing we clearly have the inequalities 
$\E[f(\zeta_1^{\varepsilon,\da})]\leq \E[f(Z_1^\me)]\leq \E[f(\zeta_1^{\varepsilon,\ua})]$.
Thus, a confidence interval $(a,b)$ for $\E[f(Z_1^\me)]$ may be constructed as follows: 
$a$ (resp. $b$) is given by the lower (resp. upper) end of the confidence interval (CI) for 
$\E[f(\zeta_1^{\varepsilon,\da})]$ (resp. $\E[f(\zeta_1^{\varepsilon,\ua})]$). 
We now use Algorithm~\ref{alg:eps-CM-SM} to estimate $\E[Z_1^\me]$ (for $\alpha>1$) 
and $\E[(Z_1^\me)^{-\alpha\rho}]$ and compare the estimates with the formulae for 
the expectations from Corollary~\ref{cor:moments}. The results are shown in 
Figure~\ref{fig:ecdf} below.

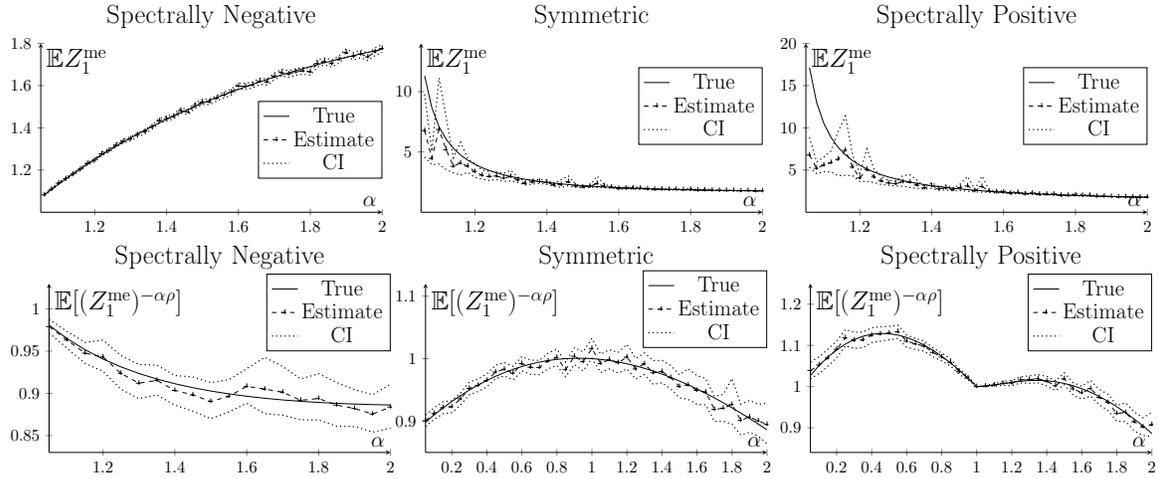
\begin{figure}[H]
	\centering{}\resizebox{.32\linewidth}{!}{
		\begin{subfigure}[h]{.55\linewidth}
			\begin{tikzpicture} 
			\begin{axis} 
			[
			xmin=1.05,
			xmax=2,
			ymin=1,
			ymax=1.8,
			title={\LARGE Spectrally Negative},
			xlabel={\Large $\alpha$},
			ylabel={\LARGE $\E Z^\me_1$},
			width=9.5cm,
			height=5.5cm,
			axis on top=true,
			axis x line=middle,
			axis y line=middle,
			legend style = {at={(1,.2)}, anchor=south east}
			]
			
			% Real
			\addplot[
			solid,
			color=black,
			]
			coordinates {
				(1.06,1.08222)(1.08,1.10824)(1.1,1.13357)(1.12,1.15824)(1.14,1.18224)(1.16,1.20559)(1.18,1.2283)(1.2,1.25039)(1.22,1.27187)(1.24,1.29276)(1.26,1.31305)(1.28,1.33279)(1.3,1.35196)(1.32,1.3706)(1.34,1.38872)(1.36,1.40632)(1.38,1.42342)(1.4,1.44005)(1.42,1.4562)(1.44,1.4719)(1.46,1.48716)(1.48,1.50199)(1.5,1.5164)(1.52,1.53041)(1.54,1.54403)(1.56,1.55727)(1.58,1.57014)(1.6,1.58265)(1.62,1.59482)(1.64,1.60665)(1.66,1.61815)(1.68,1.62934)(1.7,1.64022)(1.72,1.65081)(1.74,1.66111)(1.76,1.67113)(1.78,1.68088)(1.8,1.69036)(1.82,1.6996)(1.84,1.70858)(1.86,1.71733)(1.88,1.72585)(1.9,1.73414)(1.92,1.74221)(1.94,1.75007)(1.96,1.75773)(1.98,1.76519)(2.0,1.77245)
			};
			
			% Estimate
			\addplot[
			dashed,
			mark=+,
			color=black,
			]
			coordinates {
				(1.06,1.0823)(1.08,1.10869)(1.1,1.13765)(1.12,1.15843)(1.14,1.18037)(1.16,1.203)(1.18,1.23078)(1.2,1.24604)(1.22,1.27644)(1.24,1.29302)(1.26,1.31756)(1.28,1.33648)(1.3,1.34794)(1.32,1.36733)(1.34,1.3814)(1.36,1.40391)(1.38,1.43355)(1.4,1.44271)(1.42,1.45389)(1.44,1.4814)(1.46,1.47524)(1.48,1.50554)(1.5,1.52274)(1.52,1.52287)(1.54,1.54062)(1.56,1.55482)(1.58,1.56912)(1.6,1.59788)(1.62,1.59669)(1.64,1.60307)(1.66,1.62055)(1.68,1.61728)(1.7,1.64336)(1.72,1.6621)(1.74,1.65472)(1.76,1.66499)(1.78,1.67005)(1.8,1.66554)(1.82,1.70523)(1.84,1.71433)(1.86,1.70446)(1.88,1.72418)(1.9,1.7568)(1.92,1.73882)(1.94,1.74992)(1.96,1.73625)(1.98,1.76042)(2.0,1.77616)
			};
			
			% Low
			\addplot[
			dotted,
			color=black,
			thick,
			]
			coordinates {
				(1.06,1.08584)(1.08,1.11279)(1.1,1.14238)(1.12,1.16362)(1.14,1.18618)(1.16,1.2093)(1.18,1.23739)(1.2,1.25321)(1.22,1.28414)(1.24,1.30122)(1.26,1.32592)(1.28,1.34547)(1.3,1.35692)(1.32,1.37698)(1.34,1.39146)(1.36,1.41407)(1.38,1.44421)(1.4,1.45374)(1.42,1.46555)(1.44,1.49301)(1.46,1.48735)(1.48,1.51782)(1.5,1.53559)(1.52,1.53594)(1.54,1.55388)(1.56,1.56847)(1.58,1.58285)(1.6,1.6116)(1.62,1.61066)(1.64,1.61768)(1.66,1.63501)(1.68,1.63203)(1.7,1.6587)(1.72,1.67769)(1.74,1.67025)(1.76,1.68049)(1.78,1.68592)(1.8,1.68196)(1.82,1.72222)(1.84,1.73074)(1.86,1.72155)(1.88,1.74122)(1.9,1.77427)(1.92,1.75583)(1.94,1.76749)(1.96,1.754)(1.98,1.77862)(2.0,1.79423)
			};
			
			% Up
			\addplot[
			dotted,
			color=black,
			thick,
			]
			coordinates {
				(1.06,1.07867)(1.08,1.10447)(1.1,1.13282)(1.12,1.15311)(1.14,1.17445)(1.16,1.19664)(1.18,1.22411)(1.2,1.23867)(1.22,1.26879)(1.24,1.28498)(1.26,1.30908)(1.28,1.32787)(1.3,1.33887)(1.32,1.35778)(1.34,1.37138)(1.36,1.39365)(1.38,1.42308)(1.4,1.43176)(1.42,1.44259)(1.44,1.46945)(1.46,1.46302)(1.48,1.49322)(1.5,1.51025)(1.52,1.50986)(1.54,1.52765)(1.56,1.54122)(1.58,1.55559)(1.6,1.5844)(1.62,1.58296)(1.64,1.58838)(1.66,1.60589)(1.68,1.60231)(1.7,1.6282)(1.72,1.64688)(1.74,1.6392)(1.76,1.64912)(1.78,1.65387)(1.8,1.64928)(1.82,1.68819)(1.84,1.69789)(1.86,1.68781)(1.88,1.70712)(1.9,1.73981)(1.92,1.7216)(1.94,1.73214)(1.96,1.71869)(1.98,1.74194)(2.0,1.7581)
			};
			
			% Legend 
			\legend {\Large True,\Large Estimate,\Large CI};
			\end{axis}
			\end{tikzpicture}
	\end{subfigure}}
	\centering{}\resizebox{.32\linewidth}{!}{
		\begin{subfigure}[h]{.55\linewidth}
			\begin{tikzpicture}
			\begin{axis} 
			[
			%log basis y=2,
			title = {\LARGE $\delta\times \alpha\rho$},
			ymin = 0,
			ymax = 14,
			xmin=1.05,
			xmax=2,
			title={\LARGE Symmetric},
			xlabel={\Large $\alpha$},
			ylabel={\LARGE $\E Z^\me_1$},
			width=9.5cm,
			height=5.5cm,
			axis on top=true,
			axis x line=middle,
			axis y line=middle,
			legend style = {at={(1,.4)}, anchor=south east}
			] 
			
			% Real
			\addplot[
			solid,
			color=black,
			]
			coordinates {
				(1.06,11.3126)(1.08,8.67313)(1.1,7.09429)(1.12,6.0456)(1.14,5.29971)(1.16,4.74299)(1.18,4.31229)(1.2,3.96975)(1.22,3.69128)(1.24,3.46082)(1.26,3.26725)(1.28,3.10264)(1.3,2.96118)(1.32,2.8385)(1.34,2.73127)(1.36,2.63691)(1.38,2.55336)(1.4,2.479)(1.42,2.4125)(1.44,2.35278)(1.46,2.29896)(1.48,2.25027)(1.5,2.20611)(1.52,2.16595)(1.54,2.12933)(1.56,2.09587)(1.58,2.06525)(1.6,2.03717)(1.62,2.01139)(1.64,1.98768)(1.66,1.96586)(1.68,1.94575)(1.7,1.92721)(1.72,1.91011)(1.74,1.89432)(1.76,1.87975)(1.78,1.86629)(1.8,1.85386)(1.82,1.84238)(1.84,1.83179)(1.86,1.82203)(1.88,1.81302)(1.9,1.80474)(1.92,1.79711)(1.94,1.79011)(1.96,1.78369)(1.98,1.77782)(2.0,1.77245)
			};
			
			% Estimate
			\addplot[
			dashed,
			mark=+,
			color=black,
			]
			coordinates {
				(1.06,6.77211)(1.08,4.46232)(1.1,6.86083)(1.12,5.18721)(1.14,3.81027)(1.16,4.075)(1.18,3.81276)(1.2,3.36024)(1.22,3.06909)(1.24,2.96538)(1.26,3.00241)(1.28,2.82873)(1.3,2.95819)(1.32,2.67398)(1.34,2.38402)(1.36,2.50749)(1.38,2.54152)(1.4,2.48123)(1.42,2.26011)(1.44,2.28457)(1.46,2.53557)(1.48,2.17159)(1.5,2.09536)(1.52,2.15048)(1.54,2.38678)(1.56,2.1513)(1.58,2.0799)(1.6,1.9765)(1.62,1.99461)(1.64,1.99629)(1.66,1.98508)(1.68,1.93358)(1.7,1.897)(1.72,1.91283)(1.74,1.88084)(1.76,1.89554)(1.78,1.86612)(1.8,1.89591)(1.82,1.82628)(1.84,1.84393)(1.86,1.82792)(1.88,1.80853)(1.9,1.79881)(1.92,1.80732)(1.94,1.77743)(1.96,1.79976)(1.98,1.76787)(2.0,1.77411)
			};
			
			% Low
			\addplot[
			dotted,
			color=black,
			thick,
			]
			coordinates {
				(1.06,9.7775)(1.08,5.05407)(1.1,11.1405)(1.12,7.83051)(1.14,4.51264)(1.16,5.80994)(1.18,4.43978)(1.2,4.03325)(1.22,3.48365)(1.24,3.38591)(1.26,3.33963)(1.28,3.22783)(1.3,3.48297)(1.32,2.91315)(1.34,2.52974)(1.36,2.73538)(1.38,2.70812)(1.4,2.683)(1.42,2.36306)(1.44,2.70627)(1.46,3.00428)(1.48,2.29334)(1.5,2.18022)(1.52,2.2466)(1.54,2.99649)(1.56,2.28404)(1.58,2.16795)(1.6,2.04956)(1.62,2.06726)(1.64,2.09526)(1.66,2.05467)(1.68,1.99438)(1.7,1.94217)(1.72,1.96747)(1.74,1.91744)(1.76,1.93777)(1.78,1.90147)(1.8,1.97894)(1.82,1.86023)(1.84,1.91972)(1.86,1.86308)(1.88,1.83266)(1.9,1.82661)(1.92,1.83092)(1.94,1.79727)(1.96,1.81955)(1.98,1.7867)(2.0,1.79198)
			};
			
			% Up
			\addplot[
			dotted,
			color=black,
			thick,
			]
			coordinates {
				(1.06,4.57551)(1.08,3.94061)(1.1,3.98822)(1.12,3.53569)(1.14,3.267)(1.16,3.08315)(1.18,3.30487)(1.2,2.8887)(1.22,2.74049)(1.24,2.62876)(1.26,2.72176)(1.28,2.54444)(1.3,2.6042)(1.32,2.46954)(1.34,2.25899)(1.36,2.33198)(1.38,2.39735)(1.4,2.30806)(1.42,2.16492)(1.44,2.04719)(1.46,2.20891)(1.48,2.07)(1.5,2.01999)(1.52,2.06715)(1.54,2.04635)(1.56,2.04263)(1.58,2.00319)(1.6,1.91218)(1.62,1.932)(1.64,1.92238)(1.66,1.92248)(1.68,1.88014)(1.7,1.85564)(1.72,1.86554)(1.74,1.84556)(1.76,1.85631)(1.78,1.83345)(1.8,1.8376)(1.82,1.79457)(1.84,1.79346)(1.86,1.79705)(1.88,1.78481)(1.9,1.77335)(1.92,1.78388)(1.94,1.75771)(1.96,1.78024)(1.98,1.749)(2.0,1.75618)
			};
			
			% Legend 
			\legend {\Large True,\Large Estimate,\Large CI};
			\end{axis}
			\end{tikzpicture}
	\end{subfigure}}
	\centering{}\resizebox{.32\linewidth}{!}{
		\begin{subfigure}[h]{.55\linewidth}
			\begin{tikzpicture} 
			\begin{axis} 
			[
			%log basis y=2,
			ymax=20,
			ymin=0,
			xmin=1.05,
			xmax=2,
			title={\LARGE Spectrally Positive},
			xlabel={\Large $\alpha$},
			ylabel={\LARGE $\E Z^\me_1$},
			width=9.5cm,
			height=5.5cm,
			axis on top=true,
			axis x line=middle,
			axis y line=middle,
			legend style = {at={(1,.4)}, anchor=south east}
			]
			
			% Real
			\addplot[
			solid,
			color=black,
			]
			coordinates {
				(1.06,17.1427)(1.08,12.9914)(1.1,10.5059)(1.12,8.85276)(1.14,7.67504)(1.16,6.79421)(1.18,6.11112)(1.2,5.56632)(1.22,5.12196)(1.24,4.75285)(1.26,4.44154)(1.28,4.17558)(1.3,3.94584)(1.32,3.74549)(1.34,3.56929)(1.36,3.41319)(1.38,3.27399)(1.4,3.14912)(1.42,3.03651)(1.44,2.93447)(1.46,2.84161)(1.48,2.75675)(1.5,2.67894)(1.52,2.60734)(1.54,2.54124)(1.56,2.48006)(1.58,2.42328)(1.6,2.37044)(1.62,2.32115)(1.64,2.27509)(1.66,2.23194)(1.68,2.19145)(1.7,2.15338)(1.72,2.11753)(1.74,2.08371)(1.76,2.05176)(1.78,2.02153)(1.8,1.99289)(1.82,1.96572)(1.84,1.93992)(1.86,1.91537)(1.88,1.89201)(1.9,1.86974)(1.92,1.84849)(1.94,1.8282)(1.96,1.80879)(1.98,1.79023)(2.0,1.77245)
			};
			
			% Estimate
			\addplot[
			dashed,
			mark=+,
			color=black,
			]
			coordinates {
				(1.06,6.7808)(1.08,5.20343)(1.1,5.58314)(1.12,5.87811)(1.14,6.27225)(1.16,7.34663)(1.18,5.34677)(1.2,4.0667)(1.22,5.08428)(1.24,4.16718)(1.26,3.71117)(1.28,3.47493)(1.3,3.38038)(1.32,3.59848)(1.34,3.56655)(1.36,3.20868)(1.38,2.88315)(1.4,3.14808)(1.42,2.77631)(1.44,2.88301)(1.46,2.67094)(1.48,2.92676)(1.5,3.11612)(1.52,2.54938)(1.54,3.05206)(1.56,2.50298)(1.58,2.40196)(1.6,2.40493)(1.62,2.25484)(1.64,2.33454)(1.66,2.16078)(1.68,2.24644)(1.7,2.12202)(1.72,2.13095)(1.74,2.12811)(1.76,1.99033)(1.78,1.99613)(1.8,2.04781)(1.82,1.98989)(1.84,2.001)(1.86,1.87881)(1.88,1.89373)(1.9,1.85398)(1.92,1.82306)(1.94,1.81899)(1.96,1.80799)(1.98,1.79574)(2.0,1.78877)
			};
			
			% Low
			\addplot[
			dotted,
			color=black,
			thick,
			]
			coordinates {
				(1.06,8.85099)(1.08,5.96382)(1.1,6.54031)(1.12,7.26366)(1.14,9.59744)(1.16,11.4023)(1.18,6.76412)(1.2,4.61793)(1.22,7.50741)(1.24,5.06786)(1.26,4.24582)(1.28,3.78025)(1.3,3.69176)(1.32,4.34692)(1.34,4.42671)(1.36,3.76507)(1.38,3.02784)(1.4,3.52273)(1.42,2.96173)(1.44,3.22734)(1.46,2.81123)(1.48,3.36737)(1.5,4.23569)(1.52,2.7146)(1.54,4.27688)(1.56,2.66498)(1.58,2.53437)(1.6,2.52556)(1.62,2.37068)(1.64,2.46706)(1.66,2.24052)(1.68,2.36591)(1.7,2.23872)(1.72,2.19892)(1.74,2.24528)(1.76,2.04591)(1.78,2.04203)(1.8,2.2027)(1.82,2.10075)(1.84,2.11008)(1.86,1.91193)(1.88,1.929)(1.9,1.8897)(1.92,1.84855)(1.94,1.84135)(1.96,1.8307)(1.98,1.81569)(2.0,1.80718)
			};
			
			% Up
			\addplot[
			dotted,
			color=black,
			thick,
			]
			coordinates {
				(1.06,5.32209)(1.08,4.56637)(1.1,4.80757)(1.12,4.76509)(1.14,4.32535)(1.16,4.37249)(1.18,4.32263)(1.2,3.63784)(1.22,3.66231)(1.24,3.52126)(1.26,3.27731)(1.28,3.20716)(1.3,3.10562)(1.32,3.09196)(1.34,2.92693)(1.36,2.82226)(1.38,2.74782)(1.4,2.84186)(1.42,2.61008)(1.44,2.62987)(1.46,2.54223)(1.48,2.58757)(1.5,2.48322)(1.52,2.4122)(1.54,2.38009)(1.56,2.37258)(1.58,2.29169)(1.6,2.29997)(1.62,2.15534)(1.64,2.22335)(1.66,2.09013)(1.68,2.14414)(1.7,2.0287)(1.72,2.06669)(1.74,2.04411)(1.76,1.942)(1.78,1.95363)(1.8,1.94828)(1.82,1.91456)(1.84,1.92404)(1.86,1.84689)(1.88,1.85937)(1.9,1.82118)(1.92,1.79827)(1.94,1.79692)(1.96,1.78541)(1.98,1.77609)(2.0,1.77)
			};
			
			% Legend 
			\legend {\Large True,\Large Estimate,\Large CI};
			\end{axis}
			\end{tikzpicture}
	\end{subfigure}}
	\centering{}\resizebox{.32\linewidth}{!}{
		\begin{subfigure}[h]{.55\linewidth}
			\begin{tikzpicture} 
			\begin{axis} 
			[
			ymax=1.03,
			ymin=.83,
			xmin=1.05,
			xmax=2,
			title={\LARGE Spectrally Negative},
			xlabel={\Large $\alpha$},
			ylabel={\LARGE $\E [(Z^\me_1)^{-\alpha\rho}]$},
			width=9.5cm,
			height=5.5cm,
			axis on top=true,
			axis x line=middle,
			axis y line=middle,
			legend style = {at={(1,.6)}, anchor=south east}
			]
			
			% Real
			\addplot[
			solid,
			color=black,
			]
			coordinates {
				(1.05,0.980793)(1.1,0.964912)(1.15,0.951701)(1.2,0.940656)(1.25,0.931384)(1.3,0.923577)(1.35,0.916989)(1.4,0.911423)(1.45,0.906719)(1.5,0.902745)(1.55,0.899394)(1.6,0.896574)(1.65,0.894212)(1.7,0.892245)(1.75,0.890618)(1.8,0.889287)(1.85,0.888213)(1.9,0.887363)(1.95,0.886709)(2.0,0.886227)
			};
			
			% Estimate
			\addplot[
			dashed,
			mark=+,
			color=black,
			]
			coordinates {
				(1.05,0.979112)(1.1,0.963136)(1.15,0.947644)(1.2,0.943104)(1.25,0.923653)(1.3,0.912034)(1.35,0.915387)(1.4,0.903463)(1.45,0.898207)(1.5,0.890445)(1.55,0.896356)(1.6,0.90832)(1.65,0.905037)(1.7,0.901429)(1.75,0.891846)(1.8,0.893874)(1.85,0.886099)(1.9,0.882134)(1.95,0.875484)(2.0,0.883753)
			};
			
			% Low
			\addplot[
			dotted,
			color=black,
			thick,
			]
			coordinates {
				(1.05,0.987704)(1.1,0.973401)(1.15,0.960402)(1.2,0.963103)(1.25,0.94476)(1.3,0.934179)(1.35,0.933404)(1.4,0.921196)(1.45,0.919375)(1.5,0.912746)(1.55,0.916149)(1.6,0.929337)(1.65,0.942722)(1.7,0.931927)(1.75,0.917069)(1.8,0.923324)(1.85,0.912897)(1.9,0.904392)(1.95,0.898301)(2.0,0.910884)
			};
			
			% Up
			\addplot[
			dotted,
			color=black,
			thick,
			]
			coordinates {
				(1.05,0.971514)(1.1,0.953607)(1.15,0.93551)(1.2,0.92616)(1.25,0.904876)(1.3,0.893746)(1.35,0.898223)(1.4,0.886094)(1.45,0.878895)(1.5,0.870045)(1.55,0.877992)(1.6,0.888177)(1.65,0.875955)(1.7,0.874248)(1.75,0.868774)(1.8,0.868429)(1.85,0.861186)(1.9,0.860727)(1.95,0.854121)(2.0,0.858794)
			};
			
			% Legend 
			\legend {\Large True,\Large Estimate,\Large CI};
			\end{axis}
			\end{tikzpicture}
	\end{subfigure}}
	\centering{}\resizebox{.32\linewidth}{!}{
		\begin{subfigure}[h]{.55\linewidth}
			\begin{tikzpicture} 
			\begin{axis} 
			[
			ymax=1.12,
			ymin=.85,
			xmin=.05,
			xmax=2,
			title={\LARGE Symmetric},
			xlabel={\Large $\alpha$},
			ylabel={\LARGE $\E [(Z^\me_1)^{-\alpha\rho}]$},
			width=9.5cm,
			height=5.5cm,
			axis on top=true,
			axis x line=middle,
			axis y line=middle,
			legend style = {at={(1,.62)}, anchor=south east}
			]
			
			% Real
			\addplot[
			solid,
			color=black,
			]
			coordinates {
				(0.05,0.898652)(0.1,0.910347)(0.15,0.921312)(0.2,0.931546)(0.25,0.94105)(0.3,0.949826)(0.35,0.957879)(0.4,0.965211)(0.45,0.97183)(0.5,0.977741)(0.55,0.982952)(0.6,0.987472)(0.65,0.991309)(0.7,0.994474)(0.75,0.996978)(0.8,0.998831)(0.85,1.00005)(0.9,1.00064)(0.95,1.00062)(1.0,1.0)(1.05,0.998798)(1.1,0.997028)(1.15,0.994705)(1.2,0.991843)(1.25,0.988459)(1.3,0.984569)(1.35,0.980189)(1.4,0.975335)(1.45,0.970024)(1.5,0.964273)(1.55,0.958098)(1.6,0.951516)(1.65,0.944545)(1.7,0.9372)(1.75,0.929499)(1.8,0.921458)(1.85,0.913094)(1.9,0.904423)(1.95,0.895462)(2.0,0.886227)
			};
			
			% Estimate
			\addplot[
			dashed,
			mark=+,
			color=black,
			]
			coordinates {
				(0.05,0.900592)(0.1,0.912217)(0.15,0.921613)(0.2,0.923143)(0.25,0.935796)(0.3,0.952412)(0.35,0.95729)(0.4,0.965088)(0.45,0.978657)(0.5,0.983013)(0.55,0.976701)(0.6,0.992019)(0.65,0.986269)(0.7,0.985838)(0.75,0.996186)(0.8,1.00197)(0.85,0.982418)(0.9,1.00381)(0.95,0.995134)(1.0,1.01626)(1.05,0.99226)(1.1,0.998565)(1.15,0.993379)(1.2,1.00255)(1.25,0.983592)(1.3,0.991782)(1.35,0.977446)(1.4,0.97923)(1.45,0.969799)(1.5,0.955436)(1.55,0.958088)(1.6,0.949193)(1.65,0.946503)(1.7,0.918905)(1.75,0.920731)(1.8,0.926672)(1.85,0.90133)(1.9,0.906818)(1.95,0.900663)(2.0,0.894797)
			};
			
			% Low
			\addplot[
			dotted,
			color=black,
			thick,
			]
			coordinates {
				(0.05,0.909843)(0.1,0.921542)(0.15,0.931148)(0.2,0.932724)(0.25,0.945345)(0.3,0.962553)(0.35,0.96736)(0.4,0.97574)(0.45,0.989453)(0.5,0.993752)(0.55,0.98769)(0.6,1.00343)(0.65,0.997637)(0.7,0.997355)(0.75,1.00882)(0.8,1.01512)(0.85,0.994279)(0.9,1.01755)(0.95,1.01089)(1.0,1.03244)(1.05,1.00723)(1.1,1.0142)(1.15,1.01005)(1.2,1.02285)(1.25,1.00038)(1.3,1.00912)(1.35,0.997014)(1.4,0.997184)(1.45,0.989123)(1.5,0.974183)(1.55,0.983897)(1.6,0.972782)(1.65,0.968103)(1.7,0.939727)(1.75,0.942774)(1.8,0.968302)(1.85,0.927237)(1.9,0.933241)(1.95,0.926959)(2.0,0.928968)
			};
			
			% Up
			\addplot[
			dotted,
			color=black,
			thick,
			]
			coordinates {
				(0.05,0.891512)(0.1,0.903016)(0.15,0.912027)(0.2,0.913573)(0.25,0.926118)(0.3,0.942339)(0.35,0.947382)(0.4,0.95487)(0.45,0.968064)(0.5,0.97229)(0.55,0.966003)(0.6,0.98071)(0.65,0.974824)(0.7,0.974517)(0.75,0.983659)(0.8,0.989163)(0.85,0.970978)(0.9,0.990569)(0.95,0.980361)(1.0,1.00062)(1.05,0.977906)(1.1,0.983512)(1.15,0.977042)(1.2,0.983588)(1.25,0.967334)(1.3,0.974585)(1.35,0.958891)(1.4,0.96131)(1.45,0.95166)(1.5,0.93736)(1.55,0.934918)(1.6,0.926956)(1.65,0.926045)(1.7,0.899725)(1.75,0.899881)(1.8,0.895298)(1.85,0.87793)(1.9,0.88223)(1.95,0.876259)(2.0,0.863883)
			};
			
			% Legend 
			\legend {\Large True,\Large Estimate, \Large CI};
			\end{axis}
			\end{tikzpicture}
	\end{subfigure}}
	\centering{}\resizebox{.32\linewidth}{!}{
		\begin{subfigure}[h]{.55\linewidth}
			\begin{tikzpicture} 
			\begin{axis} 
			[
			ymax=1.25,
			ymin=.84,
			xmin=.05,
			xmax=2,
			title={\LARGE Spectrally Positive},
			xlabel={\Large $\alpha$},
			ylabel={\LARGE $\E [(Z^\me_1)^{-\alpha\rho}]$},
			width=9.5cm,
			height=5.5cm,
			axis on top=true,
			axis x line=middle,
			axis y line=middle,
			legend style = {at={(1,.6)}, anchor=south east}
			]
			
			% Real
			\addplot[
			solid,
			color=black,
			]
			coordinates {
				(0.05,1.02722)(0.1,1.05114)(0.15,1.07176)(0.2,1.08912)(0.25,1.10326)(0.3,1.11424)(0.35,1.12214)(0.4,1.12706)(0.45,1.1291)(0.5,1.12838)(0.55,1.12503)(0.6,1.11917)(0.65,1.11097)(0.7,1.10055)(0.75,1.08807)(0.8,1.07367)(0.85,1.05752)(0.9,1.03975)(0.95,1.02053)(1.0,1.0)(1.05,1.00119)(1.1,1.00392)(1.15,1.00722)(1.2,1.0104)(1.25,1.01298)(1.3,1.01461)(1.35,1.01503)(1.4,1.01407)(1.45,1.01162)(1.5,1.00762)(1.55,1.00203)(1.6,0.99485)(1.65,0.986106)(1.7,0.97584)(1.75,0.964109)(1.8,0.950983)(1.85,0.936543)(1.9,0.920874)(1.95,0.90407)(2.0,0.886227)
			};
			
			% Estimate
			\addplot[
			dashed,
			mark=+,
			color=black,
			]
			coordinates {
				(0.05,1.03914)(0.1,1.05231)(0.15,1.06964)(0.2,1.09052)(0.25,1.1184)(0.3,1.11219)(0.35,1.11321)(0.4,1.1245)(0.45,1.12784)(0.5,1.12989)(0.55,1.13332)(0.6,1.11021)(0.65,1.10361)(0.7,1.0982)(0.75,1.08295)(0.8,1.07185)(0.85,1.05872)(0.9,1.03597)(0.95,1.02091)(1.0,1.0)(1.05,1.00142)(1.1,1.0039)(1.15,1.0077)(1.2,1.00852)(1.25,1.01628)(1.3,1.01707)(1.35,1.01898)(1.4,1.01285)(1.45,1.00204)(1.5,1.01884)(1.55,0.9944)(1.6,0.984996)(1.65,0.995283)(1.7,0.973705)(1.75,0.961337)(1.8,0.934537)(1.85,0.937138)(1.9,0.914466)(1.95,0.90214)(2.0,0.90732)
			};
			
			% Low
			\addplot[
			dotted,
			color=black,
			thick,
			]
			coordinates {
				(0.05,1.05973)(0.1,1.07239)(0.15,1.09052)(0.2,1.11074)(0.25,1.1386)(0.3,1.13209)(0.35,1.13228)(0.4,1.14272)(0.45,1.14521)(0.5,1.14629)(0.55,1.14911)(0.6,1.12494)(0.65,1.11693)(0.7,1.11056)(0.75,1.09421)(0.8,1.08138)(0.85,1.06697)(0.9,1.04254)(0.95,1.02545)(1.0,1.0)(1.05,1.00311)(1.1,1.00717)(1.15,1.01242)(1.2,1.01485)(1.25,1.02395)(1.3,1.02606)(1.35,1.02907)(1.4,1.02452)(1.45,1.01432)(1.5,1.03338)(1.55,1.00975)(1.6,1.00062)(1.65,1.01502)(1.7,0.992895)(1.75,0.981297)(1.8,0.954385)(1.85,0.960776)(1.9,0.938128)(1.95,0.924151)(2.0,0.937756)
			};
			
			% Up
			\addplot[
			dotted,
			color=black,
			thick,
			]
			coordinates {
				(0.05,1.01869)(0.1,1.03205)(0.15,1.04925)(0.2,1.07017)(0.25,1.09888)(0.3,1.09258)(0.35,1.09482)(0.4,1.10687)(0.45,1.11045)(0.5,1.11343)(0.55,1.11726)(0.6,1.09558)(0.65,1.09011)(0.7,1.08592)(0.75,1.07197)(0.8,1.0622)(0.85,1.05047)(0.9,1.02936)(0.95,1.01626)(1.0,1.0)(1.05,0.999729)(1.1,1.0007)(1.15,1.00303)(1.2,1.0023)(1.25,1.00875)(1.3,1.00795)(1.35,1.00907)(1.4,1.00154)(1.45,0.990093)(1.5,1.00451)(1.55,0.979338)(1.6,0.969377)(1.65,0.976889)(1.7,0.955615)(1.75,0.941644)(1.8,0.915553)(1.85,0.914981)(1.9,0.892056)(1.95,0.880845)(2.0,0.879095)
			};
			
			% Legend 
			\legend {\Large True,\Large Estimate,\Large CI};
			\end{axis}
			\end{tikzpicture}
	\end{subfigure}}
	\caption{\footnotesize
	Top (resp. bottom) graphs show the true and estimated means $\E Z_1^\me$ 
	(resp. moments $\E[(Z_1^\me)^{-\alpha\rho}]$) with $95\%$ confidence 
	intervals based on $N=10^4$ samples as a function of $\alpha$ for the 
	spectrally negative ($\rho=1/\alpha$ as $\alpha>1$), symmetric ($\rho=1/2$) and 
	positive ($\rho=1-1/\alpha$ if $\alpha>1$ and $\rho=1$ otherwise) cases. 
	The estimates and confidence intervals of $\E Z_1^\me$ are larger and more unstable 
	for values of $\alpha$ close to $1$ (except for the spectrally negative case) since the 
	tails of its distribution are at their heaviest.}\label{fig:ecdf}
\end{figure}

The CLT is not applicable when the variables have infinite variance and can hence 
not be used for the CIs of $\E Z_1^\me$ (except for the spectrally negative case). 
Thus, we use bootstrapping CIs throughout, constructed as follows. Given an iid sample 
$\{x_k\}_{k=1}^n$ and a confidence level $1-\lambda$, we construct the sequence 
$\{\mu_i\}_{i=1}^n$, where $\mu_i=\frac{1}{n}\sum_{k=1}^n x_{k}^{(i)}$ 
and $\{x_k^{(i)}\}_{k=1}^n$ is obtained by resampling with replacement from the set 
$\{x_k\}_{k=1}^n$. We then use the quantiles $\lambda/2$ and 
$1-\lambda/2$ of the empirical distribution of the sample $\{\mu_i\}_{i=1}^n$ as the CI's endpoints for $\E[x_1]$ 
with the point estimator $\mu=\frac{1}{n}\sum_{k=1}^n x_k$. 

\subsubsection{First passage times of weakly stable processes\label{subsec:num-hit}}
%For any $x>0$, the first passage time $\sigma_x=\inf\{t>0:Z_t>x\}$ satisfies, 
%by~\citep[Rem.~5]{MR2599201}, $\E[\sigma_x^\gamma]<\infty$ for 
%$\gamma\in[0,\rho)$ if $\rho\in(0,1)$. Hence, the same is true of the first passage time 
Define the first passage time $\hat\sigma_x=\inf\{t>0: \hat Z_t>x\}$ of the weakly stable 
process $\hat{Z}=(\hat{Z}_t)_{t\geq0}=(Z_t+\mu t)_{t\geq0}$ for some $\mu\in\R$ and all
$x>0$. As a concrete example of the unbiased simulation from 
Subsection~\ref{subsec:hit-times} above, we estimate $\E\hat\sigma_x$ in the present 
subsection. To ensure that the previous expectation is finite, it suffices that 
$\E\hat Z_1=\mu+\E Z_1>0$~\cite[Ex.~VI.6.3]{MR1406564}, where
$\E Z_1=(\sin(\pi\rho)-\sin(\pi(1-\rho)))\Gamma(1-\frac{1}{\alpha})/\pi$ (see 
e.g.~\cite[Eq.~(A.2)]{ExactSim}). Since the time horizon over which the weakly
stable process is simulated is random and equal to $\varsigma$, we chose 
$g:s\mapsto 2(1+s)^{-3}$ to ensure $\E\varsigma<\infty$. The results presented in 
Figure~\ref{fig:hit-time} used the fixed values $\mu=1-\E Z_1$ and $x=1$.
\begin{figure}[H]
	\centering{}\resizebox{.32\linewidth}{!}{
		\begin{subfigure}[h]{.52\linewidth}
			\begin{tikzpicture} 
			\begin{axis} 
			[
			xmin=1.05,
			xmax=2,
			ymin=0,
			ymax=.35,
			title={\LARGE Spectrally Negative},
			xlabel={\Large $\alpha$},
			ylabel={\LARGE $\E\hat\sigma_x$},
			width=9cm,
			height=5cm,
			axis on top=true,
			axis x line=bottom,
			axis y line=middle,
			legend style = {at={(1,.65)}, anchor=south east}
			]
			
			% Estimate
			\addplot[
			solid,
			color=black,
			]
			coordinates {
				(1.05,0.0355168)(1.1,0.0305024)(1.15,0.0632221)(1.2,0.133715)(1.25,0.146268)(1.3,0.127565)(1.35,0.126847)(1.4,0.112314)(1.45,0.140099)(1.5,0.11688)(1.55,0.107495)(1.6,0.128514)(1.65,0.16347)(1.7,0.157656)(1.75,0.158637)(1.8,0.16762)(1.85,0.141316)(1.9,0.146522)(1.95,0.149783)(2.0,0.153162)
			};
			
			% Low
			\addplot[
			dotted,
			color=black,
			thick,
			]
			coordinates {
				(1.05,0.0148834)(1.1,0.0209762)(1.15,0.0331911)(1.2,0.0734562)(1.25,0.0767086)(1.3,0.0693581)(1.35,0.075434)(1.4,0.0877445)(1.45,0.0943225)(1.5,0.093253)(1.55,0.0851608)(1.6,0.106422)(1.65,0.109884)(1.7,0.128051)(1.75,0.132148)(1.8,0.13011)(1.85,0.118451)(1.9,0.119165)(1.95,0.129863)(2.0,0.131743)
			};
			
			% Up
			\addplot[
			dotted,
			color=black,
			thick,
			]
			coordinates {
				(1.05,0.0681608)(1.1,0.0438837)(1.15,0.10307)(1.2,0.205681)(1.25,0.230475)(1.3,0.212453)(1.35,0.197805)(1.4,0.141891)(1.45,0.196904)(1.5,0.145012)(1.55,0.135109)(1.6,0.154353)(1.65,0.216703)(1.7,0.195515)(1.75,0.189139)(1.8,0.213224)(1.85,0.16957)(1.9,0.182347)(1.95,0.172067)(2.0,0.1782)
			};
			
			% Legend 
			\legend {\Large Estimate, \Large CI};
			\end{axis}
			\end{tikzpicture}
	\end{subfigure}}
	\centering{}\resizebox{.32\linewidth}{!}{
		\begin{subfigure}[h]{.52\linewidth}
			\begin{tikzpicture} 
			\begin{semilogyaxis} 
			[
			log basis y=2,
			ymax=10,
			ymin=.125,
			xmin=1.05,
			xmax=2,
			title={\LARGE Symmetric},
			xlabel={\Large $\alpha$},
			ylabel={\LARGE $\E\hat\sigma_x$},
			width=9cm,
			height=5cm,
			axis on top=true,
			axis x line=bottom,
			axis y line=middle,
			legend style = {at={(1,.65)}, anchor=south east}
			]
			
			% Estimate
			\addplot[
			solid,
			color=black,
			]
			coordinates {
				(1.05,2.98598)(1.1,0.861246)(1.15,0.752034)(1.2,0.750415)(1.25,0.564297)(1.3,0.401376)(1.35,0.439275)(1.4,0.262984)(1.45,0.309363)(1.5,0.393642)(1.55,0.341917)(1.6,0.247718)(1.65,0.248733)(1.7,0.218251)(1.75,0.212116)(1.8,0.290216)(1.85,0.198565)(1.9,0.163439)(1.95,0.19663)(2.0,0.183676)
			};
			
			% Low
			\addplot[
			dotted,
			color=black,
			thick,
			]
			coordinates {
				(1.05,0.593684)(1.1,0.436451)(1.15,0.351834)(1.2,0.396243)(1.25,0.362794)(1.3,0.293672)(1.35,0.306939)(1.4,0.21902)(1.45,0.222725)(1.5,0.241723)(1.55,0.199534)(1.6,0.1928)(1.65,0.165719)(1.7,0.193732)(1.75,0.182787)(1.8,0.19743)(1.85,0.160629)(1.9,0.138997)(1.95,0.144844)(2.0,0.146013)
			};
			
			% Up
			\addplot[
			dotted,
			color=black,
			thick,
			]
			coordinates {
				(1.05,7.46912)(1.1,1.4852)(1.15,1.45231)(1.2,1.24926)(1.25,0.853479)(1.3,0.538631)(1.35,0.604994)(1.4,0.319087)(1.45,0.453252)(1.5,0.655716)(1.55,0.538893)(1.6,0.31712)(1.65,0.395309)(1.7,0.260101)(1.75,0.24864)(1.8,0.421509)(1.85,0.24289)(1.9,0.191977)(1.95,0.27596)(2.0,0.228492)
			};
			
			% Legend 
			\legend {\Large Estimate, \Large CI};
			\end{semilogyaxis}
			\end{tikzpicture}
	\end{subfigure}}
	\centering{}\resizebox{.32\linewidth}{!}{
		\begin{subfigure}[h]{.52\linewidth}
			\begin{tikzpicture} 
			\begin{semilogyaxis} 
			[
			log basis y=2,
			ymax=32,
			ymin=.125,
			xmin=1.05,
			xmax=2,
			title={\LARGE Spectrally Positive},
			xlabel={\Large $\alpha$},
			ylabel={\LARGE $\E\hat\sigma_x$},
			width=9cm,
			height=5cm,
			axis on top=true,
			axis x line=bottom,
			axis y line=middle,
			legend style = {at={(1,.65)}, anchor=south east}
			]
			
			% Estimate
			\addplot[
			solid,
			color=black,
			]
			coordinates {
				(1.05,10.5664)(1.1,6.38489)(1.15,3.1315)(1.2,2.71196)(1.25,1.72568)(1.3,1.17879)(1.35,1.10569)(1.4,0.756918)(1.45,0.870082)(1.5,0.528266)(1.55,0.571805)(1.6,0.409035)(1.65,0.36065)(1.7,0.344029)(1.75,0.251875)(1.8,0.240822)(1.85,0.257509)(1.9,0.198217)(1.95,0.189426)(2.0,0.185816)
			};
			
			% Low
			\addplot[
			dotted,
			color=black,
			thick,
			]
			coordinates {
				(1.05,8.06394)(1.1,4.38441)(1.15,2.51044)(1.2,2.07684)(1.25,1.40979)(1.3,1.02325)(1.35,0.933228)(1.4,0.64893)(1.45,0.688029)(1.5,0.447094)(1.55,0.472654)(1.6,0.327688)(1.65,0.290094)(1.7,0.286216)(1.75,0.213871)(1.8,0.206713)(1.85,0.208332)(1.9,0.16779)(1.95,0.163294)(2.0,0.148387)
			};
			
			% Up
			\addplot[
			dotted,
			color=black,
			thick,
			]
			coordinates {
				(1.05,13.5615)(1.1,9.10861)(1.15,3.96681)(1.2,3.48429)(1.25,2.10815)(1.3,1.35359)(1.35,1.30399)(1.4,0.883739)(1.45,1.08997)(1.5,0.622348)(1.55,0.700305)(1.6,0.51522)(1.65,0.4456)(1.7,0.411831)(1.75,0.296341)(1.8,0.27944)(1.85,0.319753)(1.9,0.23392)(1.95,0.218846)(2.0,0.236298)
			};
			
			% Legend 
			\legend {\Large Estimate, \Large CI};
			\end{semilogyaxis}
			\end{tikzpicture}
	\end{subfigure}}
	\caption{\footnotesize
	The graphs show estimates of $\E\hat\sigma_x$ with $95\%$ CIs based 
	on $N=4\times10^4$ samples as a function of $\alpha\in(1,2]$ for the spectrally 
	negative, symmetric and positive cases. The estimates are obtained using the 
	procedure in Subsection~\ref{subsec:hit-times} with the density function 
	$g:s\mapsto\delta(1+s)^{-1-\delta}$ for $\delta=2$. The computation of each 
	estimate (employing $N=4\times10^4$ samples) took approximately 290 seconds, 
	with little variation for different values of $\alpha$.
	}\label{fig:hit-time}
\end{figure}
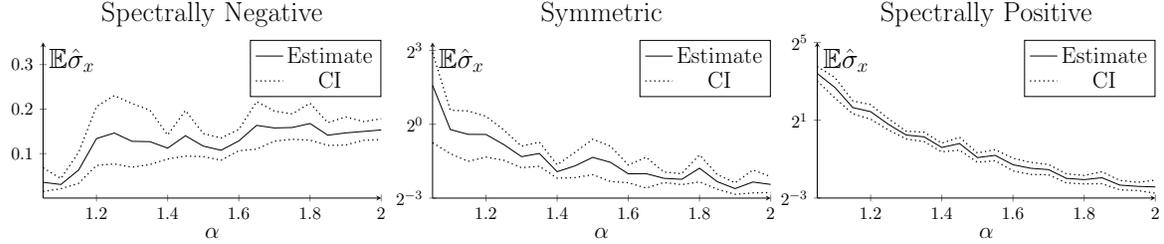

\subsubsection{Marginal of normalised stable excursions\label{subsec:num-excursion}}

Let $Z^\ex=(Z^\ex_t)_{t\in[0,1]}$ be a normalised stable excursion associated to the 
stable process $Z$ with parameters $(\alpha,\rho)$. By~\cite[Thm~3]{MR1465814}, 
if $Z$ has negative jumps (i.e. $\rho\in(0,1)$ \& $\alpha\leq 1$ or 
$\rho\in(1-\frac{1}{\alpha},\frac{1}{\alpha}]$ \& $\alpha>1$), the laws of 
$(Z^\me_t)_{t\in[0,1)}$ and $(Z^\ex_t)_{t\in[0,1)}$ are \emph{equivalent}: 
$\P(Z^\ex\in A)=\E[(Z^\me_1)^{-\alpha}\1_{\{Z^\me\in A\}}]/\E[(Z^\me_1)^{-\alpha}]$ 
for any measurable set $A$ in the Skorokhod space $\D[0,1)$~\cite[Ch.~3]{MR1700749}. 
We remark that $Z_1^\me=Z_{1-}^\me$ a.s. and $\E[(Z^\me_1)^{-\alpha}]<\infty$ since 
$\alpha<1+\alpha\rho$ and $\E[(Z_1^\me)^{\gamma}]<\infty$ for all 
$\gamma\in(-1-\alpha\rho,\alpha)$~\cite[Thm~1]{MR2599201}. As an illustration of 
Algorithm~\ref{alg:fd_meander}, we now present a Monte Carlo estimation of 
$\E Z^\ex_{1/2}$ by applying the procedure of Subsection~\ref{subsec:num-marginal} 
for the expectations on the right side of 
$\E Z^\ex_{1/2}=\E[Z^\me_{1/2}(Z^\me_1)^{-\alpha}]/\E[(Z^\me_1)^{-\alpha}]$.
\begin{figure}[H]
	\centering{}\resizebox{.49\linewidth}{!}{
		\begin{subfigure}[h]{.66\linewidth}
			\begin{tikzpicture} 
			\begin{axis} 
			[
			ymin=0,
			ymax=3,
			xmin=1.05,
			xmax=1.95,
			title={\Large Spectrally Negative},
			xlabel={\Large $\alpha$},
			ylabel={\Large $\E Z^\ex_{1/2}$},
			width=11cm,
			height=4.5cm,
			axis on top=true,
			axis x line=middle,
			axis y line=middle,
			legend style = {at={(.5,.6)}, anchor=south east}
			]
			
			% Estimate
			\addplot[
			solid,
			color=black,
			]
			coordinates {
				(1.05,0.530383)(1.1,0.559883)(1.15,0.6334)(1.2,0.656438)(1.25,0.704291)(1.3,0.757524)(1.35,0.734234)(1.4,0.794963)(1.45,0.931519)(1.5,0.980995)(1.55,0.921686)(1.6,0.9532)(1.65,0.953044)(1.7,1.10152)(1.75,1.4311)(1.8,1.01979)(1.85,1.04174)(1.9,1.11841)(1.95,1.14631)
			};
			
			% Unbiased
%			\addplot[
%			solid,
%			color=black,
%			]
%			coordinates {
%				(1.05,0.557153)(1.1,0.693992)(1.15,0.649683)(1.2,0.721026)(1.25,0.668438)(1.3,0.444905)(1.35,0.729363)(1.4,0.849506)(1.45,0.923694)(1.5,0.746332)(1.55,0.85076)(1.6,0.922178)(1.65,0.979721)(1.7,1.00573)(1.75,0.374069)(1.8,1.15357)(1.85,0.996532)(1.9,1.00553)(1.95,0.994369)
%			};
			
			% Up
			\addplot[
			dotted,
			color=black,
			thick,
			]
			coordinates {
				(1.05,0.698717)(1.1,0.863243)(1.15,0.907723)(1.2,0.993843)(1.25,1.06037)(1.3,1.4465)(1.35,0.940044)(1.4,1.26099)(1.45,1.47691)(1.5,1.60679)(1.55,1.4394)(1.6,1.43149)(1.65,2.04983)(1.7,2.29369)(1.75,2.3151)(1.8,1.72999)(1.85,1.39962)(1.9,1.6969)(1.95,1.59388)
			};
			
			% Low
			\addplot[
			dotted,
			color=black,
			thick,
			]
			coordinates {
				(1.05,0.406372)(1.1,0.367867)(1.15,0.444498)(1.2,0.434258)(1.25,0.469186)(1.3,0.39587)(1.35,0.57038)(1.4,0.512111)(1.45,0.443541)(1.5,0.392365)(1.55,0.587804)(1.6,0.625698)(1.65,0.519761)(1.7,0.492312)(1.75,0.405048)(1.8,0.603251)(1.85,0.759939)(1.9,0.727601)(1.95,0.824329)
			};
			
			% Legend 
			\legend {\large Estimate, \large CI};
			\end{axis}
			\end{tikzpicture}
	\end{subfigure}}
	\centering{}\resizebox{.49\linewidth}{!}{
		\begin{subfigure}[h]{.66\linewidth}
			\begin{tikzpicture} 
			\begin{axis} 
			[
			ymax=3,
			ymin=0,
			xmin=1.05,
			xmax=1.95,
			title={\Large Symmetric},
			xlabel={\Large $\alpha$},
			ylabel={\Large $\E Z^\ex_{1/2}$},
			width=11cm,
			height=4.5cm,
			axis on top=true,
			axis x line=middle,
			axis y line=middle,
			legend style = {at={(.5,.6)}, anchor=south east}
			]
			
			% Estimate
			\addplot[
			solid,
			color=black,
			]
			coordinates {
				(1.05,0.751573)(1.1,0.742016)(1.15,0.763336)(1.2,0.781047)(1.25,0.766908)(1.3,0.835764)(1.35,0.812025)(1.4,0.8917)(1.45,0.945328)(1.5,0.910243)(1.55,0.93589)(1.6,0.944191)(1.65,0.961377)(1.7,1.06192)(1.75,1.03209)(1.8,1.11497)(1.85,1.05355)(1.9,1.09299)(1.95,1.1213)
			};
			
			% Unbiased
%			\addplot[
%			solid,
%			color=black,
%			]
%			coordinates {
%				(1.05,0.781558)(1.1,0.828226)(1.15,0.561618)(1.2,0.748417)(1.25,0.815429)(1.3,0.722041)(1.35,0.885002)(1.4,0.998002)(1.45,0.960937)(1.5,0.90069)(1.55,0.850484)(1.6,0.890149)(1.65,0.87171)(1.7,0.852177)(1.75,0.967649)(1.8,0.987529)(1.85,0.923956)(1.9,0.764397)(1.95,1.06158)
%			};
			
			% Low
			\addplot[
			dotted,
			color=black,
			thick,
			]
			coordinates {
				(1.05,0.806632)(1.1,0.89076)(1.15,0.899819)(1.2,1.0117)(1.25,0.921736)(1.3,0.984712)(1.35,1.04617)(1.4,0.996687)(1.45,1.47601)(1.5,1.23605)(1.55,1.16515)(1.6,1.15455)(1.65,1.46503)(1.7,1.58063)(1.75,1.79306)(1.8,1.85624)(1.85,1.68168)(1.9,1.46542)(1.95,2.40068)
			};
			
			% Up
			\addplot[
			dotted,
			color=black,
			thick,
			]
			coordinates {
				(1.05,0.699596)(1.1,0.606503)(1.15,0.649044)(1.2,0.602583)(1.25,0.636219)(1.3,0.710897)(1.35,0.622683)(1.4,0.800682)(1.45,0.608943)(1.5,0.676962)(1.55,0.746059)(1.6,0.768734)(1.65,0.635523)(1.7,0.723679)(1.75,0.713566)(1.8,0.64658)(1.85,0.651235)(1.9,0.811791)(1.95,0.685003)
			};
			
			% Legend 
			\legend {\large Estimate, \large CI};
			\end{axis}
			\end{tikzpicture}
	\end{subfigure}}
	\caption{\footnotesize
	The pictures show the quotient of the Monte Carlo estimates of the expectations 
	on the right side of $\E Z^\ex_{1/2}=\E[Z^\me_{1/2}(Z^\me_1)^{-\alpha}]/
	\E[(Z^\me_1)^{-\alpha}]$ as a function of the stability parameter $\alpha\in(1,2)$ for 
	$N=4\times10^4$ samples. 
	%$\alpha$ is restricted to the interval $(1,2)$ to ensure finite expectations. 
	Computing each estimate (for $N=4\times10^4$) took approximately 160.8 
	(resp. 123.1) seconds in the spectrally negative (resp. symmetric) case, with little 
	variation in $\alpha$. }\label{fig:marginal_excursion}
\end{figure}

As before, we use the fixed precision of $\varepsilon=2^{-32}$. 
%Recall, however, 
%that adaptive precision is required by line~12 in Algorithm~\ref{alg:fd_meander}. 
%Hence, we run the algorithms until both conditions are satisfied: the precision is 
%of level at least $\varepsilon$ and the event in line~12 of 
%Algorithm~\ref{alg:fd_meander} is detected. 
The CIs are naturally constructed from the bootstrapping CIs (as in 
Subsection~\ref{subsec:num-marginal} above) for each of the expectations 
$\E[Z^\me_{1/2}(Z^\me_1)^{-\alpha}]$ and $\E[(Z^\me_1)^{-\alpha}]$ and combined
to construct a CI for $\E Z^\ex_{1/2}$.
%Indeed, the upper (resp. lower) end of the CI on 
%$\E Z^\ex_{1/2}$ is obtained by dividing the upper (resp. lower) end of the CI on 
%$\E[Z^\me_{1/2}(Z^\me_1)^{-\alpha}]$ by the lower (resp. upper) end of the CI on 
%$\E[(Z^\me_1)^{-\alpha}]$.

\section{Proofs and technical results\label{sec:proofs}}
\subsection{Approximation of piecewise linear convex functions\label{subsec:approx_linear}}
The main aim of the present subsection is to prove 
Proposition~\ref{prop:Sandwich-CM} and Proposition~\ref{prop:General_UpLo}, key 
ingredients of the algorithms in Section~\ref{sec:algorithms}. Throughout this 
subsection, we will assume that $f$ is a continuous piecewise linear  finite variation 
function on a compact interval $[a,b]$ with at most countably many faces. More 
precisely, there exists a set consisting of  $N\in\N\cup\{\infty\}$ pairwise disjoint 
nondegenerate subintervals $\{(a_n,b_n):n\in \mZ^{N+1}_1\}$ of $[a,b]$ such that 
$\sum_{n=1}^{N}(b_n-a_n)=b-a$, $f$ is linear on each $(a_n,b_n)$, and 
$\sum_{n=1}^{N} |f(b_n)-f(a_n)|<\infty$ (if $N=\infty$ we set 
$\mZ^\infty=\mathbb{Z}$ and thus $\mZ_1^\infty=\mathbb{N}$; 
recall also $\mZ^n=\{k\in\mathbb{Z}:k<n\}$ and 
$\mZ^n_m=\mZ^n\setminus\mZ^m$ for $n,m\in\mathbb{Z}$). A face of $f$, 
corresponding to a subinterval $(a_n,b_n)$, is given by the pair $(l_n,h_n)$, 
where $l_n=b_n-a_n>0$ is its length and $h_n=f(b_n)-f(a_n)\in\R$ its height.
Consequently, its slope equals $h_n/l_n$ and the following representation 
holds (recall $x^+=\max\{0,x\}$ for $x\in\R$):
\begin{equation}\label{eq:pw_linear_rep}
f(t)=f(a)+\sum_{n=1}^N h_n\min\{(t-a_n)^+/l_n,1\}, \qquad t\in[a,b].
\end{equation}
The number $N$ 
in representation~\eqref{eq:pw_linear_rep}
is not unique in general as any face may be subdivided into two faces with the same slope.
Moreover, for a fixed $f$ and $N$, the set of intervals $\{(a_n,b_n):n\in \mZ^{N+1}_1\}$ need not be unique. 
Furthermore we stress that the sequence of faces 
in~\eqref{eq:pw_linear_rep} does not necessarily respect the chronological ordering. Put differently, the sequence
$(a_n)_{n\in\mZ^{N+1}_1}$
need not be
increasing.
We start with an elementary but useful result. 
%corresponding to the intervals 
%$\{(a_n,b_n):n\in \mZ^{N+1}_1\}$ is not given in any specific order and 
%The sequence of faces is uniquely defined by $f$ and the sequence of 
%intervals $\{(a_n,b_n):n\in\mZ_1^{N+1}\}$. However, the sequence of faces and the function $f$ do not 
%uniquely define a sequence of intervals (e.g. two adjacent faces with the same 
%slope may be exchanged without altering $f$). 
%Thus we assume in 
%addition that, if an interval $[a',b']$ exists such that the slopes corresponding to 
%the face of any two subintervals $(a_k,b_k),(a_j,b_j)\subset [a',b']$ is the same 
%$h_k/l_k=h_j/l_j$, then the chronological order respects the index order: 
%$a_k<a_j$ if $k<j$ and $a_k>a_j$ otherwise. 
%Thus we assume in addition that $N_f$ is the smallest element of 
%$\N\cup\{\infty\}$ satisfying the requirements above.
%For $g:[a,b]\to\R$, define $f\geq g$ if $f(t)\geq g(t)$ for all $t\in[a,b]$. 
%The function $g$ is 
%strictly smaller than $f$ if $f\geq g$ and $g\not\geq f$.

\begin{lem}\label{lem:minimal-faces}
	Let $f:[a,b]\to\R$ be a continuous piecewise linear function with $N<\infty$ 
	faces $(l_k,h_k)$, $k\in\mZ^{N+1}_1$. 
	%given in chronological order, i.e., as they appear in the function $t\mapsto f(t)$. 
	Let $\mathcal{K}$ be the  set of 
	piecewise linear functions 
	$f_\pi:[a,b]\to\R$ 
	with initial value $f(a)$, obtained from 
	$f$ by sorting its faces according to a bijection $\pi:\mZ^{N+1}_1 \to\mZ^{N+1}_1$.
	More precisely, defining $a_k^\pi=a+\sum_{j\in \mZ^k_1} l_{\pi(j)}$ 
	for any $k\in\mZ^{N+1}_1$, $f_\pi$ in $\mathcal{K}$ is given by 
	\[f_\pi(t)=f(a)+\sum_{k=1}^{N} h_{\pi(k)}
	\min\big\{\big(t-a_k^\pi\big)^+/l_{\pi(k)},	1\big\},
	\qquad t\in[a,b].\]
	If $\pi^*:\mZ^{N+1}_1 \to\mZ^{N+1}_1$ sorts the faces by increasing slope,
	$h_{\pi^*(k)}/l_{\pi^*(k)}\leq h_{\pi^*(k+1)}/l_{\pi^*(k+1)}$ for $k\in\mZ^N_1$, 
	then $f_{\pi^*}$ is the unique convex function in $\mathcal{K}$ and satisfies $f\geq f_{\pi^*}$ pointwise. 
\end{lem}

\begin{proof}
	Relabel the faces $(l_k,h_k)$, $k\in\mZ^{N+1}_1$ of $f$ so that they are listed 
	in the chronological order, i.e. as they appear in the function $t\mapsto f(t)$ 
	with increasing $t$. If every pair of consecutive faces of $f$ is ordered by slope (i.e.
	$h_i/l_i\leq h_{i+1}/l_{i+1}$ for all $i\in\mZ^N_1$),
	then $f$ is convex and $f=f_{\pi^*}$.
	Otherwise, two consecutive faces of $f$ %, say $(l_i,h_i)$ and 
	%$(l_{i+1},h_{i+1})$ for some $i\in\mZ^N_1$, satisfy $h_i/l_i>h_{i+1}/l_{i+1}$. 
	satisfy $h_i/l_i>h_{i+1}/l_{i+1}$ for some $i\in\mZ^N_1$.
	Swapping the two faces yields a smaller function $f_{\pi_1}$, see 
	Figure~\ref{fig:one-swap}. Indeed, after the swap, the functions $f$ and 
	$f_{\pi_1}$ coincide on the set 
	$[a,a+\sum_{k\in\mZ^i_1}l_k]\cup[a+\sum_{k\in \mZ^{i+2}_1}l_k,b]$. In the 
	interval $[a+\sum_{k\in \mZ^i_1}l_k,a+\sum_{k\in \mZ^{i+2}_1}l_k]$, the 
	segments form a parallelogram whose lower (resp. upper) sides belong to the 
	graph of $f_{\pi_1}$ (resp. $f$).
		
	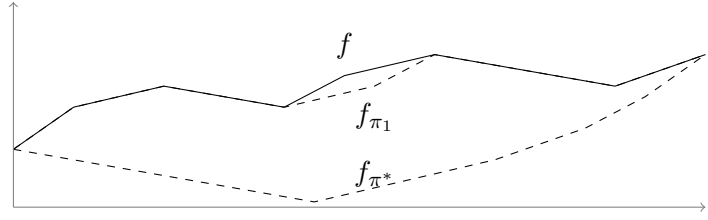
\begin{figure}[H]
		\begin{tikzpicture}
		% f
		\draw[solid, mark=o] (0,5*.14) %-- ++(5*.4,2*.14) -- ++(3*.4,-1*.14) 
		-- ++(2*.4,4*.14) -- ++(3*.4,2*.14) -- ++(4*.4,-2*.14) -- ++(2*.4,3*.14) 
		-- ++(3*.4,2*.14) -- ++(6*.4,-3*.14) --++ (3*.4,3*.14)
		node[above] {};
		% g
		\draw[dashed, mark=+] (0,5*.14) %-- ++(5*.4,2*.14) -- ++(3*.4,-1*.14) 
		-- ++(2*.4,4*.14) -- ++(3*.4,2*.14) -- ++(4*.4,-2*.14) -- ++(3*.4,2*.14)
		-- ++(2*.4,3*.14) -- ++(6*.4,-3*.14)  --++ (3*.4,3*.14)
		node[below] {};
		% f^*
		\draw[dashed, mark=+] (0,5*.14) %-- ++(5*.4,2*.14) -- ++(3*.4,-1*.14) 
		-- ++(4*.4,-2*.14) -- ++(6*.4,-3*.14) -- ++(3*.4,2*.14) -- ++(3*.4,2*.14)
		-- ++(3*.4,3*.14) -- ++(2*.4,3*.14)  --++ (2*.4,4*.14)
		node[below] {};
		
		% labels f and g
		\node[circle, scale=0.5, label=above :{$f$}] at (11*.4,12*.14) {}{};
		\node[circle, scale=0.5, label=below :{$f_{\pi_1}$}] at (12*.4,11*.14) {}{};
		\node[circle, scale=0.5, label=below :{$f_{\pi^*}$}] at (12*.4,5.5*.14) {}{};
		
		% Draw the axis
		\draw [thin, draw=gray, ->] (0,-.5*.14) -- (23*.4,-.5*.14);
		\draw [thin, draw=gray, ->] (0,-.5*.14) -- (0,19*.14);
		
		\end{tikzpicture}
		\caption{\footnotesize Swapping two consecutive and unsorted faces of $f$.
		}\label{fig:one-swap}
	\end{figure}

Applying the argument in the preceding paragraph to $f_{\pi_1}$, we either have $f_{\pi_1}=f_{\pi^*}$ or we may construct $f_{\pi_2}$, 
 which is  strictly smaller than $f_{\pi_1}$
on a non-empty open subinterval of  $[a,b]$,
satisfying $f_{\pi_1}\geq f_{\pi_2}$.
Since 
the set $\mathcal{K}$ is finite, this procedure 
necessarily terminates at $f_{\pi^*}$ after finitely many steps, 
implying $f\geq f_{\pi^*}$.
Since any convex function in $\mathcal{K}$ must have a nondecreasing derivative a.e.,
it has to be equal to $f_{\pi^*}$ and the lemma follows. 
\end{proof}

A natural approximation of a piecewise linear convex function $f$ can be 
constructed from the first $n<N$ faces of $f$ by filling in the remainder with a 
horizontal face. More precisely, for any $n\in\mZ_1^N$ let $f_n$ be the 
piecewise linear convex function with $f_n(a)=f(a)$ and faces 
$\{(l_k,h_k):k\in\mZ_1^{n+1}\}\cup\{(\ti{l}_n,0)\}$, where 
$\ti{l}_n = \sum_{k=n+1}^N l_k$.
By Corollary~\ref{cor:Conv_app}, the inequality 
$\|f-f_n\|_\infty\leq\max\{\sum_{k=n+1}^N h_k^+, \sum_{k=n+1}^N h_k^-\}$ holds, 
where $\|g\|_\infty=\sup_{t\in[a,b]}|g(t)|$ denotes the supremum norm 
of a function $g:[a,b]\to\R$ and $x^-=\max\{0,-x\}$ for any $x\in\R$.
If $f=C(X)$ is the convex minorant 
of a L\'evy process $X$, both tail sums decay to zero geometrically 
fast~\citep[Thms~1~\&~2]{LevySupSim}. However, it appears to be difficult directly to 
obtain almost sure bounds on the maximum of the two (dependent!) sums, which 
would be necessary for an $\varepsilon$SS algorithm for $C(X)$. We proceed by 
``splitting'' the problem as follows.   

The slopes of the faces of a piecewise linear convex function $f$ may form an 
unbounded set. In particular, the slopes of the faces accumulating at  $a$ could 
be arbitrarily negative, making it impossible to construct a piecewise linear 
lower bound with \textit{finitely many faces} starting at $f(a)$. 
In Proposition~\ref{prop:Sandwich-CM} we focus on functions without faces of 
negative slope (as is the case with the convex minorants of pre- and post-minimum 
processes and L\'evy meanders in Proposition~\ref{prop:pre_post} and 
Theorem~\ref{thm:meander-minorant}), which makes it easier to isolate the 
errors. We deal with the general case in Proposition~\ref{prop:General_UpLo} below.

\begin{prop}\label{prop:Sandwich-CM}
Let $f:[a,b]\to\R$ be a piecewise linear convex function with $N=\infty$ faces 
$(l_n,h_n)$, $n\in\N$, satisfying $h_n\geq 0$ for all $n$. Let the constants $(c_n)_{n\in\N}$ 
satisfy $c_n\geq \sum_{k=n+1}^\infty h_k$ and $c_{n+1}\leq c_n-h_{n+1}$ for $n\in\N$. 
There exist unique piecewise linear convex functions $f_n^\da$ and $f_n^\ua$ on $[a,b]$,  
satisfying $f_n^\da(a)=f_n^\ua(a)=f(a)$, with faces 
$\{(l_k,h_k):k\in\mZ_1^{n+1}\}\cup\{(\ti{l}_n,0)\}$ and 
$\{(l_k,h_k):k\in\mZ_1^{n+1}\}\cup\{(\ti{l}_n,c_n)\}$, respectively, 
where $\ti{l}_n = \sum_{k=n+1}^\infty l_k$.
Moreover, for all $n\in\N$ 
the following holds: 
\nf{(a)}~$f_{n+1}^\da\geq f_n^\da$, 
\nf{(b)}~$f_n^\ua\geq f_{n+1}^\ua$, 
\nf{(c)}~$f_n^\ua\geq f\geq f_n^\da$ and 
\nf{(d)}~$\|f_n^\ua-f_n^\da\|_\infty=f_n^\ua(b)-f_n^\da(b)=c_n$. 
\end{prop}

\begin{rem}\label{rem:propSandwich}
(i) Note that if $c_n\to0$ as $n\to\infty$, Proposition~\ref{prop:Sandwich-CM} 
implies the sequences $(f_n^\da)_{n\in\N}$ and $(f_n^\ua)_{n\in\N}$ converge 
uniformly and monotonically to $f$. \\
(ii) Note that the lower bounds 
$f_n^\da$ do not depend on $c_n$ and satisfy 
$\|f-f_n^\da\|_\infty = \sum_{k=n+1}^\infty h_k$.
Indeed, setting $c_n= \sum_{k=n+1}^\infty h_k$
(for all $n\in\N$) and applying  Proposition~\ref{prop:Sandwich-CM}(c) \& (d) implies
 \[
\sum_{k=n+1}^\infty h_k = f(b)-f_n^\da(b) \leq \|f-f_n^\da\|_\infty 
\leq \|f_n^\ua-f_n^\da\|_\infty = \sum_{k=n+1}^\infty h_k.
\]
(iii) Given constants $c'_n$, $n\in\N$, satisfying $c'_n\geq\sum_{k=n+1}^\infty h_k$, 
one may construct constants $c_n$, satisfying $c_n\geq 
\sum_{k=n+1}^\infty h_k$ and $c_{n+1}\leq c_n-h_{n+1}$  for all $n\in\N$ as follows: 
set $c_1=c'_1\geq\sum_{k=2}^\infty h_k$ and 
$c_{n+1}=\min\{c'_{n+1},c_n-h_{n+1}\}$ for $n\in\N$. 
The condition $c_{n+1}\leq c_n-h_{n+1}$  is only necessary for part (b), but is assumed 
throughout Proposition~\ref{prop:Sandwich-CM} as it simplifies the proof of (c).\\
(iv) The function $f$ in Proposition~\ref{prop:Sandwich-CM} may have infinitely 
many faces in a neighbourhood of any point in $[a,b]$. If this occurs at $b$, the 
corresponding slopes may be arbitrarily large.\\ 
(v) Proposition~\ref{prop:Sandwich-CM} assumes that the slopes of the faces 
of $f$ are nonnegative. This condition can be relaxed to all the slopes being 
bounded from below by some constant $c\leq 0$, in which case we use the 
auxiliary faces $(\ti{l}_n,c\ti{l}_n)$ and $(\ti{l}_n,c\ti{l}_n+c_n)$ in 
the construction of $f_n^\ua$ and  $f_n^\da$.\\
(vi) If $n=0$ and 
$c_0\geq c_1+h_1$, then $\ti l_0=b-a$ and the functions 
$f_0^\da:t\mapsto f(a)$ and $f_0^\ua:t\mapsto f(a)+(t-a)c_0$, for $t\in[a,b]$, 
satisfy the conclusion of Proposition~\ref{prop:Sandwich-CM} (with $n=0$).
Moreover, Proposition~\ref{prop:Sandwich-CM} extends easily to the case 
when
$f$ has finitely many faces. 
\end{rem}

\begin{proof}
Note that the set $\mathcal{K}$ in Lemma~\ref{lem:minimal-faces} depends only on the value 
of the function $f$ at $a$ and the set of its faces. Define the set of functions
$\mathcal{K}_n^\ua$ (resp. $\mathcal{K}_n^\da$) by the set of 
faces  
$\{(l_k,h_k):k\in\mZ_1^{n+1}\}\cup\{(\ti{l}_n,c_n)\}$
(resp. 
$\{(l_k,h_k):k\in\mZ_1^{n+1}\}\cup\{(\ti{l}_n,0)\}$) and the starting value $f(a)$ as in Lemma~\ref{lem:minimal-faces}.
Let
$f_n^\ua$ (resp. $f_n^\da$) be the
unique convex function 
in $\mathcal{K}_n^\ua$ (resp. $\mathcal{K}_n^\da$) 
constructed in  Lemma~\ref{lem:minimal-faces}.

For each $n\in\N$, $f_{n+1}^\da$ and $f_n^\da$ share all but a single face, 
which has nonnegative slope in $f_{n+1}^\da$ and a horizontal slope in $f_n^\da$. 
Hence, replacing this face in $f_{n+1}^\da$ with a horizontal one yields a smaller 
(possibly non-convex) continuous piecewise linear function $\phi$. Applying  
Lemma~\ref{lem:minimal-faces} to $\phi$ produces a convex function $\phi_*$
satisfying $f_{n+1}^\da\geq \phi_*$ and  $\phi_*(a)=f(a)$
with faces equal to those of $f_n^\da$. Since $f_n^\da$ is also convex and
satisfies $f_n^\da(a)=f(a)$, we must have $\phi_*=f_n^\da$ implying the inequality in~(a).

To establish~(b), construct a function $\psi$ by replacing the face $(\ti{l}_n,c_n)$ 
in $f_n^\ua$ with the faces $(\ti{l}_{n+1},c_{n+1})$ and $(l_{n+1},h_{n+1})$ sorted 
by increasing slope (note $l_{n+1}+\ti{l}_{n+1}=\ti{l}_n$). More precisely, 
if $(a',a'+\ti l_n)\subset [a,b]$ is the interval corresponding to the face 
$(\ti{l}_n,c_n)$ in $f_n^\ua$, for $t\in[a,b]$ we set
\[
\varphi(t)= 
\begin{cases}
h_{n+1}\min\{(t-a')^+/l_{n+1},1\} + c_{n+1}\min\{(t-a'-l_{n+1})^+/\ti l_{n+1},1\}; 
& \frac{h_{n+1}}{l_{n+1}} \leq \frac{c_{n+1}}{\ti l_{n+1}}, \\
c_{n+1}\min\{(t-a')^+/\ti l_{n+1},1\} + h_{n+1}\min\{(t-a'-\ti l_{n+1})^+/l_{n+1},1\}; 
& \frac{h_{n+1}}{l_{n+1}} > \frac{c_{n+1}}{\ti l_{n+1}}.
\end{cases}
\]
By the inequality $c_n\geq c_{n+1}+h_{n+1}$, the graph of $\varphi$ on the 
interval $(a',a'+\ti l_n)$ is below the line segment $t\mapsto c_n(t-a')/\ti l_n$. 
We then define the continuous piecewise linear  function
\[
\psi(t)= 
\begin{cases}
f_n^\ua(t) &  t\in [a,a'],\\
f_n^\ua(a') + \varphi(t)  & t\in(a',a'+\ti l_n),\\
f_n^\ua(t)+h_{n+1}+c_{n+1}-c_n & t\in[a'+\ti l_n,b],
\end{cases}
\]
which clearly satisfies $f_n^\ua\geq \psi$.
Furthermore, the faces of $\psi$ coincide with those of $f_{n+1}^\ua$. Thus, applying 
Lemma~\ref{lem:minimal-faces} to $\psi$ yields $\psi\geq f_{n+1}^\ua$, implying~(b).

Recall that a face $(l_k,h_k)$ of $f$ satisfies 
$l_k=b_k-a_k$ and $h_k=f(b_k)-f(a_k)$ for any $k\in\N$, where 
$(a_k,b_k)\subset [a,b]$. Let $g_n$ be the piecewise linear function defined by 
truncating the series in~\eqref{eq:pw_linear_rep} at $n$: 
\begin{equation*}%\label{eq:gn_proof}
g_n(t)=f(a)+\sum_{k=1}^n h_k\min\{(t-a_k)^+/l_k,1\},\qquad t\in[a,b].
\end{equation*}
By construction, $g_n+\ti{h}_n\geq f\geq g_n$ and $g_n(b)+\ti{h}_n=f(b)$, 
where $\ti{h}_n=\sum_{k=n+1}^\infty h_k$, implying $\|f-g_n\|_\infty=
\ti{h}_n$. Since the set $[a,b]\setminus\bigcup_{k=1}^n[a_k,b_k]$ consists of 
at most $n+1$ disjoint intervals, a representation of $g_n$ exists with at most 
$2n+1$ faces. Moreover the slopes of $g_n$ over all those intervals are equal to 
zero. Sorting the faces of $g_n$ by increasing slope yields $f_n^\da$. By 
Lemma~\ref{lem:minimal-faces}, the second inequality in (c), 
$f\geq g_n\geq f_n^\da$, holds.

We now establish the inequality $f_n^\ua\geq f$ for all $n\in\N$. First note that 
$f_n^\ua\geq f_n^\da$ for any $n\in\N$. Indeed, replacing the face 
$(\ti{l}_n,c_n)$ in $f_n^\ua$ with $(0,c_n)$ yields a smaller function with the 
same faces as $f_n^\da$. Hence the convexity of $f_n^\da$ and 
Lemma~\ref{lem:minimal-faces} imply the inequality $f_n^\ua\geq f_n^\da$. 
By~(b), $f_n^\ua\geq \lim_{k\to\infty} f_k^\da$.
Hence, part~(c) follows if we show that $\lim_{k\to\infty}f_k^\da=f$ pointwise.

For $k\in\N$ and $n\geq k$, define $a'_{k,n}$ and $a'_k$ by the formulae:
\[\begin{split}
a'_{k,n}&=a+\sum_{j\in\mZ_1^k} l_j \cdot \1_{\{h_j/l_j\}}(h_k/l_k)+
\sum_{j\in\mZ_1^{n+1}} l_j \cdot \1_{(h_j/l_j,\infty)}(h_k/l_k)+
\sum_{j\in\mZ_{n+1}^\infty} l_j, \\
a'_k&= a+\sum_{j\in\mZ_1^k} l_j \cdot \1_{\{h_j/l_j\}}(h_k/l_k)+
\sum_{j\in\mZ_1^\infty} l_j \cdot \1_{(h_j/l_j,\infty)}(h_k/l_k).
\end{split}\]
It is clear that $a'_{k,n}\searrow a'_k$ as $n\to\infty$.
Moreover, for $t\in[a,b]$, we have
\[%\begin{split}
f_n^\da(t)=f(a)+\sum_{k=1}^n h_k\min\{(t-a'_{k,n})^+/l_k,1\},\quad
f(t)=f(a)+\sum_{k=1}^\infty h_k\min\{(t-a'_k)^+/l_k,1\}.
%\end{split}
\]
In other words, for any $k\in\N$ and $k\geq n$, $a'_{k,n}$ (resp. $a'_k$) is 
the left endpoint of the interval corresponding to the face $(l_k,h_k)$ in a 
representation of $f_n^\ua$ (resp. $f$). Thus, for fixed $t\in[a,b]$, the terms 
$h_k\min\{(t-a'_{k,n})^+/l_k,1\}$ are a monotonically increasing sequence with 
limit $h_k\min\{(t-a'_k)^+/l_k,1\}$ as $n\to\infty$. By the monotone convergence 
theorem applied to the counting measure we deduce that $f_n^\da\to f$ pointwise, 
proving~(c).

Since $\|f_n^\ua-f_n^\da\|_\infty\geq f_n^\ua(b)-f_n^\da(b)=c_n$, claim in~(d) 
follows if  we prove the reverse inequality. Without loss of generality, the first 
face of $f_n^\da$ in the chronological order is $(\ti{l}_n,0)$. Replace this face 
with $(\ti{l}_n,c_n)$ to obtain a piecewise linear function $u_n(t)=
f(a)+c_nt/\ti{l}_n\1_{[0,\ti{l}_n]}(t)+(f_n^\da(t)+c_n)\1_{(\ti{l}_n,1]}(t)$. 
Since $u_n$ has the same faces as $f_n^\ua$, Lemma~\ref{lem:minimal-faces} 
implies $u_n\geq f_n^\ua$. Hence (d) follows from (c):
$\|f_n^\ua-f_n^\da\|_\infty\leq\|u_n-f_n^\da\|_\infty=c_n=f_n^\ua(b)-f_n^\da(b)$.
\end{proof}

%From the previous results it should be clear that the natural objective of any 
%simulation algorithm of $C(X^{\me,T})$ that aims to give arbitrarily accurate 
%samples is to simulate the faces of the convex minorant jointly with some
%(random) sequence $(c_n)_{n\in\N}$. We are unable to do this in general. 
%However, the Markov chain property from Theorem~\ref{thm:meander-minorant} 
%below renders this possible for stable meanders $C(Z^\me)$ (and by 
%Corollary~\ref{cor:General_UpLo}, for stable processes). The sequence 
%$(c_n)_{n\in\N}$ will be given by $c_n=L^{1/\alpha}_{-n}D_{-n}\geq
%\sum_{k=n+1}^\infty \xi_k^\me$ for $n\in\N$, where the inequality follows from 
%the domination $D_{-n}\geq M_{-n}=L_{-n}^{-1/\alpha}\sum_{k=n+1}^\infty
%\xi_k^\me$ (see~\eqref{eq:DominatingProcess} below for the definition of 
%$(D_n)_{n\in\mZ^1}$ and Theorem~\ref{thm:CMSM} below for that of 
%$(L_n)_{n\in\mZ^1}$).

Define 
$\tau_{[a,b]}(g)=\inf\{t\in[a,b]:\min\{g(t),g(t-)\}=\inf_{r\in[a,b]}g(r)\}$
for any \cadlag~function $g:[a,b]\to\R$. 
Consider now the problem of sandwiching a convex function $f$ with both positive and negative slopes.
Splitting $f$ into two convex functions, the pre-minimum $f^\la$ and post-minimum $f^\ra$ (see 
Proposition~\ref{prop:General_UpLo} for definition), it is natural to  apply Proposition~\ref{prop:Sandwich-CM} directly 
to each of them and attempt to construct the bounds for $f$ by concatenating the two sandwiches.
However, 
this strategy does not yield an upper and lower bounds for $f$ for the following reason:
%, in the simulation problems considered here, 
since we may \textit{not} assume to have access to the minimal value $f(s)$ of the function $f$,   
the concatenated sandwich  cannot be anchored at $f(s)$
(note that we may and do assume that we know the time $s=\tau_{[a,b]}(f)$ of the minimum of $f$). 
Proposition~\ref{prop:General_UpLo} is the analogue of 
Proposition~\ref{prop:Sandwich-CM} for general piecewise linear convex functions.
%We stress however that a direct application of Proposition~\ref{prop:Sandwich-CM}
%to pre- and post-minimum parts of the convex fuction is not 
%we show how to obtain an 
%analogous result for piecewise linear convex functions with slopes of both signs.  
%The crucial point in Corollary~\ref{cor:General_UpLo} is the ability to split 
%the function at its minimum, as in Proposition~\ref{prop:pre_post}. An impediment to 
%extending Proposition~\ref{prop:Sandwich-CM} directly to this case is that the value 
%of the function at this minimum is \emph{not} accessible. 

\begin{prop}\label{prop:General_UpLo} 
Let $f$ be a piecewise linear convex function on $[a,b]$ with infinitely many faces 
of both signs. Set $s=\tau_{[a,b]}(f)$ and let $f^\la:t\mapsto f(s-t)-f(s)$ and 
$f^\ra:t\mapsto f(s+t)-f(s)$ be the pre- and post-minimum functions, defined on 
$[0,s-a]$ and $[0,b-s]$ with sets of faces 
$\{(l_n^\la,h_n^\la):n\in\N\}$ and $\{(l_n^\ra,h_n^\ra):n\in\N\}$ of nonnegative slope, 
respectively. 
Let the constants  
$c_n^\la$ 
and
$c_n^\ra$ be 
as in Proposition~\ref{prop:Sandwich-CM} for 
$f^\la$ and $f^\ra$, respectively. For any 
$n,m\in\N$, define the functions $f_{n,m}^\ua,f_{n,m}^\da:[a,b]\to\R$ by
\begin{equation}\label{eq:f_nm}\begin{split}
f_{n,m}^\ua(t) &=f(a)+[(f^\la)^\da_n((s-t)^+)-(f^\la)^\da_n(s-a)]
+ (f^\ra)^\ua_m((t-s)^+),\\
f_{n,m}^\da(t) &=f(a)+[(f^\la)^\da_n((s-t)^+)-(f^\la)^\da_n(s-a)-c_n^\la]
+ (f^\ra)^\da_m((t-s)^+).
\end{split}\end{equation}
For any $c\in\R$, let $T_c$ be the linear tilting defined in 
Subsection~\ref{subsec:e-SS_Levy} above. Set 
$s_c=\tau_{[a,b]}(T_cf)$ and $s_{n,m}=\tau_{[a,b]}(T_cf^\da_{n,m})$. 
Then the following statements hold for any $n,m\in\N$:\\
\nf{(a)} $T_cf_{n,m}^\ua\geq T_cf\geq T_cf_{n,m}^\da$;\\
\nf{(b)} $\|T_cf_{n,m}^\ua-T_cf_{n,m}^\da\|_\infty 
=f_{n,m}^\ua(b)-f_{n,m}^\da(b)=c_n^\la+c_m^\ra$;\\
\nf{(c)} $s_{n,m}\leq  s_c  \leq s_{n,m}+\ti{l}_n^\la$  (resp.
$s_{n,m}-\ti{l}_m^\ra\leq s_c \leq s_{n,m}$) if $c\geq 0$ (resp. $c<0$), where
we denote $\ti{l}_n^\la=\sum_{k=n+1}^\infty l_k^\la$ 
(resp. $\ti{l}_m^\ra=\sum_{k=m+1}^\infty l_k^\ra$);
\\
\nf{(d)} $T_c f_{n,m}^\ua\geq T_c f_{n,m+1}^\ua$ and $T_c f_{n,m}^\ua\geq T_c f_{n+1,m}^\ua$;\\
\nf{(e)} $T_c f_{n,m+1}^\da\geq T_c f_{n,m}^\da$ and $T_c f_{n+1,m}^\da\geq T_c f_{n,m}^\da$.
\end{prop}

\begin{rem}
(i) The upper and lower bounds $f^\ua_{n,m}$ and $f^\da_{n,m}$, restricted to 
$[s,b]$, have the same ``derivative'' as the corresponding bounds in 
Proposition~\ref{prop:Sandwich-CM}. The behaviour of $f^\ua_{n,m}$ and $f^\da_{n,m}$  
on $[a,s]$ differs from that of the bounds in Proposition~\ref{prop:Sandwich-CM}.
%In particular, the bounds are  less tight than
%in  Proposition~\ref{prop:Sandwich-CM}.
Indeed, the lower bound $f^\da_{n,m}$ does not 
start with value $f(a)$ 
because the slopes of the faces of $f$ may become arbitrarily negative as $t$ approcheas $a$. 
Thus,
$f^\da_{n,m}$
is defined as a vertical translation of $f^\ua_{n,m}$ on $[a,s]$. \\
(ii) Note that all bounds in Proposition~\ref{prop:General_UpLo} hold uniformly in 
$c\in\R$, with the exception of part~(c) which depends on the sign of $c$.   
Proposition~\ref{prop:General_UpLo} extends easily to the case of a function $f$ 
without infinitely many faces of both signs. Moreover, if either $n=0$ or $m=0$, then 
as in Remark~\ref{rem:propSandwich}(vi) above, Proposition~\ref{prop:General_UpLo} still holds.
\end{rem}

\begin{proof}
Since $T_cg_1-T_cg_2=g_1-g_2$ for any functions $g_1,g_2:[a,b]\to\R$, it suffices 
to prove the claims (a), (b), (d) and (e)  for $c=0$. \\
(a) Let $\ti{h}^\la_n=\sum_{k=n+1}^\infty h^\la_k$, 
then Remark~\ref{rem:propSandwich}(ii) yields 
$\|f^\la-(f^\la)^\da_n\|_\infty\leq\ti{h}^\la_n$, so the inequality 
$c_n^\la\geq\ti{h}^\la_n$ implies 
$(f^\la)^\da_n +\ti{h}^\la_n\geq f^\la\geq (f^\la)^\da_n +\ti{h}^\la_n-c^\la_n$. 
Note that for $t\in[a,s]$ we have
$f(t) = f(a)+f^\la(s-t)-f^\la(s-a)$.
Moreover, since $\ti{h}^\la_n-f^\la(s-a) =- (f^\la)^\da_n(s-a)$,
%By the construction of $f^\ua_{n,m}$ and $f^\da_{n,m}$ in
by~\eqref{eq:f_nm} we obtain 
$(f^\la)_{n,m}^\ua(t)\geq f(t)\geq (f^\la)_{n,m}^\da(t)$ for $t\in[a,s]$. Similarly, 
since $(f^\ra)_m^\ua\geq f^\ra\geq (f^\ra)^\da_m$ by Proposition~\ref{prop:Sandwich-CM}, 
we deduce that the inequalities in (a) also hold on $[s,b]$.\\
(b) The equalities follow from the definition in~\eqref{eq:f_nm} and
Proposition~\ref{prop:Sandwich-CM}(d).\\
(c) Note that the minimum of $T_c f$ and $T_c f^\da_{n,m}$ 
	is attained after all the faces of negative slope. In 
terms of the functions $f$ and $f^\da_{n,m}$, 
the minimum takes place after all the faces with slopes less than 
$-c$. Put differently,
\begin{equation*}\begin{split}
s_c &
=\sum_{k=1}^\infty l_k^\la\cdot \1_{(-\infty,h_k^\la/l_k^\la)}(c)
+\sum_{k=1}^\infty l_k^\ra\cdot \1_{(-\infty,-h_k^\ra/l_k^\ra)}(c)\\
s_{n,m}&
=\sum_{k=1}^n l_k^\la\cdot \1_{(-\infty,h_k^\la/l_k^\la)}(c)+\ti{l}_n^\la\cdot \1_{(-\infty,0)}(c)
+\sum_{k=1}^m l_k^\ra\cdot \1_{(-\infty,-h_k^\ra/l_k^\ra)}(c)+\ti{l}_m^\ra\cdot \1_{(-\infty,0)}(c).
\end{split}\end{equation*}
If $c\geq0$, then all the terms coming from $f^\ra$ are $0$ and so is 
$\ti{l}_n^\la\cdot \1_{(-\infty,0)}(c)$, implying the first claim in~(c). A similar analysis for 
$c<0$ gives the corresponding claim.\\
(d) The result follows from the definition 
in~\eqref{eq:f_nm}
	and 
Proposition~\ref{prop:Sandwich-CM}(a)\&(b).\\
(e) The result follows from the definition in~\eqref{eq:f_nm}
	and Proposition~\ref{prop:Sandwich-CM}(a)\&(b)\&(d).
\end{proof}

\begin{cor}
\label{cor:Conv_app}
Let $f:[a,b]\to\R$ be a piecewise linear convex function with faces 
$\{(l_k,h_k):k\in\N\}$. Pick $n\in\N$ and let $g:[a,b]\to\R$ be the piecewise linear
convex function with faces 
$\{(l_k,h_k):k\in\mZ_1^{n+1}\}\cup\{(\ti{l}_n,0)\}$ (recall 
$\ti l_n=\sum_{m=n+1}^\infty l_m$), satisfying $g(a)=f(a)$.
Then the following inequality holds:
\begin{equation}
\label{eq:conv_app}
\|f-g\|_\infty\leq\max\bigg\{\sum_{k=n+1}^\infty h_k^-,\sum_{k=n+1}^\infty h_k^+\bigg\}.
\end{equation}
\end{cor}

\begin{proof}
Let 
$m_1=\sum_{k=1}^n \1_{(-\infty,-0)}(h_k)$ and $m_2=n-m_1$.
Define 
$c_{m_1}^\la=\sum_{k=n+1}^\infty h_k^-$
and 
$c_{m_2}^\ra=\sum_{k=n+1}^\infty h_k^+$.
Then, using the notation from Proposition~\ref{prop:General_UpLo}, 
the following holds 
$$g(t)=f(a)+[(f^\la)^\da_{m_1}((s-t)^+)-(f^\la)^\da_{m_1}(s-a)] +(f^\ra)^\da_{m_2}((t-s)^+) \quad\text{for any }t\in[a,b].$$
%Let $g$ be the piecewise linear convex function on $[a,b]$ with initial value 
%$f(a)$ and faces 
%$\{(l_k^\la,-h_k^\la):k\in\mZ_1^{n+1}\}\cup\{(l_k^\ra,h_k^\ra):k\in\mZ_1^{m+1}\}
%\cup\{(\sum_{k=n+1}^\infty l_k^\la+\sum_{k=m+1}^\infty l_k^\ra,0)\}$, i.e., 
%given by $g:t\mapsto f(a)+[(f^\la)^\da_n((s-t)^+)-(f^\la)^\da_n(s-a)] +(f^\ra)^\da_m((t-s)^+)$ for $t\in[a,b]$. 
Moreover, by Propositions~\ref{prop:Sandwich-CM} and~\ref{prop:General_UpLo}, we have 
$g+c_{m_2}^\ra\geq f_{m_1,m_2}^\ua\geq f\geq f_{m_1,m_2}^\da=g-c_{m_1}^\la$ and~\eqref{eq:conv_app} follows. 
\end{proof}
%In particular, 
%since $g$ does \emph{not} depend on either $c_{m_1}^\la$ or $c_m^\ra$ (as in 
%Remark~\ref{rem:propSandwich}(ii) above), 

\begin{rem} 
The proof of Corollary~\ref{cor:Conv_app} shows that using $g$ to construct lower and upper bounds on
$f$ yields poorer estimates than the ones in Proposition~\ref{prop:General_UpLo}.
Indeed, the upper (resp. lower) bound  in Proposition~\ref{prop:General_UpLo} is smaller than (resp. equal to) 
$g+c_{m_2}^\ra$ (resp. $g-c_{m_1}^\la$).
\end{rem}

\subsection{Convex minorant of stable meanders}

\begin{proof}[Proof of Theorem~\ref{thm:CMSM}]\label{proofThmCMSM}
(a) This is a consequence of Theorem~\ref{thm:meander-minorant}. 
Indeed, by the scaling property of the law of $\xi^\me_{1-n}$, each $S_n$ has the 
desired law. Moreover, $(\ell_m^\me)_{m\in\N}$ is a stick-breaking process based on 
$\Beta(1,\rho)$. By the definition of the stick-breaking process, the sequence 
$(U_n)_{n\in\mZ^0}$ has the required law.  
The independence structure is again implied by Theorem~\ref{thm:meander-minorant}, 
since the conditional law of $(S_n)_{n\in\mZ^0}$, given $(\ell_m^\me)_{m\in\N}$, 
no longer depends on the lengths of the sticks.\\
(b) The recursion~\eqref{eq:SMPerEq} follows form the definition in~\eqref{eq:MC}. 
Since $M_n$ is independent of $(U_n,S_n)$,the Markov property is a direct 
consequencel of~(a) and~\citep[Prop.~7.6]{MR1876169}. The stationarity of 
$((S_n,U_n))_{n\in\mZ^0}$ in (a) and the identity
\begin{equation*}%\label{eq:infty_rec}
M_n=\sum_{m\in\mZ^n}\bigg(\prod_{k\in\mZ_{m+1}^n}(1-U_k)^{1/\alpha}\bigg)
U_m^{1/\alpha}S_{m},\quad n\in\mZ^1,
\end{equation*}
imply that $(M_n)_{n\in\mZ^1}$ is also stationary.\\
(c) The perpetuity follows from (b). It has a unique solution 
by~\citep[Thm~2.1.3]{MR3497380}.
\end{proof}

\begin{rem}
A result  analogous to Theorem~\ref{thm:CMSM} for stable processes 
and their convex minorants holds (see~\citep[Prop.~1]{ExactSim}).  In fact, the proof in~\citep{ExactSim}
implies a slightly stronger result, namely, a perpetuity for the triplet $(Z_T,\un{Z}_T,\tau_{[0,T]}(Z))$.
\end{rem}

The following result, which may be of independent interest, is a 
consequence of Theorem~\ref{thm:CMSM}. Parts of it were used to numerically test 
our algorithms in Subsection~\ref{subsec:numerics} above. 

\begin{cor}\label{cor:moments}
Consider some $S\sim\mS^+(\alpha,\rho)$.\\
{\normalfont(a)} If $\alpha>1$, then $\E Z^\me_1=
\frac{\Gamma(1/\alpha)\Gamma(1+\rho)}{\Gamma(\rho+1/\alpha)}\E S=
\frac{\Gamma(1/\alpha)\Gamma(1-1/\alpha)}{\Gamma(\rho+1/\alpha)\Gamma(1-\rho)}$.\\
{\normalfont(b)} For any $(\alpha,\rho)$ we have
$\E\big[(Z_1^\me)^{-\alpha\rho}\big]=
\Gamma(1+\rho)/\Gamma(1+\alpha\rho)$.\\
{\normalfont(c)} For $\gamma>0$ 
let $ k_{\alpha,\gamma} 
=\big((1+\gamma/\alpha)^{\min\{\gamma^{-1},1\}}-1\big)^{-\max\{\gamma,1\}}$.
Then we have
\[
\rho\E[S^\gamma]\leq\E[(Z_1^\me)^\gamma]
\frac{\Gamma(\rho+\gamma/\alpha)}{\Gamma(\rho)\Gamma(1+\gamma/\alpha)}
\leq\min\{\rho k_{\alpha,\gamma},1\}\E[S^\gamma].
\]
{\normalfont(d)} For $\gamma\in(0,\alpha\rho)$, we have
$\E\big[(Z_1^\me)^{-\gamma}\big]\leq\Gamma(1+\rho)\Gamma(1-\gamma/\alpha)
\E\big[S^{-\gamma}\big]/\Gamma(1+\rho-\gamma/\alpha)$.
\end{cor}

\begin{proof}%[Proof of Corollary~\ref{cor:moments}]\label{proofMoments}
(a) Recall that $\E[V^r]=\frac{\Gamma(\theta_1+r)\Gamma(\theta_1+\theta_2)}
{\Gamma(\theta_1+\theta_2+r)\Gamma(\theta_1)}$ for any $V\sim 
\Beta(\theta_1,\theta_2)$. Taking expectations in~\eqref{eq:SMPerEq} and solving 
for $\E Z_1^\me$ gives the formula.\\
(b) Let $V\sim \Beta(\rho,1-\rho)$ be independent of $Z^\me$ and denote the 
supremum of $Z$ by $\ov{Z}_1=\sup_{t\in[0,1]}Z_t$.
Then~\citep[Cor.~VIII.4.17]{MR1406564} implies that
$V^{1/\alpha}Z_1^\me\overset{d}{=}\ov{Z}_1$. By Breiman's 
lemma~\citep[Lem.~B.5.1]{MR3497380},
\[
1=\lim_{x\to\infty}\frac{\P\left(V^{-1/\alpha}(Z_1^\me)^{-1}>x\right)}
{\E\big[(Z_1^\me)^{-\alpha\rho}\big]\P\big(V^{-1/\alpha}>x\big)}
=\lim_{x\to\infty}\frac{\P(\ov{Z}_1^{-1}>x)\Gamma(1+\rho)\Gamma(1-\rho)}
{\E\big[(Z_1^\me)^{-\alpha\rho}\big]x^{-\alpha\rho}}.
\]
Since \citep[Thm~3a]{MR0415780} gives 
$1=\lim_{x\to\infty}\Gamma(1-\rho)\Gamma(1+\alpha\rho)
\P\big(\ov{Z}_{1}^{-1}>x\big)/x^{-\alpha\rho}$, we get (b).\\
(c) Note $\E\big[\big(Z_1^\me\big)^{\gamma}\big]
=\Gamma(\rho)\Gamma(1+\frac{\gamma}{\alpha})
\E\big[\ov{Z}_1^\gamma\big]/\Gamma(\rho+\frac{\gamma}{\alpha})$
for $\gamma>-\alpha\rho$ since $V$ and $Z_1^\me$ are independent and 
$V^{1/\alpha}Z_1^\me\overset{d}{=}\ov{Z}_1$. Hence, we need only prove that
\begin{equation}\label{eq:ineq_supermum}
\rho\E[S^\gamma]\leq\E[\ov{Z}_1^\gamma]
\leq\min\{\rho k_{\alpha,\rho},1\}\E[S^\gamma].
\end{equation}
Recall that for a nonnegative random variable $\vartheta$ we have
$\E[\vartheta^\gamma]=\int_0^\infty\gamma x^{\gamma-1}\P(\vartheta>x)dx$.
Since $\P(\ov{Z}_1>x)\geq\P(Z_1^+>x)=\rho\P(S>x)$, we get 
$\E[\ov{Z}_1^\gamma]\geq\rho\E[S^\gamma]$. Next, fix any $x>0$ 
and let $\sigma_x=\inf\{t>0:Z_t>x\}$. By the strong Markov property, the 
process $Z'$ given by $Z'_t=Z_{t+\sigma_x}-Z_{\sigma_x}$, $t>0$, has the same law 
as $Z$ and is independent of $\sigma_x$. Thus, we have
\[\begin{split}
\P(\ov{Z}_1>x) 
&=\P(Z_1>x) + \P(\ov{Z}_1>x,Z_1\leq x) 
\leq\P(Z_1>x) + \P(\sigma_x<1,Z'_{1-\sigma_x}\leq 0) \\
&=\P(Z_1>x) + (1-\rho)\P(\sigma_x<1) 
=\P(Z_1>x) + (1-\rho)\P(\ov{Z}_1>x),
\end{split}\] 
implying $\P(\ov{Z}_1>x)\leq\P(S>x)$. Hence the same argument 
gives $\E[S^\gamma]\geq\E[\ov{Z}_1^\gamma]$. Note
\begin{equation*}
k_{\alpha,\gamma} 
%=\big((1+\gamma/\alpha)^{\min\{\gamma^{-1},1\}}-1\big)^{-\max\{\gamma,1\}}
=\begin{cases}
\big((1+\gamma/\alpha)^{1/\gamma}-1\big)^{-\gamma}&\text{ if }\gamma>1\\
\alpha/\gamma & \text{ if }\gamma\leq 1.
\end{cases}
\end{equation*}
The last inequality $\E[\ov{Z}_1^\gamma]
\leq\rho k_{\alpha,\gamma}\E[S^\gamma]$ in~\eqref{eq:ineq_supermum} 
follows from the perpetuity for the law of 
$\ov{Z}_1$ in~\citep[Eq.~(2.1)]{ExactSim} and the inequality 
in the proof of~\citep[Lem.~2.3.1]{MR3497380}.\\
(d) Note that~\eqref{eq:SMPerEq} and the Mellin transform of $S$
(see~\citep[Sec.~5.6]{MR1745764}) imply
\[\begin{split}
\E\big[\big(Z_1^\me\big)^{-\gamma}\big] 
& = \E\big[\big(U^{1/\alpha}Z_1^\me
+(1-U)^{1/\alpha}S\big)^{-\gamma}\big]
\leq\E\big[\big(1-U\big)^{-\gamma/\alpha}\big]
\E\big[S^{-\gamma}\big]\\
& = \frac{\Gamma(1-\frac{\gamma}{\alpha})\Gamma(1+\rho)}
{\Gamma(1+\rho-\frac{\gamma}{\alpha})}\E\big[S^{-\gamma}\big]
=\frac{\Gamma(1+\rho)\Gamma(1-\gamma)}
{\Gamma(1+\rho-\frac{\gamma}{\alpha})}
\frac{\Gamma(1-\frac{\gamma}{\alpha})\Gamma(1+\frac{\gamma}{\alpha})}
{\Gamma(1-\gamma\rho)\Gamma(1+\gamma\rho)}<\infty.
\end{split}\]
\end{proof}

\begin{rem}
(i) Bernoulli's inequality implies $k_{\alpha,\gamma}\leq\alpha^\gamma$ 
for $\gamma>1$.\\
(ii) From $V^{1/\alpha}Z_1^\me\overset{d}{=}\ov{Z}_1$ we get 
$\E\ov{Z}_1=\alpha\rho\E S=\frac{\alpha\Gamma(1-1/\alpha)}
{\Gamma(\rho)\Gamma(1-\rho)}$ when $\alpha>1$. Similarly, for 
$\gamma\in(0,\alpha\rho)$, we have $\E\big[\ov{Z}_1^{-\gamma}\big]\leq
\rho(1+(1-\rho)/(\rho-\gamma/\alpha))\E\big[S^{-\gamma}\big]/(1-\gamma/\alpha)$
by the proof in~(d) applied to the perpetuity in~\citep[Thm~1]{ExactSim}.\\ 
(iii) Note that equation~\eqref{eq:SMPerEq} and the Grincevi\u{c}ius-Grey 
theorem~\citep[Thm~2.4.3]{MR3497380} give 
$\lim_{x\to\infty}\frac{1+\rho}{\rho}\P(U^{1/\alpha}S>x)/\P(Z_1^\me>x)=1$.
Next, Breiman's lemma~\citep[Lem.~B.5.1]{MR3497380} gives 
$\lim_{x\to\infty}(1+\rho)\P(U^{1/\alpha}S>x)/\P(S>x)=1$. 
Hence,~\citep[Sec.~4.3]{MR1745764} gives the assymptotic tail behaviour 
$\lim_{x\to\infty}\P(Z_1^\me>x)/x^{-\alpha}
=\Gamma(\alpha)\sin(\pi\alpha\rho)/(\pi\rho)$ (cf.~\citep{MR2599201}).
\end{rem}

\subsection{Computational complexity\label{subsec:complexity}}
The aim of the subsection is to analyse the computational complexity of the
$\varepsilon$SS algorithms from Section~\ref{sec:algorithms} and the exact simulation 
algorithm of the indicator of certain events (see Subsection~\ref{subsec:indicators} above). 
Each algorithm in Section~\ref{sec:algorithms} constructs an approximation of a random
element $\Lambda$ in a metric space $(\mathbb{X},d)$, given by a sequence 
$(\Lambda_n)_{n\in\N}$ in $(\mathbb{X},d)$ and upper bounds $(\Delta_n)_{n\in\N}$ 
satisfying $\Delta_n\geq d(\Lambda,\Lambda_n)$ for all $n\in\N$. The $\varepsilon$SS 
algorithm terminates as soon as $\Delta_n<\varepsilon$. Moreover, the computational 
complexity of constructing the finite sequences $\Lambda_1,\ldots,\Lambda_n$ and 
$\Delta_1,\ldots,\Delta_n$ is linear in $n$ for the algorithms in Section~\ref{sec:algorithms}. 
For $\varepsilon>0$, the runtime of the $\varepsilon$SS algorithm is thus 
proportional to $N^\Lambda(\varepsilon)=\inf\{n\in\N:\Delta_n<\varepsilon\}$ since 
the element $\Lambda_{N^\Lambda(\varepsilon)}$ is the output of the $\varepsilon$SS. 
Proposition~\ref{prop:MainRuntime} below shows that for all the algorithms in 
Section~\ref{sec:algorithms}, we have $\Delta_n\to0$ a.s. as $n\to\infty$, implying 
$N^\Lambda(\varepsilon)<\infty$ a.s. for $\varepsilon>0$.

The exact simulation algorithm of an indicator $\1_A(\Lambda)$, for some subset $A\subset\mathbb{X}$ 
satisfying $\P(\Lambda\in\partial A)=0$, has a complexity proportional to 
\begin{equation}\label{eq:indicators_stop_time}
B^\Lambda(A)=\inf\{n\in\N:\Delta_n<d(\Lambda_n,\partial A)\},
\end{equation}
since $\1_A(\Lambda)=\1_A(\Lambda_{B^\Lambda(A)})$ a.s.
Indeed, if $d(\Lambda,\Lambda_n)\leq\Delta_n<d(\Lambda_n,\partial A)$ then 
$\1_A(\Lambda)=\1_A(\Lambda_n)$. %where the latter is computable, as is $B^\Lambda(A)$.
Moreover, $B^\Lambda(A)<\infty$ a.s. since 
$d(\Lambda_n,\partial A)\to d(\Lambda,\partial A)>0$ and $\Delta_n\to0$ a.s. The next 
lemma provides a simple connection between the tail probabilities of the complexities 
$N^\Lambda(\varepsilon)$ and $B^\Lambda(A)$. It will play a key role in the 
%analysis 
%of Algorithm~\ref{alg:fd_meander} and the 
proof of Theorem~\ref{thm:all_complexities}.

\begin{lem}\label{lem:indicators_time} 
Let $A\subset\mathbb{X}$ satisfy $\P(\Lambda\in\partial A)=0$. Assume that 
some positive constants $r_1$, $r_2$, $K_1$ and $K_2$ and a nonincreasing 
function $q:\N\to[0,1]$ with $\lim_{n\to\infty}q(n)=0$ satisfy
\begin{equation}\label{eq:indicators_decay_asm}
\P(d(\Lambda,\partial A)<\varepsilon) \leq K_1\varepsilon^{r_1},\quad
\P(N^\Lambda(\varepsilon)>n) \leq K_2\varepsilon^{-r_2}q(n),
\end{equation}
for all $\varepsilon\in(0,1]$ and $n\in\N$. Then, for all $n\in\N$, we have
\begin{equation}\label{eq:indicators_decay}
\P(B^\Lambda(A)>n)\leq (K_1 + 2^{r_2}K_2)q(n)^{r_1/(r_1+r_2)}.
\end{equation}
\end{lem}

\begin{proof}
Note that $\{d(\Lambda,\partial A)\geq 2\varepsilon\}\subset
\{N^\Lambda(\varepsilon)\geq B^\Lambda(A)\}$ for any $\varepsilon>0$ since 
$d(\Lambda,\partial A)\geq 2\varepsilon$ and $d(\Lambda,\Lambda_n)\leq 
\Delta_n<\varepsilon$ imply $\Delta_n < \varepsilon < d(\Lambda_n,\partial A)$. 
Thus, if we define, for each $n\in\N$, $\varepsilon_n = 
q(n)^{1/(r_1+r_2)}/2\in(0,1/2)$, we get
\begin{equation*}\begin{split}
\P(B^\Lambda(A)>n) 
& = \P(B^\Lambda(A)>n,d(\Lambda,\partial A)<2\varepsilon_n)
+ \P(B^\Lambda(A)>n,d(\Lambda,\partial A)\geq 2\varepsilon_n)\\
& \leq \P(d(\Lambda,\partial A)<2\varepsilon_n)
+ \P(N^\Lambda(\varepsilon_n)>n)\\
& \leq 2^{r_1}K_1\varepsilon_n^{r_1} + K_2\varepsilon_n^{-r_2}q(n) 
= (K_1 + 2^{r_2}K_2)q(n)^{r_1/(r_1+r_2)}.
\end{split}\end{equation*}
\end{proof}

\subsubsection{Complexities of Algorithm~\ref{alg:eps-CM-SM} and the exact simulation algorithm of the indicator $\1_{\{Z^\me_1>x\}}$}

%The work and notation of this section is closely linked to~\citep{ExactSim}. 
Recall the definition $c_{-n}=L_n^{1/\alpha}D_n$, $n\in\mZ^0$,
in~\eqref{eq:c_n}. The dominating process  $(D_n)_{n\in\mZ^1}$, 
defined in Appendix~\ref{sec:notationSupSim} (see~\eqref{eq:DominatingProcess} below),
is inspired by the one in~\citep{ExactSim}. 
In fact, the sampling of the process 
$(D_n)_{n\in\mZ^1}$
is achieved by using~\citep[Alg.~2]{ExactSim}
as explained in the appendix.
The computational complexities of Algorithms~\ref{alg:fd_meander} and~\ref{alg:eps-CM-SM}   
are completely determined by 
%By Proposition~\ref{prop:Sandwich-CM}, the other results from  
%Subsection~\ref{subsec:approx_linear} and Theorem~\ref{thm:CMSM}, we have 
%reduced all $\varepsilon$SS to that of $M_0\sim\mS^\me(\alpha,\rho)$. To that 
%end, we will sample $((\Phi_n^0,D_n))_{n\in\mZ^1}$ backwards in time and use the 
%estimates $(\Phi_n^0(0),\Phi_n^0(D_n))$ that satisfy $\Phi_n^0(0)\leq 
%M_0\leq\Phi_n^0(D_n)$ and $\Phi_n^0(D_n)-\Phi_n^0(0)=L_n^{1/\alpha}D_n$. 
%Hence, for an $\varepsilon$-strong sample of $M_0$ we need only simulate until 
\begin{equation}\label{eq:N_epsilon}
N(\varepsilon)=\sup\{n\in\mZ^0:L_n^{1/\alpha}D_n<\varepsilon\},\quad\varepsilon>0.
\end{equation}
It is thus 
our aim to develop bounds on the tail probabilities of $N(\varepsilon)$, 
which requires the analysis of the sampling algorithm in~\citep[Alg.~2]{ExactSim}.
We start by proving that the error bounds $(c_m)_{m\in\N}$ are strictly decreasing.
%In the proofs of 
%(see Appendix~\ref{sec:notationSupSim} for the definition of 
%the r
%Since the definition of $(D_n)_{n\in\mZ^1}$ 
%involves certain free parameters, we now introduce them: 
%fix $d\in(0,\frac{1}{\alpha\rho})$ and $\delta\in(0,d)$; define the constant 
%and fix some $\gamma>0$ with $\E[S_1^\gamma]<\infty$ (see 
%Appendix~\ref{sec:notationSupSim} below and~\citep[Sec.~5.6]{MR1745764}).

\begin{lem}\label{lem:eps-strong}
The sequence  $(c_m)_{m\in\N}$, given by $c_m=L_{-m}^{1/\alpha}D_{-m}>0$, %_{n\in\N}$ 
is strictly decreasing: $c_m>c_{m+1}$ a.s. for all $m\in\N$. 
\end{lem}

\begin{proof}
Fix $n\in\mZ^1$ and note that by~\eqref{eq:DominatingProcess}
\begin{equation*}\begin{split}
L_n^{1/\alpha}D_n & =  e^{\sup_{k\in\mZ^{n+1}}W_{k}}E_{n},\quad\text{where}\\
E_n & =  \frac{e^{(d-\delta)\chi_{n}+n\delta}}{1-e^{\delta-d}}+
\sum_{k\in\mZ_{\chi_n}^n}e^{(k+1)d}U_{k}^{1/\alpha}S_k.
\end{split}\end{equation*}
The random walk $(W_k)_{k\in\mZ^1}$, the random variables $\chi_n$, $n\in\mZ^1$, and the constants $d,\delta$ are given in Appendix~\ref{sec:notationSupSim} below.
The pairs $(U_k,S_k)$ are given in Theorem~\ref{thm:CMSM} (see also Algorithm~\ref{alg:eps-CM-SM}).
Since $\sup_{k\in\mZ^{n}}W_k\leq \sup_{k\in\mZ^{n+1}}W_k$ for all $n\in\mZ^1$, it suffices
to show that $E_{n-1}<E_n$. From the definition in~\eqref{eq:chi} of $\chi_n$ it follows that
\[\sum_{k\in\mZ_{\chi_{n-1}}^{\chi_n}}e^{(k+1)d}U_k^{1/\alpha}S_k
\leq\sum_{k\in\mZ_{\chi_{n-1}}^{\chi_n}}e^{(k+1)d}e^{\delta(n-k-1)}
=\frac{e^{(d-\delta)\chi_n+n\delta}\big(1-e^{(d-\delta)(\chi_{n-1}-\chi_n)}\big)}
{1-e^{\delta-d}}.\]
The inequality in display then yields
\begin{equation*}\begin{split}
E_{n-1}-E_{n}
& = \frac{e^{(d-\delta)\chi_{n-1}+(n-1)\delta}
	-e^{(d-\delta)\chi_n+n\delta}}{1-e^{\delta-d}}
-e^{nd}U_{n-1}^{1/\alpha}S_{n-1}
+\sum_{k\in\mZ_{\chi_{n-1}}^{\chi_n}}e^{(k+1)d}U_k^{1/\alpha}S_k\\
&\leq e^{(d-\delta)\chi_n+n\delta}
\big(e^{(d-\delta)(\chi_{n-1}-\chi_n)-\delta}-1
+1-e^{(d-\delta)(\chi_{n-1}-\chi_n)}\big)/(1-e^{\delta-d})\\
& =  e^{(d-\delta)\chi_n+n\delta}e^{(d-\delta)(\chi_{n-1}-\chi_n)}
\big(e^{-\delta}-1\big)/(1-e^{\delta-d})<0,
\end{split}\end{equation*}
implying $E_{n-1}<E_n$ and concluding the proof.
\end{proof}

%\begin{rem}
%A direct and important consequence is that for a given $\varepsilon>0$ the 
%inequality $|M_0-\Phi_n^0(x)|<\varepsilon$ holds for all 
%$n\in\mZ^{N(\varepsilon)+1}$ and $x\in[0,D_n]$.
%\end{rem}

We now analyse the 
tail of $N(\varepsilon)$ defined in~\eqref{eq:N_epsilon}.

\begin{prop}\label{prop:MainRuntime}
Pick $\varepsilon\in(0,1)$ and let the constants 
$d,\delta,\gamma$ and $\eta$ be as in Appendix~\ref{sec:notationSupSim}.
Define the constants $r=(1-e^{\delta-d})/2>0$, $m^\ast=\lfloor\frac{1}{\delta\gamma}\log\E[S^\gamma]\rfloor+1$ 
(here $S\sim\mS^+(\alpha,\rho)$) and  
\begin{equation}\label{eq:K}
K=e^{d\eta}+e^{\delta\gamma}(e^{d\eta}-1)\max\left\{\frac{\E[S^\gamma]}
{(1-e^{-\delta\gamma})(1-e^{-\delta\gamma m^\ast}\E[S^\gamma])},
e^{\delta\gamma m^\ast}\right\}>0.
\end{equation}
Then $|N(\varepsilon)|$ has exponential moments: for all $n\in\mathbb{N}$, we have 
\begin{equation}\label{eq:tail_Ne}
\P(|N(\varepsilon)|>n)
\leq (K/r^\eta)\varepsilon^{-\eta}e^{-n\min\{\delta\gamma,d\eta\}}
\bigg(\frac{\1_{\R\setminus\{\delta\gamma\}}(d\eta)}{|e^{\delta\gamma-d\eta}-1|}
+n\cdot \1_{\{\delta\gamma\}}(d\eta)\bigg)
%\begin{cases}
%1/|e^{\delta\gamma-d\eta}-1|,&\delta\gamma\neq d\eta,\\
%n, & \delta\gamma=d\eta.
%\end{cases}
\end{equation}
%Hence $\P(|N(\varepsilon)|>n)=\O(e^{-\min\{\delta\gamma,d\eta\}n}
%(1+n\cdot \1_{d\eta=\delta\gamma}))$.
\end{prop}

\begin{proof}[Proof of Proposition~\ref{prop:MainRuntime}]\label{proofPropMainRuntime}
Fix $n\in\mZ^0$, put $\varepsilon'=-\log((1-e^{\delta-d})\varepsilon/2)
=-\log(r\epsilon)>0$ and let $R_0=\sup_{m\in\mZ^1}W_m$. Since 
$L_n^{1/\alpha}=\exp(W_n+nd)$, then by~\eqref{eq:DominatingProcess}, we have
\begin{equation*}\begin{split}
L_n^{1/\alpha}D_n
& %< L_n^{1/\alpha}D_n' =
< e^{nd+\sup_{m\in\mZ^{n+1}}W_m}\bigg(\frac{1}{1-e^{\delta-d}}
+\sum_{k\in\mZ^n}e^{-(n-k-1)d}S_k\bigg)\\
& \leq e^{R_0}\bigg(\frac{e^{nd}}{1-e^{\delta-d}}
+\sum_{k\in\mZ^n}e^{(k+1)d}S_k\bigg).
\end{split}\end{equation*}
Assume that $n\leq\chi_m$ for some $m\in\mZ^1$, then $n\leq\chi_m<m$ and thus
\[
\sum_{k\in\mZ^n}e^{(k+1)d}S_k 
\leq \sum_{k\in\mZ^n}e^{(k+1)d}e^{\delta(m-k-1)}
= \frac{e^{\delta m}e^{n(d-\delta)}}{1-e^{\delta-d}}
< \frac{e^{md}}{1-e^{\delta-d}}.
\]
Hence $L_n^{1/\alpha}D_n<2\exp(R_0+md)/(1-e^{\delta-d})$. Thus, the choice 
$m=\lfloor -(\varepsilon'+R_0)/d\rfloor$ gives $L_n^{1/\alpha}D_n<\varepsilon$ 
where $\lfloor x\rfloor=\sup\{n\in\mathbb{Z}:n\leq x\}$ for $x\in\R$. This yields 
the bound $|N(\varepsilon)|\leq |\chi_{\lfloor -(\varepsilon'+R_0)/d\rfloor}|$. 
Since $(\chi_n)_{n\in\mZ^0}$ is a function of $(S_n)_{n\in\mZ^0}$ and 
$R_0$ is a function of $(U_n)_{n\in\mZ^0}$, the sequence $(\chi_n)_{n\in\mZ^0}$ is 
independent of $R_0$. By~\citep{MR1759244} (see
also~\citep[Rem.~4.3]{ExactSim}) there exists an exponential random variable $E$ 
with mean one, independent of $(\chi_n)_{n\in\mZ^0}$, satisfying 
$R_0\leq\eta^{-1}E$ a.s. Since the sequence $(\chi_n)_{n\in\mZ^0}$ is nonincreasing, 
we have $|\chi_{\lfloor -(\varepsilon'+R_0)/d\rfloor}| \leq |\chi_{-J}|$
where $J=\lceil(\varepsilon'+\eta^{-1}E)/d\rceil$. By definition~\eqref{eq:chi},
for any $n\in\N$ we have
\begin{equation}\label{eq:ne_1}\begin{split}
\P(|N(\varepsilon)|>n)
&\leq \P(|\chi_{-J}|>n) = \P(J\geq n)+\P(J<n,|\chi_{-J}|>n)\\
& = \P(J \geq n)
+\1_{(\lceil\varepsilon'/d\rceil,\infty)}(n)\cdot\sum_{k=\lceil \varepsilon'/d\rceil}^{n-1}
\P(J =k)\P(|\chi_{-k}|>n)\\
& \leq e^{(\varepsilon'-(n-1)d)\eta}
+\1_{(\lceil\varepsilon'/d\rceil,\infty)}(n)\cdot(e^{d\eta}-1)e^{\varepsilon'\eta}
\sum_{k=\lceil \varepsilon'/d\rceil}^{n-1}e^{-kd\eta}\P(|\chi_{-k}|>n).
\end{split}\end{equation}

We proceed to bound the tail probabilities of the variables $\chi_{-k}$. 
For all $n,k\in\N$,  by~\eqref{eq:chi} and~\eqref{eq:tail_chi_n} below, we obtain
\[\P(|\chi_{-k}|>n+k)=\P(|\chi_0|>n)\leq K'e^{-\delta\gamma n},
\quad\text{where}\quad K' = e^{\delta\gamma m^\ast}\max\{K_0,1\}\]
and $K_0$ is defined in~\eqref{eq:tail_chi_n}.
Thus we find that, for $n>\lceil\varepsilon'/d\rceil$ and 
$m=n-\lceil \varepsilon'/d\rceil$, we have 
\begin{equation*}\begin{split}
\sum_{k=\lceil \varepsilon'/d\rceil}^{n-1}e^{-kd\eta}\P(|\chi_{-k}|>n)
&\leq\sum_{k=\lceil \varepsilon'/d\rceil}^{n-1}e^{-kd\eta}K'e^{-\delta\gamma(n-k)}
= K'e^{-n\delta\gamma}
\sum_{k=\lceil \varepsilon'/d\rceil}^{n-1} e^{k(\delta\gamma-d\eta)}\\
&= K'e^{-m\delta\gamma}e^{-\lceil \varepsilon'/d\rceil d\eta}
\bigg(\frac{e^{m(\delta\gamma-d\eta)}-1}{e^{\delta\gamma-d\eta}-1}\cdot 
\1_{\R\setminus\{\delta\gamma\}}(d\eta)+m\cdot \1_{\{\delta\gamma\}}(d\eta)\bigg)\\
&\leq K'e^{-m\min\{\delta\gamma,d\eta\}}
e^{-\lceil \varepsilon'/d\rceil d\eta}
\bigg(\frac{\1_{\R\setminus\{\delta\gamma\}}(d\eta)}{|e^{\delta\gamma-d\eta}-1|}
+n\cdot \1_{\{\delta\gamma\}}(d\eta)\bigg).
\end{split}\end{equation*}

Note that $K$ defined in~\eqref{eq:K} equals 
$K=e^{d\eta}+(e^{d\eta}-1)K'e^{\delta\gamma}$. Let $n'=n-\varepsilon'/d$. 
Using~\eqref{eq:ne_1}, $\varepsilon'/d+1>\lceil\varepsilon'/d\rceil\geq\varepsilon'/d$ 
and the inequality in the previous display, we obtain
\begin{equation*}
\P(|N(\varepsilon)|>n)
\leq e^{-n'd\eta}e^{d\eta}+
\1_{(\lceil\varepsilon'/d\rceil,\infty)}(n)\cdot (K-e^{d\eta})e^{-n'\min\{\delta\gamma,d\eta\}}
\bigg(\frac{\1_{\R\setminus\{\delta\gamma\}}(d\eta)}{|e^{\delta\gamma-d\eta}-1|}
+n\cdot\1_{\{\delta\gamma\}}(d\eta)\bigg).
\end{equation*}
Since $r\varepsilon<1$ and $e^{n'd}=r\varepsilon e^{nd}$,
the result follows by simplifying the previous display.
\end{proof}

Recall from Theorem~\ref{thm:CMSM} that $M_0=Z^\me_1$. In applications, we often 
need to run the chain in Algorithm~\ref{alg:eps-CM-SM} until, for a given $x>0$, we can 
%are able to 
detect which of the events $\{M_0>x\}$ or $\{M_0<x\}$ occurred (note that 
$\P(M_0=x)=0$ for all $x>0$). This task is equivalent to simulating exactly the indicator 
$\1_{\{M_0>x\}}$. We now analyse the tail of the running time of such a simulation algorithm.

%For example, in the 
%context of Algorithm~\ref{alg:fd_meander}, the inner loop terminates when the 
%condition in line~12 is satisfied. Inside the loop, the algorithm is simultaneously 
%running the (fixed number of) Markov chains associated to lines~5 and~8 (see 
%Remark~\ref{rem:alg_1}(ii)).

\begin{prop}\label{prop:hit_time_tail}
For any $x>0$, let $B(x)$ be the number of steps required to sample $\1_{\{M_0>x\}}$. 
Let $d,\eta,\delta$ and $\gamma$ be as in Proposition~\ref{prop:MainRuntime}. Then 
$B(x)$ has exponential moments: 
\begin{equation}\label{eq:hit_time_tail}
\P(B(x)>n)\leq K_0[e^{-sn}(1+n\cdot \1_{\{\delta\gamma\}}(d\eta))]^{1/(1+\eta)},
\quad\text{for all }n\in\N,
\end{equation}
where $s=\min\{d\eta,\delta\gamma\}$ and $K_0>0$ do not depend on $n$.
\end{prop}

\begin{proof}[Proof of Proposition~\ref{prop:hit_time_tail}]\label{proofPropHittingTimeRuntime}
The inequality in~\eqref{eq:hit_time_tail} will follow from Lemma~\ref{lem:indicators_time} 
once we identify the constants $r_1$, $K_1$, $r_2$, $K_2$ and the function 
$q:\N\to(0,\infty)$ that satisfy the inequalities in~\eqref{eq:indicators_decay_asm}. 
By~\citep[Lem.~8]{MR2599201}, the distribution $\mS^\me(\alpha,\rho)$ has a continuous 
density $f_\me$, implying that the distribution function of $\mS^\me(\alpha,\rho)$ is 
Lipschitz at $x$. Thus we may set $r_1=1$ and there exists some $K_1>f_\me(x)$ such that 
the first inequality in~\eqref{eq:indicators_decay_asm} holds. Similarly, \eqref{eq:tail_Ne} 
and~\eqref{eq:K} in Proposition~\ref{prop:MainRuntime} imply that the second inequality 
in~\eqref{eq:indicators_decay_asm} holds if we set $r_2=\eta$, 
$K_2=K/(r^\eta|e^{\delta\gamma-d\eta}-1|)$ and 
$q(n)=e^{-sn}(1+n\cdot\1_{\{\delta\gamma\}}(d\eta))$, where $s=\min\{d\eta,\delta\gamma\}$. 
Thus, Lemma~\ref{lem:indicators_time} implies the inequality in~\eqref{eq:hit_time_tail} for 
$K_0=K_1+2^{r_2} K_2$.
\end{proof}

\begin{rem}\label{rem:cor_hit_time}
We stress that the constant $K_0$ is not explicit since 
the constant $K_1$ in the proof above depends on the behaviour of the density $f_\me$ 
of $Z^\me_1$ in a neighbourhood of $x$. 
To the best of our knowledge 
even the value $f_\me(x)$ 
is currently not available in the literature. 
\end{rem}

\subsubsection{Proof of Theorem~\ref{thm:all_complexities}}
\label{proofThmComplexities}
The computational complexity of Algorithm~\ref{alg:eps-CM-SM} is bounded above by a 
constant multiple of $|N(\varepsilon)| \log |N(\varepsilon)|$, 
cf.~Remark~\ref{rem:eps-CM-SM}(i) following the statement of
Algorithm~\ref{alg:eps-CM-SM}. By Proposition~\ref{prop:MainRuntime}, its computational 
complexity has exponential moments. Since Algorithm~\ref{alg:eps-CM} amounts to 
running Algorithm~\ref{alg:eps-CM-SM} twice, its computational complexity also has 
exponential moments. By Proposition~\ref{prop:hit_time_tail}, the running time of the exact 
simulation algorithm for the indicator $\1_{\{M_0>x\}}$ has exponential moments.  It 
remains to analyse the runtime of Algorithm~\ref{alg:fd_meander}.

Recall that line~8 in Algorithm~\ref{alg:fd_meander} requires sampling a beta random 
variable and two meanders at time 1 with laws $Z^\me_1$ and $(-Z)^\me_1$. The former 
(resp. latter) meander has the positivity parameter $\rho$ (resp. $1-\rho$). Moreover, we 
may use Algorithm~\ref{alg:eps-CM-SM} (see also Remark~\ref{rem:eps-CM}) to obtain 
an $\varepsilon$SS of $Z^\me_1$ and $(-Z)^\me_1$ by running backwards the dominating 
processes (defined in~\eqref{eq:DominatingProcess}, see 
Appendix~\ref{sec:notationSupSim}) of the Markov chains in Theorem~\ref{thm:CMSM}. 
Let $(d,\eta,\delta,\gamma)$ and $(d',\eta',\delta',\gamma')$ be the parameters 
required for the definition of the respective dominating processes, introduced at the 
beginning of Appendix~\ref{sec:notationSupSim}. The $\varepsilon$SS algorithms invoked 
in line~8 of Algorithm~\ref{alg:fd_meander} require $2m+1$ independent dominating 
processes to be simulated ($m+1$ of them with 
parameters $(d,\eta,\delta,\gamma)$ and $m$ of them with parameters 
$(d',\eta',\delta',\gamma')$). Denote by $N_k(\varepsilon)$, $k\in\{1,\ldots,2m+1\}$ and 
$\varepsilon>0$, their respective termination times, defined as in~\eqref{eq:N_epsilon}. 

Note that, in the applications of Algorithm~\ref{alg:eps-CM-SM}, the sampled faces need 
not be sorted (see Remark~\ref{rem:eps-CM-SM}), thus eliminating the logarithmic 
effect described in Remark~\ref{rem:eps-CM-SM}(i). The cumulative complexity of 
executing $i$ times the loop from line~3 to line~12 in Algorithm~\ref{alg:fd_meander}, 
producing an $\varepsilon_i$-strong sample of 
$(Z_{t_1},\ldots,Z_{t_m})$ conditioned on $\un{Z}_{t_1}\geq0$, 
%, using $\varepsilon_1=\varepsilon/(m+1)$ in the notation therein) 
only depends on the precision $\varepsilon_i$ and not on the index $i$. Hence, the 
cumulative complexity is bounded by a constant multiple of $N^\Lambda(\varepsilon)=
\sum_{k=1}^{2m+1} |N_k((t_{\lfloor k/2\rfloor+1}-t_{\lfloor k/2\rfloor})^{-1/\alpha}
\varepsilon/(2m+1))|$, where we set $\varepsilon=\varepsilon_i$. Let $B'$ 
denote the sum of the number of steps taken by the dominating processes until the 
condition in line~12 of Algorithm~\ref{alg:fd_meander} is satisfied. 
We now prove that $B'$ has exponential moments. %HERE
%As before, using the notation of Lemma~\ref{lem:indicators_time}, by the absolute continuity 
%of the law $\mS^\me(\alpha,\rho)$, we may take $r_1=1$. 

Note that $N^\Lambda(\varepsilon)\leq 
(2m+1)\max_{k=\{1,\ldots,2m+1\}}|N_k(T^{-1/\alpha}\varepsilon/(2m+1))|$. 
Moreover, for any $n'$ independent random variables 
$\vartheta_1,\ldots,\vartheta_{n'}$, we have 
\[\P\bigg(\max_{k\in\{1,\ldots,n'\}}\vartheta_k>x\bigg)=
\P\bigg(\bigcup_{k=1}^{n'}\{\vartheta_k>x\}\bigg)
\leq \sum_{k=1}^{n'}\P(\vartheta_k>x),\quad x\in\R.\]
Proposition~\ref{prop:MainRuntime} implies that the second inequality 
in~\eqref{eq:indicators_decay_asm} is satisfied by $N^\Lambda(\varepsilon)$ with 
$r_2=\max\{\eta,\eta'\}$, $q(n)=e^{-sn}n$ and some $K_2>0$, where 
$s=\min\{d\eta,\delta\gamma,d'\eta',\delta'\gamma'\}$. Thus, by 
Lemma~\ref{lem:indicators_time}, we obtain $\P(B'>n)
\leq K'(e^{-sn}n)^{1/(1+\max\{\eta,\eta'\})}$ for some $K'>0$ and all $n\in\N$.
%In the previous 
%argument, we assumed that $d\eta\neq\delta\gamma$ and $d'\eta'\neq\delta'\gamma'$. 
%If these conditions are not satisfied, we need to modify the function $q(n)$. but the 

The loop from line~2 to line~13 of Algorithm~\ref{alg:fd_meander} executes lines~4 
through~12 a geometric number of times $R$ with success probability  $p=\P(\un{Z}_{t_m}\geq0|\un{Z}_{t_1}\geq0)>0$. Hence, the running time $B''$ of  
Algorithm~\ref{alg:fd_meander} can be expressed as $\sum_{i=1}^R B'_i$, where $B'_i$ 
are iid with the same distribution as $B'$, independent of $R$.
Note that $m:\lambda\mapsto\E[e^{\lambda B'}]$ is finite for any 
$\lambda<s/(1+\max\{\eta,\eta'\})$. Since $m$ is an analytic function and $m(0)=1$, 
then there exists some $x^*>0$ such that $m(x)<1/(1-p)$ for all $x\in(0,x^*)$. Hence, 
the moment generating function of $B''$ satisfies, $\E e^{\lambda B''}=\E m(\lambda)^R$,
which is finite if $\lambda<x^*$, concluding the proof. 

\appendix

\section{\label{sec:notationSupSim}Auxiliary processes and the construction of $\{D_n\}$}

Fix constants $d$ and $\delta$ satisfying $0<\delta<d<\frac{1}{\alpha\rho}$ and 
let $\eta=-\alpha\rho-\mathcal{W}_{-1}\big(-\alpha\rho de^{-\alpha\rho d}\big)/d$, 
where $\mathcal{W}_{-1}$ is the secondary branch of the Lambert W 
function~\cite{MR1414285} ($\eta$ is only required in~\citep[Alg.~2]{ExactSim}). 
Let $I_k^n=\1_{\{ S_k>e^{\delta(n-k-1)}\} }$ for all $n\in\mZ^0$ and $k\in\mZ^n$. 
Fix $\gamma>0$ with $\E[S^\gamma]<\infty$ (see~\citep[App.~A]{ExactSim}), 
where $S\sim\mS^+(\alpha,\rho)$. By Markov's inequality, we have 
\begin{equation}\label{eq:probTails}
p(n)=\P(S\leq e^{\delta n})\geq 1-e^{-\delta\gamma n}\E[S^{\gamma}],\quad 
n\in\N\cup\{0\},
\end{equation}
implying $\sum_{n=0}^{\infty}(1-p\left(n\right))<\infty$. Since the sequence
$(S_k)_{k\in\mZ^0}$ is iid with distribution $\mS^+(\alpha,\rho)$ (as in 
Theorem~\ref{thm:CMSM}), the Borel-Cantelli lemma ensures that, for a fixed 
$n\in\mZ^0$, the events $\{S_k>e^{\delta(n-k-1)}\}=\{I_k^n=1\}$ occur for only 
finitely many $k\in\mZ^n$ a.s. For $n\in\mZ^1$ let $\chi_n$ be the smallest time 
beyond which the indicators $I_k^n$ are all zero:
\begin{equation}\label{eq:chi}
\chi_{n}=\min\left\{n-1,\inf\left\{ k\in\mZ^n:I_{k}^{n}=1\right\}\right\},
\end{equation}
with the convention $\inf\emptyset=-\infty$. Note that $-\infty<\chi_n<n$ a.s. for 
all $n\in\mZ^0$. Since $\mZ^0$ is countable, we have $-\infty<\chi_n<n$ for all 
$n\in\mZ^0$ a.s. Let $m^*=\lfloor\frac{1}{\delta\gamma}\log\E[S^\gamma]\rfloor+1$ 
and note that $e^{-\delta\gamma m}\E[S^\gamma]<1$ for all $m\geq m^*$ . Hence, for 
all $n\in\N\cup\{0\}$, the following inequality holds (cf.~\citep[Sec.~4.1]{ExactSim}) 
\begin{equation}\label{eq:tail_chi_n}
\P(|\chi_0|>n+m^\ast)\leq K_0 e^{-\delta\gamma n},\quad\text{where}\quad
K_0 = \frac{e^{-\delta\gamma m^\ast}\E[S^\gamma]}
{(1-e^{-\delta\gamma})(1-e^{-\delta\gamma m^\ast}\E[S^\gamma])}.
\end{equation}
Indeed, by the inequality~\eqref{eq:probTails}, for every $m\geq m^*$ we have 
\[\begin{split}
\P(|\chi_0|\leq m)
&=\prod_{j=m}^\infty p(j)\geq\prod_{j=m}^\infty (1-e^{-\delta\gamma j}\E[S^\gamma])
=\exp\bigg(\sum_{j=m}^\infty\log(1-e^{-\delta\gamma j}\E[S^\gamma])\bigg)\\
&=\exp\bigg(-\sum_{j=m}^\infty\sum_{k=1}^\infty
\frac{1}{k}e^{-\delta\gamma jk}\E[S^\gamma]^k\bigg)
\geq\exp\bigg(-\sum_{k=1}^\infty\frac{e^{-\delta\gamma mk}\E[S^\gamma]^k}
{1-e^{-\delta\gamma k}}\bigg)\\
&\geq\exp\bigg(-\frac{e^{-\delta\gamma m}\E[S^\gamma]}
{(1-e^{-\delta\gamma})(1-e^{-\delta\gamma m}\E[S^\gamma])}\bigg)\\
&\geq\exp\big(-K_0e^{-\delta\gamma(m-m^*)}\big)
\geq 1-K_0e^{-\delta\gamma(m-m^*)}.
\end{split}\]

Define the iid sequence $(F_n) _{n\in\mZ^0}$ by $F_n=d+\frac{1}{\alpha}
\log(1-U_n)$. Note that $d-F_n$ is exponentially distributed with $\E[d-F_n]=
\frac{1}{\alpha\rho}$. Let $(W_n)_{n\in\mZ^1}$ be a random walk defined by 
$W_n=\sum_{k\in\mZ_n^0}F_k$. Let $(R_n)_{n\in\mZ^1}$ be reflected process of 
$(W_n)_{n\in\mZ^1}$ from the infinite past
\[
R_n=\max_{k\in\mZ^{n+1}}W_k-W_n,\quad n\in\mZ^1.
\]
For any $n\in\mZ^1$ define the following random variables
\begin{equation}\label{eq:DominatingProcess}\begin{split}
D_n 
& = \exp(R_n)\left(\frac{e^{(\delta-d)(n-\chi_{n})}}{1-e^{\delta-d}}
+ \sum_{k\in\mZ_{\chi_n}^n}e^{-(n-k-1)d}U_k^{1/\alpha}S_k\right),\\
D_n'
& = \exp(R_n)\left(\frac{1}{1-e^{\delta-d}}+D_n^{\prime\prime}\right),
\quad\text{where}\quad 
D_n^{\prime\prime}=\sum_{k\in\mZ^n}e^{-(n-k-1)d}S_k.
\end{split}\end{equation}
Note that the series in $D_{n}^{\prime\prime}$ is absolutely convergent by the 
Borel-Cantelli lemma, but $D_{n}'$ cannot be simulated directly as it depends on 
an infinite sum. In fact, as was proven in~\citep[Sec.~4]{ExactSim}, it is 
possible to simulate $((\Theta_n,D_{n+1}))_{n\in\mZ^0}$ backward in time.

Let $\mathcal{A}=(0,\infty)\times(0,1)$ put $\Theta_n=(S_n,U_n)$. Define the 
update function $\phi:(0,\infty)\times\mathcal{A}\to(0,\infty)$ given by 
$\phi(x,\theta)=(1-u)^{1/\alpha}x+u^{1/\alpha}s$ where $\theta=(s,u)$. 
By \citep[Lem.~2]{ExactSim}, $M_n\leq D_n\leq D_n'$ for $n\in\mZ^1$ 
and that $((\Theta_n,R_n,D_{n+1}'))_{n\in\mZ^0}$ is Markov, stationary, and 
$\varphi$-irreducible (see definition~\citep[p.~82]{MR2509253}) with respect 
to its invariant distribution.

Hence, we may iterate~\eqref{eq:SMPerEq} to obtain for $m\in\mZ^1$ and 
$n\in\mZ^m$,
\begin{equation}\label{eq:nth_step}\begin{split}
M_m 
&=\bigg(\prod_{k\in\mZ_n^m}(1-U_k)^{1/\alpha}\bigg)M_n+\sum_{k\in\mZ_n^m}
\bigg(\prod_{j\in\mZ_k^m}(1-U_j)^{1/\alpha}\bigg)U_k^{1/\alpha}S_k\\
&=\underbrace{\phi(\cdots\phi}_{m-n}(M_n,\Theta_n),\ldots,\Theta_{m-1}).
\end{split}\end{equation}
%Finally, for any $m\in\mZ^1$ and $n\in\mZ^m$ define the random function
%\begin{equation}\label{eq:alt_n_step}
%\Phi_n^m:x\mapsto\underbrace{\phi(\cdots\phi}_{m-n}(x,\Theta_n),\ldots,
%\Theta_{m+1}),\quad x\geq 0.
%\end{equation}
%Since $M_m=\Phi_n^m(M_n)$, the main idea is to construct approximations of $M_m$ 
%via the variables $(\Phi_n^m(D_n))_{n\in\mZ^m}$. By construction we have 
%$|M_m-\Phi_n^m(x)|=\big(\prod_{k\in\mZ_n^m}(1-U_k)^{1/\alpha}\big)|M_n-x|$, where 
%the product vanishes as $n\to\infty$. In particular, by the domination 
%$D_n\geq M_n$, we have
%\begin{equation}\label{eq:error-nm}
%\max\{|M_m-\Phi_n^m(D_n)|,|M_m-\Phi_n^m(0)|\}
%\leq\bigg(\prod_{k\in\mZ_n^m}(1-U_k)^{1/\alpha}\bigg)D_n.
%\end{equation}

%\section{Proof of Corollary~\ref{cor:moments}\label{sec:proofs_cor}}

\section{On regularity}
\label{app:On_regularity}
\begin{lem}
	\label{lem:Kallenberg_regularity}
Assume that $X$ is a L\'{e}vy process generated by $(b,\sigma^{2},\nu)$
and define the function 
$\ov{\sigma}^2(u)=\sigma^{2}+\int_{-u}^{u}x^{2}\nu(dx)$. 
If $\lim_{u\searrow0}u^{-2}|\log u|^{-1}\ov{\sigma}^2(u)=\infty$, 
then~(\nameref{asm:(K)})
holds. 
\end{lem}
\begin{proof}
(the proof is due to Kallenberg \citep{MR628873}) Let $\psi(u)=\log\E[e^{iuX_1}]$
and note that for large enough $|u|$ and fixed $t>0$ we have
\[
-\log\big(\big|e^{t\psi(u)}\big|\big)
=\frac{1}{2}tu^{2}\sigma^{2}+t\int_{-\infty}^{\infty}(1-\cos(ux))\nu(dx)
\geq\frac{1}{3}tu^{2}\ov{\sigma}^2\big(|u|^{-1}\big)\geq2|\log|u||.
\]
Hence, $|e^{t\psi(u)}|=|\mathbb{E}[e^{iuX_{t}}]|=\mathcal{O}(u^{-2})$
as $|u|\to\infty$, which yields (\nameref{asm:(K)}).
\end{proof}

\bibliographystyle{amsalpha}
\bibliography{References}

\section*{Acknowledgements}
JGC and AM are supported by The Alan Turing Institute under the EPSRC grant EP/N510129/1;
AM supported by EPSRC grant EP/P003818/1 and the Turing Fellowship funded by the Programme on Data-Centric Engineering of Lloyd's Register Foundation;
GUB supported by CoNaCyT grant FC-2016-1946 and UNAM-DGAPA-PAPIIT grant IN115217; 
JGC supported by CoNaCyT scholarship 2018-000009-01EXTF-00624 CVU 699336. 

\end{document}